\documentclass[11pt,reqno]{amsart}

\numberwithin{equation}{section}

\usepackage[all]{xy}
\usepackage{amssymb}
\usepackage{amsmath}
\usepackage{amsbsy}
\usepackage{amscd}
\usepackage{amsthm}
\usepackage{amsfonts}
\usepackage{enumerate}
\usepackage{color}
\usepackage{framed}
\usepackage{enumitem}
\definecolor{darkspringgreen}{rgb}{0.05, 0.5, 0.06} %{0.09, 0.45, 0.27}

\newcommand\new[1]{#1} %{\color{darkspringgreen}{#1}}}
\usepackage[normalem]{ulem}
\usepackage{soul}

\usepackage[colorlinks,bookmarks,linkcolor=blue,citecolor=blue]{hyperref} 

\newtheorem{theorem}{Theorem}[section]
\newtheorem{hypothesis}[theorem]{Hypothesis}
\newtheorem{claim}[theorem]{Claim}
\newtheorem{lemma}[theorem]{Lemma}
\newtheorem{corollary}[theorem]{Corollary}
\newtheorem{definition}[theorem]{Definition}

\newtheorem{conjecture}[theorem]{Conjecture}

\newtheorem{proposition}[theorem]{Proposition}

\theoremstyle{definition}

\newtheorem{remark}[theorem]{Remark}
\theoremstyle{definition}

\newcommand{\cL}{\mathcal{L}}

\newcommand{\bC}{\mathbb{C}}

\newcommand{\bR}{\mathbb{R}}
\newcommand{\bZ}{\mathbb{Z}}
\newcommand{\bQ}{\mathbb{Q}}

\newcommand{\bN}{\mathbb{N}}

\newcommand{\tphi}{\tilde{\phi}}
\newcommand{\tlambda}{\tilde{\lambda}}
\newcommand{\tchi}{\tilde{\chi}}
\newcommand{\tLambda}{\tilde{\Lambda}}

\newcommand{\liea}{\mathfrak{a}}

\newcommand{\lieg}{\mathfrak{g}}
\newcommand{\lieh}{\mathfrak{h}}
\newcommand{\liek}{\mathfrak{k}}

\newcommand{\liem}{\mathfrak{m}}
\newcommand{\lien}{\mathfrak{n}}
\newcommand{\liep}{\mathfrak{p}}

\newcommand{\liez}{\mathfrak{z}}

\newcommand{\bfG}{\mathbf{G}}
\newcommand{\bfH}{\mathbf{H}}
\newcommand{\bfJ}{\mathbf{J}}
\newcommand{\bfT}{\mathbf{T}}
\newcommand{\bfA}{\mathbf{A}}

\newcommand{\rmm}{\mathrm{m}}

\newcommand{\SL}{\operatorname{SL}}

\newcommand{\GL}{\operatorname{GL}}
\newcommand{\PSL}{\operatorname{PSL}}
\newcommand{\Ad}{\operatorname{Ad}}

\newcommand{\rank}{\operatorname{rank}}

\newcommand{\Aut}{\operatorname{Aut}}
\newcommand{\supp}{\operatorname{supp}}
\newcommand{\Orb}{\operatorname{Orb}}

\newcommand{\Stab}{\operatorname{Stab}}

\newcommand{\SO}{\operatorname{SO}}
\newcommand{\Gal}{\operatorname{Gal}}
\newcommand{\Diff}{\operatorname{Diff}}

\newcommand{\onto}{\xymatrix{\ar@{>>}[r]&}}
\newcommand{\da}[4]{\xymatrix{#1 \ar@<.5ex>[r]^{#2} \ar@<-.5ex>[r]_{#3} & #4}}
\newcounter{subconst}[subsection]

\newcounter{const}

\newcounter{CONST}

\renewcommand{\emph}[1]{{\bf #1}}

\newcommand{\td}{\tilde}

\newcommand{\sm}{\smallsetminus}
\newcommand{\R}{\mathbb {R}}
\newcommand{\Q}{\mathbb {Q}}
\newcommand{\Z}{\mathbb {Z}}
\newcommand{\T}{\mathbb {T}}

\newcommand{\inv}{^{-1}}

\newcommand{\restrict}[2]{{#1}{\restriction_{{ #2}}}}

\usepackage{marginnote}

\newcommand{\diff}{\Diff}

\def\Homeo{\mathrm{Homeo}}
\def\homeo{\Homeo}

\newcommand{\hide}[1]{}

\newcommand{\namedthm}[3]{\theoremstyle{plain}
   \newtheorem*{thm#1}{#2}\begin{thm#1}#3\end{thm#1}}

\begin{document}

\title[Global rigidity of hyperbolic lattice actions]{Global smooth and topological rigidity of hyperbolic lattice actions}
\author[A.~Brown]{Aaron Brown}
\address{University of Chicago, Chicago, IL 60637, USA}
\email{awb@uchicago.edu}

\author[F.~Rodriguez Hertz]{Federico Rodriguez Hertz}
\address{Pennsylvania State University, State College, PA 16802, USA}
\email{hertz@math.psu.edu}

\author[Z.~Wang]{Zhiren Wang}
\address{Pennsylvania State University, State College, PA 16802, USA}
\email{zhirenw@psu.edu}
%\date{\today}

\setcounter{page}{1}
\begin{abstract}
In this article we prove global rigidity results for hyperbolic actions of higher-rank lattices.

Suppose $\Gamma$ is a lattice in semisimple Lie group, all of whose factors have rank $2$ or higher. Let  $\alpha$ be a smooth $\Gamma$-action on a compact nilmanifold $M$ that lifts to an action on the universal cover. If the linear data $\rho$ of $\alpha$ contains a hyperbolic element, then there is a continuous semiconjugacy intertwining the actions of  $\alpha$ and $\rho$, on a finite-index subgroup of $\Gamma$. If $\alpha$ is a $C^\infty$ action and contains an Anosov element, then the semiconjugacy is a $C^\infty$ conjugacy.

As a corollary, we obtain $C^\infty$ global rigidity for Anosov actions by cocompact lattices in semisimple Lie group with all factors rank $2$ or higher.  We also obtain global rigidity of Anosov actions of $\SL(n,\bZ)$ on $\T^n$ for $ n\geq 5$ and probability-preserving Anosov actions of arbitrary higher-rank lattices on nilmanifolds.   %on a torus that contains an Anosov diffeomorphism is $C^\infty$-conjugate to an action by toral automorphisms, after restricting to a finite-index subgroup.
\end{abstract}
\maketitle
\small\tableofcontents

\newcommand{\foot}[1]{\mbox{}\marginpar{\raggedleft\hspace{0pt}
\Tiny #1}}

%\red{[Working version for technical theorem. To be polished later]

%\section*{changes}
%\begin{enumerate}
%\item fixed some reference notation
%\item changed $(P_i)_{*}$ to $P_{*,i}$ throughout.  Indeed, as I consider deck transformations rather than fundamental groups, the induced map $P_{*,i}$ is associated to a choice of lift $\td P_i\colon \td M \to N_i$ rather than the base map $M\to N_i/\Lambda_i$.  
%\end{enumerate}
\section{Introduction and statement of main results}

\subsection{Background and motivation}
Let $G$ be a connected, semisimple Lie group with finite center, no compact factors, and all almost-simple factors of real-rank at least 2.  Let $\Gamma\subset G$ be a lattice; that is $\Gamma$ is a discrete subgroup of $G$ such that $G/\Gamma$ has finite Haar volume.
The celebrated superrigidity theorem of Margulis states that, for $G$ and $\Gamma$ as above, any linear representation $\psi\colon \Gamma \to \PSL(d, \R)$ is of algebraic nature; that is $\psi$   extends to a continuous representation $\psi'\colon G\to \PSL(d, \R)$ up to a compact error.  See Theorem \ref{Superrigidity} and Proposition \ref{prop:ourSR} below for more formal statements.

Shortly after, based on the analogy between linear groups and  diffeomorphism groups $\Diff^\infty(M)$ of   compact manifolds, Zimmer proposed a number of  conjectures for representations of $\Gamma$ into  $\Diff^\infty(M)$.   These and related conjectures are referred to as the  {\it Zimmer program}, which aims to understand and classify smooth actions by higher-rank lattices.
We refer the reader to the excellent survey \cite{F11} by Fisher for a detailed account of the Zimmer program.

A major direction of research in the Zimmer program is the classification of actions containing some degree of  hyperbolicity  (see \cite{MR1271833} and \cite[Section 7]{F11} for further discussion). For instance, the following conjecture is motivated by works of Feres--Labourie \cite{MR1643954} and Goetze--Spatzier \cite{MR1740993}.

\begin{conjecture}[{\cite[Conjecture 1.3]{F11}}] \label{conj:Fish}If $\Gamma$ is a lattice in $SL(n,\R)$ where $n\geq 3$, then all $C^\infty$-actions by $\Gamma$ on a compact manifold that both preserves a volume form and contains an Anosov diffeomorphism are algebraically defined. \end{conjecture}
Here, being algebraically defined means the action is  smoothly conjugate to an action on an infranilmanifold by affine automorphisms.
See also \cite[Conjecture 1.1]{MR1271833} and \cite[Conjecture 1.1]{MR1380646} for related conjectures.
We recall that it is conjectured that infranilmanifolds are the only manifolds supporting Anosov diffeomorphisms.

The assumption in Conjecture \ref{conj:Fish} that the action preserves a volume is a standard assumption in results on the rigidity of group actions.
The majority of advances in the Zimmer program, including most predecessors of the results discussed in this paper (\cite{MR1367273, MR1740993, MR1826664}), assume the  action $ \Gamma \to  \diff^\infty(M)$ preserves a Borel probability measure on $M$.  In such settings, Zimmer's superrigidity theorem for cocycles (generalizing Margulis's  superrigidity theorem for linear representations; see Theorem \ref{CocycleSuperrigidity} below) gives that the derivative cocycle is measurably cohomologous to a linear representation of $G$ up to a compact correction.  This provides evidence for the conjectures behind the Zimmer program and is the starting point for many of the local and global rigidity results preceding this paper.

For the remainder, we consider representations $\alpha \colon\Gamma \to \diff^\infty(M) $ where $M$ is either a torus $\T^d$ or a compact nilmanifold $N/\Lambda$.
 If $M$ is a torus (or if $M$ is a nilmanifold and the action $\alpha$ lifts to an action $\td \alpha \colon \Gamma \to \diff^\infty(N)$) one can define a linear representation $\rho\colon \Gamma \to \GL(d,\Z)$ (or $\rho\colon \Gamma \to \Aut(\Lambda)$) associated to the action $\alpha$ called the \emph{linear data of $\alpha$}. We then obtain an action $\rho\colon \Gamma \to \Aut(\T^d)$ (or $\rho\colon \Gamma \to \Aut(N/\Lambda)$.) See Section \ref{sec:lineardata} for more details.
We assume throughout that $\rho$ is hyperbolic, that is $D\rho(\gamma)$ is a hyperbolic linear transformation for some $\gamma\in \Gamma$.

As a primary example of such an action consider any homomorphism $\rho\colon \Gamma \to \Aut(\Lambda)$ where $\Lambda $ is a (cocompact) lattice in a nilpotent, simply connected Lie group $N$.  Then $\rho$ induces and action by automorphisms $\rho\colon\Gamma \to \diff^\infty(N/\Lambda)$ and hence coincides with its linear data.  Similarly, one can build model algebraic actions of $\Gamma$ by affine transformations of $N/\Lambda$ (see \cite{MR1236179} for constructions and discussion.)
Early rigidity results in this setting focused on various notions of {rigidity} for non-linear perturbations of affine actions.  For instance, in \cite{MR1154597} Hurder proved a number of \emph{deformation rigidity} results for certain standard affine actions; that is, under certain hypotheses, a 1-parameter family of perturbations of an affine action $\rho$ a  smoothly conjugate to $\rho$.  A related rigidity phenomenon, the \emph{infinitesimal rigidity}, has been studied for affine actions in  \cite{MR1154597, MR1338481, MR1058434, MR1415754}.

The primary rigidity phenomenon studied  for perturbations is \emph{local rigidity}; that is, given an affine action $\rho\colon \Gamma \to \diff^\infty(N/\Lambda)$ and $\alpha \colon\Gamma\to \diff^\infty(N/\Lambda)$ with $\alpha(\gamma_i)$ sufficiently $C^1$-close to $\rho(\gamma_i)$  for a finite generating set $\{\gamma_i\}\subset \Gamma$, one wishes to find a $C^\infty$ change of coordinates $h\colon N/\Lambda\to N/\Lambda$ with
$
h\circ \alpha(\gamma) = \rho(\gamma)\circ h
$
for all $\gamma\in \Gamma$.
For isometric actions,  local rigidity has been shown to hold for cocompact lattices considered above \cite{MR1779610} and for property (T) groups \cite{MR2198325}.
%See also \cite{MR2521112} for local rigidity results in other settings.
For hyperbolic affine actions on tori and nilmanifolds, local rigidity has been established for a  number of specific   actions or under additional dynamical hypotheses in \cite{MR1154597, MR1631740, MR1164591, MR1367273,MR1740993,MR1332408}.

For the general case of actions by higher-rank lattices on nilmanifolds, the local rigidity problem for affine Anosov actions was settled by Katok and Spatzier in \cite{MR1632177}. In \cite{MR1826664} Margulis and Qian extended local rigidity to weakly hyperbolic affine actions.  Fisher and Margulis \cite{MR2521112} established local rigidity in full generality for quasi-affine actions by higher-rank lattices  which, in particular, includes actions by nilmanifold automorphisms without assuming any hyperbolicity.  We note that the local rigidity results discussed above require  property (T); in particular, they do not hold for irreducible lattices in products of rank-1 Lie groups.

We turn our attention for the remainder to  the question of \emph{global rigidity} of  actions on tori and nilmanifolds.  That is, given an action $\alpha \colon \Gamma \to \Diff^\infty(N/\Lambda)$ with linear data $\rho\colon \Gamma \to \Aut(N/\Lambda)$ we ask:
\begin{enumerate}
	\item {\it topological rigidity:} is there a continuous $h\colon N/\Lambda\to N/\Lambda$ with $h\circ \alpha(\gamma) = \rho(\gamma)\circ h$ for all $\gamma$ in a finite index subgroup?
	\item {\it smooth rigidity:} if so, is $h$ a $C^\infty$ diffeomorphism?
\end{enumerate}

Note that for a general finitely generated discrete group $\Gamma$ and an action $\alpha \colon \Gamma \to \Diff^\infty(N/\Lambda)$ there is no expectation that such an $h$ would exist.  Indeed when $\Gamma$ is a finitely generated free group, examples of actions $\alpha$ (including actions containing Anosov elements) exist for which no $h$ as above exists.
On the other hand, for $C^\infty$ actions  on nilmanifolds of higher-rank lattices $\Gamma$ as introduced above,
one may expect such a continuous $h$ to exist.  However, examples constructed in \cite{MR1380646} by blowing up fixed points show that (even when $\alpha$ is real analytic, volume preserving, and ergodic) such $h$ need not be invertible.  However, if the non-linear action possess an {Anosov} element $\alpha(\gamma_0)$ then any $h$ as above is necessarily invertible. In this setting, one may expect such $h$ to be $C^\infty$.

Global rigidity results under strong dynamical hypotheses appear already in \cite{MR1154597}.
Global rigidity for Anosov actions by $SL(n,\Z)$ on $\T^n$, $n\geq 3$, were obtained in \cite{MR1380646,MR1367273}. Other global rigidity results appear in \cite{MR1401783}.
We also remark that Feres-Labourie \cite{MR1643954} and Goetze-Spatzier \cite{MR1740993} established very strong global rigidity properties for Anosov actions, in which no assumptions on the topology of  $M$ are made.
  In both works, under strong dynamical hypotheses  including that the dimension of $M$ is  small relative  to $G$,
  it is shown that  $M$ is necessarily an infranimanifold and the action is algebraically defined.  

Global topological rigidity results for Anosov actions on general nilmanifolds were proven in \cite[Theorem 1.3]{MR1826664}.  Here, a $C^0$-conjugacy is obtained assuming the existence of a fully supported invariant measure for the non-linear action. Topological conjugacies between  actions on more general manifolds $M$ whose action on $\pi_1(M)$ factors through an action of a finitely-generated, torsion-free, nilpotent group are studied in \cite{MR1866848}.  (See Section \ref{sec:factors} and Theorem \ref{main:factorsfull} for related results in this direction.)

In this paper we study the global rigidity problem for actions of higher-rank lattices on nilmanifolds with hyperbolic linear data. See Theorems \ref{ConjThm} and \ref{SmoothThm} below.   We provide complete solutions to the global rigidity questions above   under the mild assumption that the action lifts to an action on the universal cover (see Remark \ref{LiftingRmk} and Section \ref{sec:Lifting} for discussion on when the lifting is guaranteed to hold.)  In particular, for such actions $\alpha$ with hyperbolic linear data, we construct a continuous semiconjugacy to the linear data when restricted to a finite index subgroup.  Moreover if the action contains an Anosov element we show that the semiconjugacy (which is necessarily a homeomorphism in this case) is, in fact, a $C^\infty $ diffeomorphism.

We remark  that the majority of  global rigidity results discussed above assume the existence of an (often smooth or fully supported) invariant measure for the action.  We emphasize that   we do not assume the existence of an invariant measure in Theorems \ref{ConjThm} and \ref{SmoothThm}.

%We remark that a separate but related program exists studying the local and global rigidity for Anosov actions of higher-rank Abelian groups.  For actions on nilmanifolds and tori, major global rigidity results were recently obtained in \cite{MR2983009} and \cite{MR3260859}.  For hyperbolic and partially hyperbolic actions higher-rank Abelian groups arising from restrictions of  diagonal actions on homogeneous spaces to  has recently  been shown to hold in large generality \cite{1510.00848} generalizing many previous results \cite{MR1307298,MR1632177,MR2342703, MR2838045, MR1266758, MR2753946, MR2672298}

%\subsection{Motivation}
\subsection{Topological rigidity for maps}   \label{sec:topo}
Consider a homeomorphism $f\colon \T^n\to \T^n$. Recall that there exists a unique $A\in \GL(n, \Z)$ such that any lift $\td f\colon \R^n \to \R^n$ is of the form \begin{align}\label{eq:1}\td f(x) = Ax + u(x)\end{align} where $u\colon \R^n\to \R^n$ is $\Z^n$-periodic.  We call $A$ the linear data of $f$.
 As $A$ preserves the lattice $\Z^n$ in $\R^n$ we have an induced map $L_A\colon \T^n\to \T^n$.  It follows that $f$ is homotopic to $L_A$.  A similar construction holds for diffeomorphisms of nilmanifolds.

%We let $L_A$ denote the natural induced action of $A$ on the torus.
 %Let $f_*$ denote the induced action on $\pi_1(\T^d)$.  Fixing the standard identification $\pi_1(\T^d)= \Z^d$ with the group of deck transformations of the cover $\R^d\to \R^d$ we have that $f_*$ is represented by a matrix $A\in \GL(d,\Z)$.

The starting point for the global rigidity problem we study is the following classical theorem of Franks.

%\cite{MR0271990}, Franks proved the following.

\namedthm{1}{Franks' Theorem \cite{MR0271990}}{Assume $A$ has no eigenvalues of modulus 1.  Then there is a  continuous $h\colon \T^d\to \T^d$, homotopic to the identity,  such that \begin{equation}\label{eq12} L_A\circ h = h \circ f.\end{equation}
}

Moreover, fixing a lift  $\td f $ of $f$, the map $h\colon \T^n\to \T^n$ is unique among the continuous maps having a lift $\td h\colon \R^n \to \R^n $ satisfying $\td h \circ \td f = A\circ \td h$. %for some lift $\td f $ of $f$.

 A map $h$ satisfying \eqref{eq12} is called a \emph{semiconjugacy} between $f$ and $L_A$.

Recall that a diffeomorphism $f$ of a manifold $M$ is \emph{Anosov} if $TM$ admits a continuous decomposition $E^u\oplus E^s$ that is preserved by  $Df$ such that $E^u$ and $E^s$ are, respectively, uniformly expanded and contracted by $Df$. The fundamental examples of Anosov diffeomorphisms are affine automorphisms of nilmanifolds and tori;  that is,  diffeomorphisms of the form $b\cdot A(x)$ on a nilmanifold $M=N/\Lambda$ where $b\in N$ and $A\in \Aut(M)$ such that $\restrict{DA}{T_eN}$ is a hyperbolic linear transformation of the Lie algebra $\lien$ of $N$.  It is conjectured that the only manifolds admitting  Anosov diffeomorphisms are finite quotients of tori or nilmanifolds.

%If $f\colon M\to M$ is an Anosov diffeomorphism of a nilmanifold $M = N/\Lambda$ the it  is well known that
In the case that $f$ is an Anosov diffeomorphism of a torus or nilmanifold, it is well known that the linear data of $f$ is hyperbolic.  Moreover, it follows from the work of Franks \cite{MR0271990} and Manning \cite{MR0358865}  that
the map  $h$ satisfying \eqref{eq:1} is a  homeomorphism; in this case we call such an $h$ a \emph{conjugacy} between $f$ and $L_A$.  Moreover, one can show in this case that $h$ is bi-H\"older.  However, in general one can not obtain any additional regularity of $h$ even when $f$ is Anosov.

\subsection{Setting for main  results}
%The global rigidity problem we study in this paper is the generalization of  Franks' Theorem to  actions on tori and nilmanifolds  by lattices $\Gamma\subset G$ in higher-rank semisimple Lie groups $G$.
Let $\Gamma $ be a discrete group and $\alpha$ an action of $\Gamma$ by homeomorphisms on a compact nilmanifold $N/\Lambda$.  In Section \ref{sec:lineardata} we define the \emph{linear data} $\rho\colon \Gamma \to \Aut(N/\Lambda)$ for such actions
%Suppose the linear data $\rho\colon \Gamma\to \Aut(N)$ is well defined,
under the assumption either that $N$ is abelian or that the action $\alpha $ lifts to an action by homeomorphisms of $N$.  % lifting hypotheses or because $N$ is abelian.
We assume for the time being that the linear data $\rho$ associated to $\alpha$ is defined.
For   individual elements $\alpha(\gamma)$ of the action such that $\rho(\gamma)$ is hyperbolic, one can build a semiconjugacy between the elements $\alpha(\gamma)$ and $\rho(\gamma)$.  %to the corresponding element of the linear action.
However, even assuming the action lifts,   one rarely expects to be able to build a single map $h\colon N/\Lambda\to N/\Lambda$ such that
 $$\rho(\gamma) \circ h = h \circ \alpha (\gamma) $$  holds for {\it every element} of $\Gamma$, or even for every element in a finite-index subgroup of $\Gamma$.

 We focus in this paper on  discrete groups $\Gamma$ exhibiting certain rigidity properties---namely lattices in higher-rank, semisimple Lie groups.  For such $\Gamma$, we can exploit certain properties of $\Gamma$ to
study the rigidity of actions of $\Gamma$.
%  construct a semiconjugacy between the linear and non-linear {actions}.  Additionally, if the non-linear action has an Anosov element, we show the conjugacy is smooth.
%Let $G$ be a semisimple Lie group and let $\Gamma\subset G$ be a lattice.
 For most  results, we will assume the following hypothesis.
\begin{hypothesis}\label{HigherRank}$G$ is a connected semisimple Lie group with finite center, all of whose non-compact almost-simple factors  have $\R$-rank $2$ or higher. $\Gamma$ is a lattice in $G$. \end{hypothesis}

\subsection{Topological rigidity for actions with hyperbolic linear data}\label{sec:CorTopo}

%The global rigidity problem we study in this paper is the generalization of  Franks' Theorem to  actions on tori and nilmanifolds  by lattices $\Gamma\subset G$ in higher-rank semisimple Lie groups $G$.

Our first main theorem is a solution to the topological global rigidity problem assuming the action lifts and the linear data $\rho(\gamma_0)$ is hyperbolic for some $\gamma_0\in \Gamma$.

\begin{theorem}\label{ConjThm} Let $G$ and $\Gamma$ be as in Hypothesis \ref{HigherRank}. Let $\alpha$ be a $C^0$ action of
 $\Gamma$ on a compact nilmanifold $M=N/\Lambda$. Suppose $\alpha$ can be lifted to an action on the universal cover $N$ of $M$ and let $\rho$ be the associated linear data of $\alpha$.
If $D\rho(\gamma)$ is hyperbolic for some element $\gamma\in\Gamma$ then there is a finite-index subgroup $\Gamma_1<\Gamma$ and a surjective continuous map $h\colon M\to M$, homotopic to identity,   such that $\rho(\gamma)\circ h=h\circ\alpha(\gamma)$ for all $\gamma\in\Gamma_1$. If $\alpha$ acts by Lipschitz homeomorphisms, then$h$ is H\"older continuous.
\end{theorem}

\begin{remark} As  genuinely affine actions exist,  the restriction to a finite-index subgroup $\Gamma_1$ is necessary. This is demonstrated by an example of Hurder \cite[Theorem 2]{MR1236179}.
\end{remark}

In the case that $N= \R^d$ and $M= \T^d$, the obstruction to lifting the action $\alpha$ of $\Gamma$ on $\T^d$ to an action   of $\Gamma$ on $\R^d$ is represented by an element in the group cohomology   $H^2_{\rho}(\Gamma, \Z^d)$. If this element vanishes in $H^2_\rho(\Gamma, \bR^d)$ then it vanishes in $H^2_{\rho}(\Gamma, \Z^d)$ after possibly passing to a finite-index subgroup of $\Gamma$ and
the lifting of the action is automatic.
For actions on nilmanifolds, the action lifts assuming the vanishing of certain obstructions in the group cohomology associated to a finite number of induced representations.
Sufficient conditions for the vanishing of the cohomological obstructions are given by   \cite[Theorem 3.1]{MR0228504}  and \cite[Theorem 4.4]{MR642850}.
  In particular, the lifting hypotheses can be verified using only knowledge about the linear data $\rho$ (in the case $M = \T^d$) or the induced $\alpha_\#\colon \Gamma \to \mathrm{Out}(\Lambda)$ (in the case $M = N/\Lambda$).
See Section \ref{sec:Lifting} for more details.

In particular,
%
% This argument was given in \cite{MR1866848}*{Section2 and paragraph following Theorem 3.2}.

\begin{remark}\label{LiftingRmk}\rm
%The assumption that the restriction of $\alpha$ to a finite-index subgroup lifts to an action on $N$ is often automatically satisfied. For example, it is true if one of the following condition holds.
The  restriction of $\alpha$ to a finite-index subgroup lifts to an action on $N$ assuming any of the  following condition holds.
\begin{enumerate}[ref = (\arabic*)]
%\item  \label{liftrem:1} $\Gamma=\SL(n,\Z)$, $n\geq 9$.  \marginnote{do we believe this}
\item\label{liftrem:2} $\Gamma=\SL(d,\Z)$ acting on $\T^d$, $d\ge 5$.
\item \label{liftrem:3}$\Gamma$ is a cocompact lattice in $G$.
\item \label{liftrem:4}  $\alpha$ is an action of  $\Gamma$    on a torus $\T^d$ which preserves a probability measure $\mu$.
\item \label{liftrem:5} $\Gamma$ is as in Hypotheses \ref{HigherRank}, $\alpha$ is an action of  $\Gamma$    on a compact nilmanifold $N/\Lambda$ which preserves a probability measure $\mu$, and $\alpha(\gamma_0)$ is Anosov for some $\gamma_0$. \end{enumerate}
%Additionally, \cite[Theorem 3.1]{MR0228504} together with the discussion in Section \ref{sec:Lifting} gives sufficient conditions for the action to lift on a finite index subgroup. The condition in \cite[Theorem 3.1]{MR0228504} can be verified from the linear data $\rho$.

\end{remark}

%When $M=\bT^d$, t

The main advantage of our method of proof is that, unlike the majority of previous results discussed above, we do not assume the existence of an invariant measure for the action $\alpha$ in order to construct a semiconjugacy.  Note that given a conjugacy between the linear and non-linear actions, one can obtain a $\Gamma$-invariant measure for the non-linear action.  However, in the case of a semiconjugacy, the existence of an invariant measure for the non-linear actions is more subtle.  We present as a corollary of Theorem  \ref{ConjThm} certain  conditions under which a non-linear non-Anosov action of $\Gamma$ has a ``large''  invariant measure.

\begin{theorem}\label{MeasureThmCase}Let $\Gamma \subset \SL(n+1,\Z)$, $n\ge 2$ be of finite index and let $\alpha $ be an action of $\Gamma$ on $\T^{n+1}$ by $C^{1+\beta}$ diffeomorphisms.  Suppose the action $\alpha$ lifts to an action on $\R^{n+1}$, and let $h$ denote the semiconjugating map guaranteed by Theorem \ref{thm:mainnilmanifoldfull}.
Moreover, suppose that the linear data $\rho\colon \Gamma \to \GL(n+1,\Z)$ is the identity representation $\rho(\gamma) = \gamma$.

Then there exists a unique, $\alpha$-invariant, absolutely continuous probability measure $\mu$ on $\T^{n+1}$ such that $h_*\mu$ is the Haar measure on $\T^{n+1}$.
Moreover, $\mu$ is the unique ergodic    $\alpha$-invariant measure on $\T^n$ with $h_*\mu$ is not atomic. % has positive Hausdorff dimension in $\T^{n+1}$.  %$h_\mu(\alpha(\gamma))>0$ has positive entropy and such that $h_{h_*\mu}(\rho(\gamma))>0$ for some $\gamma\in \Gamma$.

\end{theorem}

\begin{proof}
We may find a copy of $\Z^n$ inside $\Gamma$ so that, in the terminology of \cite{MR2261075}, the corresponding linear action $\restrict{\rho}{\Z^n}$ is a  linear Cartan action.  It follows from the results of \cite{MR1406432} that  any ergodic, $\restrict{\rho}{\Z^n}$-invariant measure with positive Hausdorff dimension is Haar measure.
As $\Z^d$ is amenable, there exists an ergodic, $\restrict{\alpha}{\Z^n}$-invariant, measure on $\T^{n+1}$ projecting to the Haar measure under $h$.
From the main Theorem \cite{MR2261075}, it follows that any such $\mu$ is absolutely continuous.

Moreover, from the main result of \cite{MR2285730} it follows that there is a unique measure $\mu$ on $\T^{n+1}$ such that $h_*\mu$ is the Haar measure.  It follows from this uniqueness criterion that $\mu$ is invariant under the entire action $\alpha$. Finally, if $\nu$ is $\alpha$-invariant and if $h_*\nu$ is not atomic  then, as  $h_*\nu$ is  $\rho$-invariant, a Fourier analysis argument shows it must be   Haar measure.
\end{proof}

\subsection{Smooth rigidity for Anosov actions}\label{sec:CorSmooth}
The main result of the paper is the following solution to the global smooth rigidity problem.  We show that, in the setting of Theorem \ref{ConjThm}, if $\alpha$ is an action by $C^\infty$ diffeomorphisms and if  $\alpha(\gamma)$ is an Anosov  for some element $\gamma$ of $\Gamma$, then semiconjugacy $h$ (which is necessarily invertible) is  a diffeomorphism.  %smooth.

\begin{theorem}\label{SmoothThm} Let $G$ and $\Gamma$ be as in Hypothesis \ref{HigherRank}. Let $\alpha$ be a $C^\infty$ action of
 $\Gamma$ on a compact nilmanifold $M=N/\Lambda$. Suppose $\alpha$ can be lifted to an action on the universal cover $N$ of $M$ and let  $\rho$ be the associated linear data of $\alpha$.
If $\alpha(\gamma)$ is Anosov for some element $\gamma\in\Gamma$, then there is a finite-index subgroup $\Gamma'<\Gamma$ and a $C^\infty$ diffeomorphism $h\colon M\to M$, homotopic to identity,  such that $h \circ\alpha(\gamma)=\rho(\gamma)\circ h $ for all $\gamma\in\Gamma'$.\end{theorem}

In light of Remark \ref{LiftingRmk}, from Theorem \ref{SmoothThm} it immediately follows that:

\begin{corollary}\label{AnosovLifting}Suppose any one of the following:
\begin{enumerate}
\item $n\geq 5$, and $\alpha$ is a $C^\infty$ action by $\Gamma=\SL(n,\Z)$ on $M=\T^n$;
%\item $n\geq 9$, and $\alpha$ is a $C^\infty$ action by $\Gamma=\SL(n,\Z)$ on any nilmanifold $M$; \marginnote{believe?}
\item $G$ is as in Hypothesis \ref{HigherRank}, $\Gamma$ is a cocompact lattice in $G$, and $\alpha$ is a $C^\infty$ action by $\Gamma$ on any nilmanifold $M$;
%\item   $\Gamma$ is as in Hypotheses \ref{HigherRank} and  $\alpha$ is an action of  $\Gamma$    on a torus $\T^d$ which preserves a probability measure $\mu$;
\item   $\Gamma$ is as in Hypotheses \ref{HigherRank} and  $\alpha$ is an action of  $\Gamma$    on a compact nilmanifold $N/\Lambda$ which preserves a probability measure $\mu$.  %\item  $G$ is as in Hypothesis \ref{HigherRank}, $\Gamma$ is a lattice in $G$, and $\alpha$ is a $C^\infty$ action by $\Gamma$ on any nilmanifold $M$ preserving a probability measure.
\end{enumerate}
If $\alpha(\gamma)$ is Anosov for some $\gamma\in\Gamma$, then there is a finite-index subgroup $\Gamma'<\Gamma$, a $\Gamma'$-action $\rho$ on $M$ by linear automorphisms and a $C^\infty$ diffeomorphism $h\colon M\to M$, homotopic to identity,  such that $h\circ \alpha(\gamma)= \rho(\gamma)\circ h$ for all $\gamma\in\Gamma'$.
\end{corollary}

We remark that from Corollary \ref{AnosovLifting} one recovers the  previous local rigidity results for affine Anosov actions.

\subsection{Organization of Paper}
%In the remainder of this section we review some terminology and present some corollaries of the main technical theorems.
  In Section \ref{sec:Prelim} we present the major technical background and definitions for the paper. In Section \ref{sec:3}, we present the main technical theorems,  Theorems \ref{thm:mainnilmanifoldfull} and \ref{main:factorsfull}, which assert the existence of a semiconjugacy between a non-linear action and its linear data  under a number of technical hypotheses.   In Section \ref{sec:4} we introduce suspension spaces which converts the problem of building a semiconjugacy between $\Gamma$-actions into a problem of building a semiconjugacy between   $G$-actions. In Section \ref{sec:5}, we obtain this semiconjugacy for $G$-actions by first constructing it for a single element of a  Cartan subgroup in $G$, and then extending to a    semiconjugacy  between entire $G$-actions.
%  is also a semiconjugacy for the Cartan subgroup and the root subgroups.
In Section \ref{sec:bigfacts} we present a number of classical results that will be used in the following sections.
  In Section \ref{sec:6}, we  show that the technical assumptions required by Theorem \ref{thm:mainnilmanifoldfull} are   satisfied in the setting of Theorem \ref{ConjThm}. In Section \ref{sec:Smooth}, we prove Theorem \ref{SmoothThm} by finding a large abelian subgroup of $\Gamma$ whose action contains an Anosov diffeomorphism. Finally, the lifting hypothesis and Remark \ref{LiftingRmk} are discussed in Section \ref{sec:Lifting}.\newline

\noindent {\bf Acknowledgments.} We are grateful to Anatole Katok, David Fisher, and Ralf Spatzier for suggesting the problem and for inspiring discussions. We also thank Amir Mohammadi for helpful discussion regarding superrigidity.  A.~B. was partially supported by an NSF postdoctoral research fellowship DMS-1104013.
F.~R.~H. was supported by NSF grants DMS-1201326 and DMS-1500947.  Z.~W. was supported by NSF grants DMS-1201453 and DMS-1501295, as well as a research membership at the Institute for Advanced Study.

\section{Preliminaries}\label{sec:Prelim}

In this section we present the main definitions and constructions that will be used for our main technical theorems in Section \ref{sec:3}.  We also recall and prove some related facts.

\subsection{Linear data associated to torus and nilmanifold actions}\label{sec:lineardata}
%The main  goal of this project is to extend  Theorem \ref{thm:semiconj} to  continuous actions of discrete groups on tori and nilmanifolds.   To do so,
  To extend  Franks' Theorem  to  continuous actions of discrete groups on tori and nilmanifolds we need an appropriate notion of the linearization of   such an action.  % higher-rank lattices $\Gamma$.  % Our groups will be lattices $\Gamma\subset G$ in higher-rank semisimple Lie groups $G$.

\subsubsection{Linear data associated to torus actions.}
Given $B\in \GL(d,\Z)$, we write $L_{B}\colon \T^d\to \T^d$ for the the map on $\T^d$ induced by $B$.
Given $f\in \Homeo(\T^d)$, recall the unique $A_f\in \GL(d, \Z)$ as in \eqref{eq:1}  with $f$  homotopic to $L_{A_f}$.
For $f,g\in \Homeo(\T^d)$ we verify from the characterization  \eqref{eq:1}  that  $A_{f\circ g}= A_f A_g$.

Consider a discrete group $\Gamma$ and an action $\alpha\colon \Gamma\to \homeo(\T^d)$. % where $\T^d= \R^d/\Z^d$ is the standard torus.
It follows that there exists an induced homomorphism \begin{equation}\label{eq:lindatatorus} \rho\colon \Gamma \to \GL(d, \R),\quad \rho\colon \gamma\to A_{\alpha(\gamma)}.\end{equation}  Moreover, for each $\gamma\in \Gamma$,   $\alpha (\gamma)$ is homotopic to $L_{\rho(\gamma)}$.  %We all $\rho$ the \emph{linear data associated to $\alpah$}.
Below, we abandon the notation $L_{\rho(\gamma)}$ and simply write $\rho(\gamma) \colon \T^d\to \T^d$; whether $\rho(\gamma)$ is an element of $\GL(d,\R)$ or $\homeo (\T^d)$ will be clear from context.
The representation $\rho\colon \Gamma \to \GL(d, \R)$ is called the \emph{linear data} of $\alpha$.

\subsubsection{Linear data associated to  nilmanifold actions.}
In the case of  actions on nilmanifolds, the above situation is more complicated.  Indeed let $M= N/\Lambda$ where $N$ is a simply connected nilpotent Lie group and $\Lambda$ is a finite volume discrete subgroup.  Consider an action $\alpha \colon \Gamma\to \homeo(M)$.  As the map $\alpha(\gamma)\colon N/\Lambda\to N/\Lambda$ need not fix a base point,   the map $\Gamma\to \Aut(\Lambda)$ corresponding to \eqref{eq:lindatatorus} is  defined only up to conjugation.  Thus, one has an induced homomorphism $\alpha_\#\colon \Gamma \to \mathrm{Out}(\Lambda) = \Aut(\Lambda)/\mathrm{Inn}(\Lambda)$.  % (Note in the case $N= \R^d$ above, $\mathrm{Inn}(\Lambda)$ is trivial.)

However, under the additional assumption that $\alpha\colon \Gamma\to \homeo(M)$ lifts to an action $\td \alpha \colon \Gamma \to \homeo (N)$  by the canonical identification of  $\Lambda$ with the group of deck transformations for the cover $N\to N/\Lambda$, we obtain a well defined  action $\rho\colon \Gamma\to \Aut (\Lambda)$.  As it is necessary in our method of proof to assume the lift  $\td \alpha$ of the action exists, this is not a very restrictive assumption.
By \cite[Theorem 5]{MR0039734}, every element of $\Aut (\Lambda)$ extends uniquely to an element of $\Aut (N)$; in particular,  we may extend $\rho$ to a homomorphism $\rho\colon \Gamma\to \Aut (N)$.  Moreover, for each $\gamma\in \Gamma$ we have that $\alpha (\gamma) $ is homotopic to $\rho(\gamma)\colon N/\Lambda\to N/\Lambda$.  Here, as above, we use $\rho(\gamma)$ to indicate  both an element of $\Aut(N)$ and the induced element of  $\Aut(N/\Lambda)$.
%However, under the additional assumption that $\alpha\colon \Gamma\to \homeo(M)$ lifts to an action $\td \alpha \colon \Gamma \to \homeo (N)$, by the canonical identification of  $\Lambda$ with the group of deck transformations for the cover $N\to N/\Lambda$, we obtain a well defined  action $\rho\colon \Gamma\to \Aut (\Lambda)$.
%By \cite{MR0039734}*{Theorem 5}, every element of $\Aut (\Lambda)$ extends uniquely to an element of $\Aut (N)$; in particular,  we may extend $\rho$ to a homomorphism $\rho\colon \Gamma\to \Aut (N)$.  Moreover, for each $\gamma\in \Gamma$ we have that $\alpha (\gamma) $ is homotopic to $\rho(\gamma)\colon N/\Lambda\to N/\Lambda$.  Here, as above, we use $\rho(\gamma)$ to indicated both an element of $\Aut(N)$ and the induced element of  $\Aut(N/\Lambda)$.
\begin{definition}
If $\alpha\colon \Gamma\curvearrowright N/\Lambda$ either acts on a torus or lifts to the universal cover $N$, we call $\rho\colon \Gamma\to \Aut(N)$ (or the induced $\rho\colon \Gamma\to \Aut(N/\Lambda)$)   the \emph{linear data} associated to $\alpha$.
\end{definition}

\subsection{Structure of compact nilmanifolds}\label{sec:StructureNil}

We collect some standard facts about nilpotent Lie groups and their lattices. A standard reference is \cite{MR0507234}.

Let $N$ be a simply connected, nilpotent Lie group and $\Lambda\subset N$ be a lattice. Write $M=N/\Lambda$ for the quotient nilmanifold.  We have that $M$ is compact.
The set $\exp^{-1}(\Lambda)$ generates a lattice in the Lie algebra $\lien$. This, together with the coordinate system $\exp\colon \lien\to N$, determines a $\bQ$-structure on $N$.  For a connected closed normal subgroup $N'\lhd N$, the following are equivalent:
\begin{enumerate}
 \item $N'$ is defined over $\bQ$;
 \item $N'\cap\Lambda$ is a lattice in $N$;
 \item $\Lambda/(N'\cap\Lambda)$ is a lattice in $N/N'$;
 \item $N/N'\Lambda=(N/N')\Big/\big(\Lambda/(N'\cap\Lambda)\big)$ defines a compact nilmanifold that is naturally a quotient of $N/\Lambda$.
\end{enumerate}
In any of the above  cases, we say $N'$ is {\bf rational} and $M'=N/N'\Lambda$ is an {\bf algebraic factor} of $M$.

Recall that for an automorphism  $f\in \mathrm{Aut}(N)$, $D_e f$ is an automorphism of $\lien$, and $f\circ\exp=\exp\circ D_e f$. $f$ preserves $\Lambda$ if and only if $D_e f$ preserves $\exp^{-1}(\Lambda)$. In this case, $f$  descends to an automorphism of $M$. Hence $\Aut(M)=\Aut(N)$ can be regarded as a subgroup of $\GL(d,\bZ)$  if we identify the subgroup generated by $\exp^{-1}(\Lambda)\subset\lien$ with $\bZ^d\subset\bR^d$.

If in addition $ f$ preserves a rational normal subgroup $N'$, then it further descends an automorphism of $M'$.

Let $Z(N)$ denote the center of $N$. Then
\begin{enumerate}
	\item $Z(N)$ is normal and rational;
	\item Any element in $\Aut(M)$ preserves $Z(N) $, and thus descends to an element of $\Aut(N/Z(N)\Lambda)$.
\end{enumerate}

It follows that we have a series of central extensions
\begin{equation}\label{eq:centralextensionN} N= N_0 \to N_1 \to N_2\to \dots \to N_{r-1} \to N_r = \{e\}.%\label{eq:centralextensionN}
\end{equation}
Here $r$ is the degree of nilpotency.  Each $N_i$ is a simply connected, nilpotent Lie group and the kernel of the map $N_i\to N_{i+1}$ is the center of $N_i$.
We have a corresponding  series of central extensions
\begin{equation} \label{eq:centralextensionL}\Lambda= \Lambda_0 \to \Lambda_1 \to \Lambda_2\to \dots \to  \Lambda_r = \{e\}\end{equation}
where $\Lambda_{i+1} = \Lambda_i/(\Lambda_i\cap Z(N_i))$
and $\Lambda_i$ a lattice in $N_i$.  %As automorphisms preserve the center of  a group, an automorphism $ f$ of $N$ preserving $\Lambda$  descends inductively to an automorphism of $M_i=N_i/\Lambda_i$ for each $i$.

Suppose $\Gamma$ is a discrete group and suppose we have an action $\rho\colon \Gamma\to \Aut (\Lambda)$. % By Malcev it follows that
 $\rho$ extends uniquely to an action $\rho\colon \Gamma\to \Aut (N)$ and induces an action by homeomorphisms on $N/\Lambda$.  Moreover,
%that
$\rho$ descends to a $\Lambda_i$-preserving action $\rho_i\colon \Gamma\to \Aut(N_i)$ for every element of the central series \eqref{eq:centralextensionN} and \eqref{eq:centralextensionL}.  As above, we write $\rho_i(\gamma)$ to denote both an element of $\Aut(N_i)$ and the induced element of $\Aut(M_i)$.

\subsection{$\pi_1$-factors} \label{sec:factors}
In the sequel, we establish the existence of a semiconjugacy between actions in a more general setting than in the introduction.
  Consider $M$ any connected finite CW-complex.  %differentiable  manifold.  % with fundamental group $\Lambda_M$.
  Let $\td M$ be any normal %\footnote{check that this is needed}
  covering of $M$ and let $\Lambda_M$ denote the corresponding group of deck transformations.  We denote the action of $\Lambda_M$ on $\td M$ on the  right.   Let $\Gamma$ be a discrete group and $\alpha\colon \Gamma\to \homeo(M)$ an action.  We assume $\alpha$ lifts to an  action $\td \alpha\colon \Gamma\to \homeo(\td M)$; we then obtain an   induced  action $\alpha_* \colon \Gamma\to \Aut(\Lambda_M)$ defined by $$\td \alpha(\gamma)(x\lambda) = \td \alpha(\gamma)(x)  \alpha_*(\gamma)(\lambda).$$

Let $N$ be a simply connected, nilpotent Lie group and let $\Lambda\subset N$ be a lattice.
Suppose there is a surjective homomorphism $P_*\colon \Lambda_M\to \Lambda$.
%Recall   $\alpha$ lifts to an action of $\td M$, we obtain an $\alpha_*\colon \Gamma\to \Aut(\Lambda_M)$.
 We moreover assume that $\alpha_*(\gamma)(\ker P_*) = \ker P_*$  for all $\gamma\in \Gamma$.  Then   $\alpha_*$ induces an action
 $$\rho\colon \Gamma\to \Aut(\Lambda),\quad \rho(\gamma)(P_*(\lambda)) = P_*\alpha_*(\gamma)(\lambda).$$
 We extend $\rho$ to  $\rho\colon \Gamma\to \Aut(N)$.   As $N/\Lambda$ is a  $K(\Lambda,1)$,  there is a continuous  $ P\colon M\to N/\Lambda$ such that $ P$ lifts to $\td P\colon \td M\to N$ and the map between deck transformation groups  $\Lambda_M$ and  $\Lambda$  induced by $\td P$  coincides with $P_*$; \new{that is, $\td P(x\lambda) = \td P(x) \cdot  P_*(\lambda)$}.  %(See \cite{MR1866848}*{p.g. 184} and  \cite{MR1867354}*{Prop 1B.9}.)
In particular,  for each $\gamma\in \Gamma$ we have $P\circ \alpha (\gamma) \colon M\to N/\Lambda$ is homotopic to $\rho(\gamma)\circ P \colon M\to N/\Lambda$.

\begin{definition}
Under that above hypotheses, we say that the action  $\rho\colon \Gamma\to \Aut(\Lambda)$ (or $\rho\colon \Gamma\to \Homeo(N/\Lambda)$)  is a \emph{$\pi_1$-factor of $\alpha$} induced by the map $P_*$.
%$\alpha$ \emph{factors through $\rho\colon \Gamma\to \Aut(N/\Lambda)$} via $\td P\colon \td M\to N$.
\end{definition}

Let $P_0= P$ and $P_i$ be the composition of $P$ with the natural map $N/\Lambda\to N_i/\Lambda_i$ where $N_i$ and $\Lambda_i$ are as in \eqref{eq:centralextensionN} and \eqref{eq:centralextensionL}.  We similarly obtain maps $\td P_i\colon \td M \to N_i$ and \new{$P_{*,i} \colon \Lambda_M\to \Lambda_i$.}   Let $\rho_i$ be the action of $\Gamma$ on $N_i$ induced by $\rho$.

We have the following
\begin{claim}
 If $\rho$ is a $\pi_1$-factor of $\alpha$ induced by $P_*$ then  for every $i$, $\rho_i$ is a $\pi_1$-factor of $\alpha$ induced by $P_{*,i}$.
%   $\alpha $ factors through $\rho_i\colon \Gamma\to \Aut(N_i/\Lambda_i)$ via $\td P_i\colon \td M\to N_i$ then
\end{claim}

\begin{remark} The requirement that $P_*$ is surjective is not very restrictive. In fact, if the image of $P_*$ is a proper subgroup $\Lambda'$ of $\Lambda$, then $\Lambda'$ will be a lattice in its Zariski closure $N'$ (\cite[Theorem II.2.3]{MR0507234}). After replacing $N$ and $\Lambda$ with $N'$ and $\Lambda'$, we have a $\pi_1$-factor.\end{remark}

\def\wrd{d_{\mathrm{word}}}

\subsection{Coarse geometry of lattices}\label{sec:QI}
Let $G$ be a connected semisimple Lie group equipped with a right-invariant metric  and let  $\Gamma\subset G$ be a finitely generated discrete subgroup.
We may equip $\Gamma$ with two metrics: the word metric $\wrd$ induced by a fixed choice of generators and a right-invariant metric $d_G$ on $\Gamma$ inherited as a subset of $G$ from a right invariant Riemannian metric on $G$.

We say these two metrics are \emph{quasi-isometric} (or   that $\Gamma$ is \emph{quasi-isometrically embedded in $G$}) if there are $A>1$ and $B>0$ such that for any $\gamma_1, \gamma_2\in \Gamma$ we have
\begin{equation}\label{eq:QI}A\inv\cdot \wrd (\gamma_1, \gamma_2) - B\le d_G(\gamma_1, \gamma_2) \le A\cdot \wrd (\gamma_1, \gamma_2) + B.\end{equation}
Note that all word metrics are quasi-isometric.  For the remainder  we fix a finite generating set  $F=\{\gamma_\ell\}$ for $\Gamma$ with induced word metric $\wrd(\cdot, \cdot)$.

Note that if $\Gamma\subset G$ is a cocompact lattice then $\Gamma$ is automatically quasi-isometrically embedded in $G$.
For non-uniform lattices, we have the following result.  % from \cite{MR1828742}.  %, we have:

\begin{theorem}[Lubotzky-Mozes-Raghunathan, \cite{MR1828742}] \label{QI}%{\rm (Lubotzky-Mozes-Raghunathan)}
$\Gamma$ is quasi-isometricallly embedded in $G$ if the projection of $\Gamma$ to any $\R$-rank 1 factor is dense. In particular,  $\Gamma$ is quasi-isometrically embedded in $G$ under Hypothesis \ref{HigherRank}.   \end{theorem}

\subsection{Non-resonant linear representations}\label{sec:NonRes}
Let $G$ be a semisimple Lie group.   Let  $\lieg$  be the Lie algebra of $G$.
%As usual, write $\lieg$ for the lie algebra of $G$.
We fix a Cartan involution  $\theta$  of $\lieg$ and write $\liek$ and $\liep$, respectively,  for the $+1$ and $-1$ eigenspaces of $\theta$.  Denote by $\liea$   the maximal abelian subalgebra of $\liep$ and by $\liem$ the centralizer of $\liea$ in $\liek$.  %Note that the elements of $\Sigma$ are real linear functionals on $\liea$.
Recall that $\dim_\R(\liea)$ is the $\R$-rank of $G$.

Consider  a linear representation  $\tau \colon G\to \GL(n, \R)$ of $G$.  Then $\tau$ induces a representation   $d\tau\colon \lieg\to \mathfrak{gl}(n,\R)$.  
Let $\{\chi_i\}$  denote the restricted weights of $d\tau$ relative to $\liea$ and let  $\Sigma:= \{\zeta_j\}$ denote the restricted roots of $\lieg$ relative to $\liea$; that is $\Sigma$ is the restricted weights of the adjoint representation.
%We let $\Sigma$ denote the set of
%restricted roots of $\lieg$ with respect to $\liea$.
Given a simple factor $\lieg'\subset \lieg$ we denote by  $\Sigma(\lieg')$ the irreducible restricted root system.
 Recall that each $\chi_i$ and $\zeta_j$ is a  real linear functional on $\liea$.  Given $\zeta\in \Sigma$, let $\lieg^\zeta$ denote the corresponding subspace.  Recall that $\lieg^0= \liem \oplus \liea.$

\begin{definition}
Let $\psi\colon  \lieg\to \mathfrak{gl}(n,\R)$ be a linear representation.
We say a restricted root $\zeta_j$ of $\lieg$ is {\bf resonant} (with $\psi$) if there is a $c>0$ and a restricted weight $\chi_i$ of $\psi$ such that $\zeta_j= c\chi_i$;  otherwise we say $\zeta_j$ is {\bf non-resonant}.
Let $\Sigma_{NR}$ denote the set of non-resonant restricted roots.

We say the representation $\psi$ is  {\bf  strongly non-resonant} (with respect to $\mathrm{Ad}$) if  every non-trivial restricted root of $\lieg$ is non-resonant with $\psi$.
 We say the representation $\psi$ is  {\bf weakly non-resonant} (with respect to $\mathrm{Ad}$) if the set $$\lieg^0\cup\bigcup_{\zeta\in \Sigma_{NR}}\lieg^\zeta$$ generates $\lieg$ as a Lie algebra.

\end{definition}
If $\psi = d\tau$ for $\tau \colon G\to \GL(n, \R)$, we say $\tau$ is  weakly non-resonant if $d\tau$ is.  
Note that the existence of a non-resonant root implies that the $\R$-rank of $G$ is at least 2.

In the remainder, we will be interested only in representations with all weights \emph{non-trivial}; that is, representations for which the  weight space corresponding to  the zero weight is trivial.  We note that for many classical simple Lie groups, in particular for $G= \SL(n,\R)$,  there are infinitely many irreducible representations with all weights non-trivial.

Given a semisimple Lie algebra $\lieg$ containing rank-1 factors, a representation $\psi\colon\lieg \to  \mathfrak{gl}(n,\R)$ with all weights non-trivial may or may not be weakly non-resonant.  However, if all non-compact factors have $\R$-rank at least 2, the following lemma guarantees that all  representations we consider in the sequel are  weakly non-resonant.  Note however, that there are representations with  all weights non-trivial
for which there are resonant restricted roots.

\begin{lemma}\label{NonResonance}Suppose $\lieg$ is a semisimple real Lie algebra such that every non-compact factor has $\R$-rank $2$ or higher.  %, $A$ is a maximal $\R$-split torus, and
Let $\psi$ be a finite-dimensional, real representation of $\lieg$ such that all restricted weights of $\psi$ are non-trivial.
Then $\psi$ is weakly non-resonant.

Moreover, if no non-compact simple factors of $\lieg$ have restricted root system of type  $C_\ell$ then $\psi$ is strongly non-resonant.
 \end{lemma}

 As an example showing that we must consider  weakly non-resonant representations,  consider the standard action of $\mathrm{Sp}(4, \R)$ on $\R^4$.  The Lie algebra $\lieg$ of $\mathrm{Sp}(4, \R)$  is of type $C_2$ (and is moreover a split real form).  Relative to a certain basis $\{e_1, e_2\}$ of $\liea$, the restricted roots of $\lieg$ are  $\{\pm \epsilon _1 \pm \epsilon_2\} \cup \{\pm 2 \epsilon_1, \pm 2 \epsilon_2\}$ where $\epsilon_i (e_j) =\delta_{ij}.$  Take a representation whose highest weight  is given by $\lambda= \epsilon_1$.  Then the weights of the representation are $\{\pm \epsilon_1, \pm \epsilon_2\} $, and  hence are all non-trivial.
 The resonant restricted roots are $\{\pm 2 \epsilon_1, \pm 2 \epsilon_2\}$; however the Lie algebra is generated by
$\lieg^0$ and the root spaces corresponding to the set of non-resonant roots $\{\pm \epsilon _1 \pm \epsilon_2\}$.

\begin{proof}  [Proof of Lemma \ref{NonResonance}]
We recall some facts from the representation theory of semisimple Lie algebras.  Though  usually stated for complex representations of complex Lie algebras, all facts used here hold for real representations of real Lie algebras.  Consider first $\psi$   irreducible.  Then there is a restricted weight $\lambda$, called the \emph{highest weight}, such that
every other weight $\chi$ of $\psi$ is of the form $$\chi = \lambda- \sum n_i \beta_i$$
where $\{\beta_i\}$ is a set of simple positive roots and $n_i$ are positive integers.
We have that every weight $\chi$ is \emph{algebraically integral}: that is,
$2\frac{\langle \chi,\xi\rangle}{\langle \xi,\xi\rangle} \in \Z$   for every root $\xi\in \Sigma$,  where $\langle \cdot,\cdot\rangle$ is the inner product on $\liea^*$ induced from the Killing form.
Given a simple positive root $\beta_j$ there is a distinguished \emph{fundamental weight} $\varpi_j$ defined by  $2\frac{\langle \varpi_i,\beta_j\rangle}{\langle \beta_j,\beta_j\rangle} = \delta_{ij}.$
We have that every highest weight $\lambda$ is a positive integer combination of the fundamental weights $\varpi_i$.
%We note that the highest weight $\lambda$ may not uniquely determine an irreducible representation.

If $\lieg$ is simple and the root system  $\Sigma(\lieg)$ is reduced with   Cartan matrix $C= [C_{ij}]$, then the simple roots and fundamental weights are related by $\beta_i =\sum C_{ij} \varpi_j$.
The only simple non-reduced root system is of type $BC_\ell$; the roots of $BC_\ell$ are the union of the roots of $B_\ell$ and $C_\ell$ and the fundamental weights are those from $C_\ell$.

We proceed with the proof of the lemma.  Suppose that a weight $\chi$ and a root $\xi$ are positively proportional.    We can assume $\chi$ is a weight of an irreducible component  of $\psi$ and that $\xi$ is a root of a non-compact simple factor $\lieg_k$  of $\lieg=\oplus \lieg_k$.
Moreover, as the fundamental weights for distinct simple factors of $\lieg$ are linearly independent, we may assume $\chi = \lambda- \sum n_i \beta_i$ where $\beta_i$ are simple roots for $\Sigma(\lieg_k)$ and $\lambda= \sum k_j \varpi_j$ where $\varpi_j$ are the fundamental weights of the root system $\Sigma(\lieg_k )$ and $k_j$ are non-negative integers.
We may also take $\xi$ so that    $\frac 1 2 \xi$ is not a root.  Then there is an element of the Weyl group of $\lieg_k$ that sends $\xi$ to a simple root $\beta_{i_0}$ of $\Sigma(\lieg_k)$ \cite[Proposition 2.62]{MR1920389}. Moreover, the Weyl group preserves weights of $\psi $ hence we may assume that $\chi$ is positively proportional to a simple positive root $\beta_{i_0}$ of $\lieg_k$.

First consider the case that $\Sigma(\lieg_k)$ is not of type $B_\ell, C_\ell$ or $BC_\ell$.
Suppose $\chi= t\beta_{i_0}$.  We have $\chi = \sum_{j} m_j \varpi_j$ for some integers $m_j$.  Since the functionals $\varpi_j$ are linearly independent, it follows that $m_j = t   C_{i_0j}$ for every $j$.  Then $t$ is rational and $t= \frac{p}q$ where $q\in \mathbb{N}$ is smaller than the greatest common factor of all entries in the ${i_0}$th row of $[C_{ij}]$.  For $\Sigma(\lieg_k)$   not of type $B_\ell, C_\ell$ or $BC_\ell$, the entries of every row of the corresponding  Cartan matrix $[C_{ij}]$ have greatest common factor of $1$.  Thus $t$ is an integer.
Then, the weights of $\psi$ include the chain $-t\beta_{i_0}, -(t-1)\beta_{i_0},\cdots, (t-1)\beta_{i_0}, t\beta_{i_0}$.  As we assume $0$ is not a (non-trivial) weight, it follows that no such positively proportional pair $\chi$ and $\beta_{i_0}$ exist.
It follows that if all simple factors $\lieg_k$ are not of type $B_\ell, C_\ell$ or $BC_\ell$ then $\psi$ is strongly non-resonant.

In the case that $\Sigma(\lieg_k)$ is of type $C_\ell$, the Cartan matrix contains 1 row whose entries have greatest common factor 2; all other rows have greatest common factor 1.   Then there is at most 1 simple root that is resonant with $\psi$.  The orbits of the remaining simple roots under the Weyl group generate all of $\lieg_k$.
In the case that $\Sigma(\lieg_k)$ is of type $B_\ell$, from the tables of root and fundamental weight data (c.f.\ \cite[Appendix C]{MR1920389}), the only root  system of type $B_\ell$ admitting a representation with resonant roots and all weights non-trivial occurs for $B_2$.  However, $B_2 $ is isomorphic to $C_2$.

Finally, if   $\Sigma(\lieg_k)$ is of type $BC_\ell$ then the fundamental weights of $\Sigma(\lieg_k)$ coincide with those of type $C_\ell$ and, comparing tables of root data (c.f.\ \cite[Appendix C]{MR1920389}), it follows that the   fundamental weights of  $\Sigma(\lieg_k)$ are linear combinations of restricted roots of  $\Sigma(\lieg_k)$.  In particular, if $\chi$ is resonant with a root of $\lieg_k$ of type $BC_\ell$, then $\chi  =  k \beta_i$ for some positive integer $k$.  It follows that $0$ is a non-trivial weight  of $\psi$.   \end{proof}

\section{The main technical theorem}\label{sec:3}
To state the main technical theorem,  fix $G$ to be a connected semisimple Lie group with finite center.  %We may assume $G = \prod G_i$.
\subsection{Main theorem: actions on nilmanifolds}
\begin{theorem}\label{thm:mainnilmanifoldfull}
%Let $G$ be a semisimple Lie group and
Let $\Gamma\subset G$ be a lattice. Let $N$ be a simply connected nilpotent Lie group with Lie algebra $\lien$ and let $\Lambda\subset N$ be a lattice. Let  $M = N/\Lambda$  and let $\alpha\colon \Gamma\to \Homeo(M)$ be an action.
Assume   $\alpha \colon \Gamma\to \Homeo(M)$ lifts to an action $\td \alpha \colon \Gamma\to \Homeo(N)$ and let $\rho\colon \Gamma\to \Aut(N)$ denote the associated linear data.

Assume the following technical hypotheses are satisfied:

\begin{enumerate}
	\item the linear data $\rho\colon \Gamma \to \Aut(N)$ is the restriction to $\Gamma$ of a continuous morphism $\rho\colon G\to \Aut(N)$;
	\item $\Gamma$ is quasi-isometrically embedded in $G$;
	\item the representation $D\rho\colon G\curvearrowright\lien$ is weakly non-resonant with $\mathrm{Ad}\colon G\curvearrowright \lieg$;
	\item all restricted weights of the representation $D\rho$ with respect to $\liea$ are non-trivial, where $\liea$ is as in Section \ref{sec:NonRes}.
\end{enumerate}
Then, there exists a surjective,  continuous  map $h\colon M\to M$, homotopic to the identity, such that
	$$h\circ \alpha(\gamma)= \rho(\gamma)\circ h$$ for every $\gamma\in \Gamma$.
\end{theorem}

The surjectivity of the map $h$ in Theorem  \ref{thm:mainnilmanifoldfull} follows from elementary degree arguements.

 \subsection{Main theorem: $\pi_1$-factors}
Under the setup introduced in Section \ref{sec:factors}, we have the following generalization of Theorem  \ref{thm:mainnilmanifoldfull}.
\begin{theorem}\label{main:factorsfull}
%Let $G$ be a semisimple Lie Group and
Let $\Gamma\subset G$ be a lattice.
Let $M$ be a connected  finite CW-complex  and $\alpha\colon \Gamma\to \Homeo(M)$ an action.   Suppose for some normal cover $\td M$ of $M$ with deck group $\Lambda_M$ we have that $\alpha$ lifts to an action $\td \alpha \colon \Gamma \to  \Homeo(\td M)$.

Let $N$   be a nilmanifold and $\Lambda\subset N$ a lattice.  Assume there is a surjective homomorphism $P_*\colon \Lambda_M\to \Lambda$ inducing a $\pi_1$-factor $\rho\colon \Gamma\to \Aut (N)$.
%Moreover assume there  is a map $\td P\colon \td M\to  N$  such that  $\alpha $ factors through $\rho\colon G\to \Aut (N)$ via $\td P \colon \td M \to N$.
Assume the following technical hypotheses are satisfied:
\begin{enumerate}
	\item the linear representation $\rho\colon \Gamma \to \Aut(N)$ is the restriction to $\Gamma$ of a continuous morphism $\rho\colon G\to \Aut(N)$;
	\item $\Gamma$ is quasi-isometrically embedded in $G$;
	\item the representation $D\rho\colon G\curvearrowright\lien$ is weakly non-resonant with $\mathrm{Ad}\colon G\curvearrowright \lieg$;
	\item all restricted weights of the representation $D\rho$ with respect to $\liea$ are non-trivial, where $\liea$ is as in Section \ref{sec:NonRes}.
\end{enumerate}
Then, there exists a continuous map $h\colon M\to N/\Lambda$, homotopic to  $ P\colon M\to N/\Lambda$, such that
	$$h\circ \alpha(\gamma)= \rho(\gamma)\circ h$$ for every $\gamma\in \Gamma$.
\end{theorem}

We prove Theorem  \ref{main:factorsfull} inductively on the step of nilpotency.  As in Section \ref{sec:factors}, given $\rho\colon G\to \Aut (N)$, let $\rho_i\colon G\to \Aut(N_i)$ denote the induced action on the factor $N_i$ of \eqref{eq:centralextensionN}.  \new{Recall for $0\le i\le r$ we let  $ P_i\colon  M \to N_i/\Lambda_i$,  $ \td P_i\colon  \td M \to N_i$, and $P_{*,i}\colon \Lambda_M\to \Lambda_i$  be the compositions   of $ P$,  $\td P$, and $P_*$ followed by the the natural projection $N/\Lambda \to N_i/\Lambda_i$, $N\to N_i$, $\Lambda\to \Lambda_i$.    Note   that   $\td P_i(x\cdot \lambda ) = \td P_i(x)\cdot P_{*,i}(\lambda)$.  If $h\colon M\to N_i/\Lambda_i$ is homotopic to $P_i$ we say a lift $\td h \colon \td M_i \to N_i$ is $\Lambda_M$-equivariantly homotopic to $\td P_i$ if there is a homotopy from $\td h$ to $\td P_i$ that factors over a homotopy from $h $ to $P_i$.  If $\td h$ is  $\Lambda_M$-equivariantly homotopic to $\td P_i$ we have $\td h(x\cdot \lambda) = \td h(x) \cdot P_{*,i}(\lambda)$.  }
%and $\rho_i$ for the induced linear data.

Also note that $N_i/\Lambda_i$ has a natural structure of a  fiber bundle  over $N_{i+1}/\Lambda_{i+1}.$  Note that if conditions (1), (3), and (4) of Theorem \ref{main:factorsfull} hold, then they hold for the  action $\rho_i\colon G\to \Aut (N_i)$.

\begin{theorem}\label{thm:inductive}
Let $M$, $G$, $\Gamma$, $\alpha$ and $\rho$ be as in Theorem \ref{main:factorsfull}.
 Let $N_i/\Lambda_i$ be one of the factors appearing in \eqref{eq:centralextensionN} and \eqref{eq:centralextensionL}.  Assume there exists a map $h_{i+1}\colon M \to N_{i+1}/\Lambda_{i+1}$, homotopic to  $P_{i+1}\colon M\to N_{i+1}/\Lambda_{i+1}$, such that
 	$$h_{i+1}\circ \alpha(\gamma) = \rho_{i+1} (\gamma)\circ h_{i+1}$$
	for all $\gamma\in \Gamma$.  Moreover, assume $h_{i+1}\colon M \to N_{i+1}/\Lambda_{i+1}$ lifts to $\td h_{i+1}\colon \td M \to N_{i+1} $ with %\begin{equation}\label{eq10}
	$\td h_{i+1}\circ \td \alpha(\gamma) = \rho_{i+1} (\gamma)\circ \td h_{i+1} $ \new{and $\td h_{i+1}$   $\Lambda_M$-equivariantly homotopic to $\td P_{i+1}$.} 
	
	% \end{equation}

Then there exists a continuous map $h_i\colon M \to N_{i}/\Lambda_{i}$ such that $h_i$ is homotopic to  $P_{i}\colon M\to N_{i}/\Lambda_{i}$, $h_i\colon M\to N_i/\Lambda_i$ lifts $h_{i+1}$, and
 	$$h_{i}\circ \alpha(\gamma) = \rho_{i} (\gamma)\circ h_{i}$$
	for all $\gamma\in \Gamma$.
Moreover, $h_i$ is the \new{unique} map having a lift  $\td h_i\colon \td M\to N_i$ with $\td h_{i}\circ \td \alpha(\gamma) = \rho_{i} (\gamma)\circ \td h_{i}$    \new{and $\td h_{i}$   $\Lambda_M$-equivariantly homotopic to $\td P_{i}$.} % and the map  $h_i$ is uniquely determined by the above properties.}

\end{theorem}

Theorem \ref{main:factorsfull} follows immediately from backwards induction on Theorem \ref{thm:inductive} (with base case $i= r-1$ and  $N_{i+1} = N_r= \{e\}$).  Theorem \ref{thm:mainnilmanifoldfull} follows immediately from Theorem \ref{main:factorsfull} taking $P_*$ and $P$ to be the identity maps.

\section{Preparatory constructions for the proof of Theorem \ref{thm:inductive}}\label{sec:4}
%In what follows we omit the subscripts on the actions $\rho_i$.

\subsection{Lifting property}\label{sec:lifting}
We retain all notation  appearing in Theorem \ref{thm:inductive}.
Recall that $N_{i}/\Lambda_i$ has the structure of  a fiber bundle over $N_{i+1}/\Lambda_{i+1}$ (with fiber $Z_i/(\Lambda_i\cap Z_i)$ isomorphic to $\T^{d_i}$.)
By construction,  $P_{i+1}\colon  M \to N_{i+1}/\Lambda_{i+1}$  lifts to $ P_i\colon  M \to N_{i}/\Lambda_{i}$.
Write $p_{i, i+1}\colon N_{i}/\Lambda_{i} \to N_{i+1}/\Lambda_{i+1}$  and 
$\td p_{i, i+1}\colon N_{i} \to N_{i+1} $ for the natural projection maps. 
As we assume $h_{i+1}$  is homotopic to $P_{i+1}$, by the lifting property of fiber bundles we may find a continuous  $\phi\colon  M \to N_{i}/\Lambda_{i}$ such that
	\begin{enumerate}
		\item $\phi$ is homotopic to $P_i$, 
		\item $\phi$ is a lift of $h_{i+1}$,
		\item \new{the homotopy from $\phi$ to $P_i$ factors through $p_{i, i+1}$ to the homotopy from $h_{i+1} $ to $P_{i+1}$.}
	\end{enumerate}
%Write $p_{i, i+1}\colon N_{i}/\Lambda_{i} \to N_{i+1}/\Lambda_{i+1}$ for the natural projection map.  
%In particular$p_{i, i+1}\circ \phi = h_{i+1}$.  
In particular, as $h_{i+1}$ intertwines the linear and non-linear $\Gamma$-actions, we have equality of maps from $M\to N_{i+1}/\Lambda_{i+1}$
\begin{equation}\label{eq:partialconj}p_{i, i+1}( \phi (\alpha(\gamma)(x))) = \rho_{i+1}(\gamma) (p_{i, i+1}\circ \phi (x))\end{equation} for all $\gamma\in \Gamma$.
Our goal in proving Theorem \ref{thm:inductive} will be to correct $\phi$ so that \eqref{eq:partialconj} remains valid without the projection factor $p_{i, i+1}$.

\new{
Applying the homotopy lifting property to the bundle $\td M\to M$ we may select a distinguished lift  $\td \phi\colon M\to N_i$ such that $\td \phi$ is $\Lambda_M$-equivariantly homotopic to $\td P_i$.  
Note that  $\td p_{i,i+1}\circ \td \phi$ is a lift $p_{i,i+1}\circ \phi= h_{i+1}$.  Moreover, the image of the homotopy from $\td \phi$ to $\td P_i$ under   $\td p_{i,i+1}$ is a lift of the homotopy from $h_{i+1}$ to $P_{i+1}$.  Since $\td p_{i,i+1}\circ \td P_{i} = \td P_{i+1}$ it follows that $\td p_{i,i+1}\circ \td \phi = \td h_{i+1}$.}
%$\td p_{i,i+1}\circ \td \phi = 
In particular for  $\td \phi$
we have %may select a lift $\td \phi\colon M\to N_i$ with
\begin{equation}\label{eq4}\td p_{i, i+1} \circ \td \phi \circ \td \alpha (\gamma )= \rho_{i+1}(\gamma) \circ \td p_{i, i+1} \circ \td \phi \end{equation}
and $\td \phi(x\cdot \lambda) = \td \phi(x) \cdot P_{*,i}(\lambda)$. 

%Write $\td p_{i,i+1}\colon N_i\to N_{i+1}$ for the natural projection.  Recall $h_{i+1}$ is assumed to lift to $\td h_{i+1}\colon \td M\to N_{i+1}$ intertwining the $\Gamma$-actions $\td \alpha$ and $\rho$ and $\Lambda_M$-equivariantly homotopic to $\td P_{i+1}$.  
%
%
%
%Note that given any lift $\td \phi:\td M\to N_i$ of $\phi$,  $\td p_{i,i+1}\circ \td \phi$ is a lift $p_{i,i+1}\circ \phi= h_{i+1}$.  It follows we may select a lift $\td \phi$ of $\phi$ with $\td p_{i,i+1}\circ \td \phi = \td h_{i+1}$; in particular for such a lift $\td \phi$
%we may select a lift $\td \phi\colon M\to N_i$ with
%\begin{equation}\label{eq4}\td p_{i, i+1} \circ \td \phi \circ \td \alpha (\gamma )= \rho_{i+1}(\gamma) \circ \td p_{i, i+1} \circ \td \phi.\end{equation}

%We fix the distinguished lift $\td \phi$ for this and the following section.

\subsection{Suspension spaces.}\label{sec:suspspace}
Recall that, as $\alpha$ is assumed to lift to $\td \alpha \colon \Gamma\to \Homeo(\td M)$,  we have an action $\alpha_* \colon \Gamma\to \Aut(\Lambda_M)$.
We define the (right) semi-direct product $\Gamma\ltimes_{\alpha_*}\Lambda_M$ by $$( \gamma,\lambda) \cdot (\bar \gamma, \bar \lambda) = ( \gamma \bar \gamma, \alpha_*(\bar \gamma\inv) (\lambda) \bar \lambda).$$
$\Gamma\ltimes_{\alpha_*} \Lambda_M$  acts on $G\times \td M$ on the right by
	$$(g,x) \cdot ( \gamma, \lambda) = (g\gamma, [\alpha(\gamma\inv)(x)]\lambda)$$
We similarly define the (right) semi-direct product $\Gamma\ltimes_\rho\Lambda_i$ by
%$$( \gamma,\lambda) \cdot (\bar \gamma, \bar \lambda) = ( \gamma \bar \gamma, \rho(\bar \gamma\inv) (\lambda) \bar \lambda)$$
$$( \gamma,\lambda) \cdot (\bar \gamma, \bar \lambda) = ( \gamma \bar \gamma, \rho_i(\bar \gamma\inv) (\lambda) \bar \lambda)$$
acting on
$G\times N_i$  by $$
(g,n) \cdot  (\gamma,\lambda) = (g\gamma, n  \rho_i(g\gamma)(\lambda)).$$
We remark that the asymmetry in the actions is intentional.

We have right $\Gamma$- and $\Lambda_M$-actions (respectively $\Gamma$- and $\Lambda_i$-actions)  on $G\times \td M$ (resp.\ $G\times N_i$) induced by the natural embeddings of $\Gamma$ and $\Lambda_M$ into $\Gamma\ltimes_{\alpha_*}\Lambda_M$ (resp.\ $\Gamma$ and $\Lambda_i$ into $\Gamma\ltimes_{\rho}\Lambda_i$.)

$P_{*,i}\colon \Lambda_M \to \Lambda_i$ can be extended to $$\Psi\colon \Gamma\ltimes_{\alpha_*}\Lambda_M\to \Gamma\ltimes_\rho\Lambda_i$$
by $$\Psi(\gamma, \lambda) = (\gamma, P_{*,i}(\lambda)).$$
We check that $\Psi$ defines a homomorphism.

Recall we have a continuous representation $\rho\colon G\to \Aut(N)$ which in turn descends to $\rho_i\colon G\to \Aut(N_i)$.
We define left $G$-actions on $G\times \td M$ and $G\times N_i$ by $$a\cdot (g, x) = (ag, x),\quad \quad a\cdot (g,n) = (ag, \rho_i(a)n)$$
for all $a\in G, g\in G, x\in M$ and $n\in N_i$.   (Again the asymmetry in the definitions is intentional.)

As the left and right actions defined above commute, we obtain left $G$-actions on the quotient spaces
 \begin{enumerate}
	\item $M^\alpha:= G\times M /\Gamma\ltimes_{\alpha_*}\Lambda_M$;
	\item $(N_i/\Lambda_i)_\rho:= G\times N_i/\Gamma\ltimes_\rho\Lambda_i$.
\end{enumerate}
Here, the upper subscript denotes the standard suspension space construction.   The lower subscript denotes a twisted \emph{Lyapunov suspension} space.

%%\begin{remark}
%We remark that the twisted Lyapunov suspension $(N_i/\Lambda_i)_\rho$ is equivalent to the standard suspension construction induced by $(g,n)\cdot (\gamma,\lambda)=(g\gamma,\rho_i(\gamma^{-1})(n)\cdot\lambda)$. We choose the current construction as the  hyperbolicity of the left $G$-action is  better observed through this construction.
%Additionally, in this construction we may view $\Gamma \ltimes_\rho\Lambda_i$ as a subgroup of $G\ltimes_\rho N_i$ whence
%$(N_i/\Lambda_i)_\rho = G\ltimes_\rho N_i/\Gamma\ltimes_\rho\Lambda_i$ is a homogeneous space.  We use this point of view  in Proposition \ref{prop:Ratner} below.

%\end{remark}
\new{
\begin{remark}\label{rem:homospace}

We use the twisted Lyapunov suspension $(N_i/\Lambda_i)_\rho$ in this and the next section as the hyperbolicity of the left $G$-action on the fibers is best observed through this construction.  % is  better observed through.
%We remark that the  is equivalent to the standard suspension construction induced by $(g,n)\cdot (\gamma,\lambda)=(g\gamma,\rho_i(\gamma^{-1})(n)\cdot\lambda)$. 

However, the standard suspension of $\rho$ acting on $N_i/\Lambda_i$ has the advantage that it can be viewed  as a homogenous space, which we will use in  Proposition \ref{prop:Ratner} below.  Indeed, consider the semi-direct product $G \ltimes_\rho N_i$   given by $$(g,n)\cdot (\bar g, \bar n) = (g\bar g, \rho_i(\bar g\inv)(n) \bar n).$$
Then $\Gamma \ltimes_\rho\Lambda_i$ is a subgroup of $G \ltimes_\rho N_i$ and acts on the right as
$$(g,n)\cdot (\gamma,\lambda) = (g\gamma, \rho_i(\gamma\inv)(n)\lambda).$$
 For $a\in G$ we have $$(a,e)\cdot (g,n) = (ag, n)$$   inducing a left $G$-action which commutes with the right action of  $\Gamma \ltimes_\rho\Lambda_i$. We then obtain a natural $G$-action on the homogeneous space  $(N_i/\Lambda_i)^\rho:=(G \ltimes_\rho N_i)/(\Gamma \ltimes_\rho\Lambda_i).$   

Let $\td \Upsilon \colon G \ltimes_\rho N_i\to G \times N_i$ be given by $$\td \Upsilon (g,n) = (g, \rho_i(g)(n)).$$
We claim that $\Upsilon$ intertwines left $G$-actions and right $(\Gamma \ltimes_\rho\Lambda_i)$-actions and hence induces a continuous $$ \Upsilon\colon (N_i/\Lambda_i)^\rho\to  (N_i/\Lambda_i)_\rho$$ intertwining $G$-actions.  
Thus the two suspension spaces are equivalent.  

%We will use the homogenous space  point of view  $(N_i/\Lambda_i)^\rho$   in Proposition \ref{prop:Ratner} below.
\end{remark}
}

\subsection{Approximate conjugacy}

%Let $\td \phi\colon \td M\to N_i$ be a lift of the $\phi$ constructed in Section \ref{sec:lifting}.
%Let $\td p_{i, i+1}\colon N_i \to N_{i+1}$ denote the natural projection.
%
%\begin{lemma}
%We can pick $\td \phi$ so that for all $\gamma\in \Gamma$
%\begin{equation}\label{eq4}\td p_{i, i+1} \circ \td \phi \circ \td \alpha \gamma = \rho(\gamma) \circ \td p_{i, i+1} \circ \phi.\end{equation}
%\end{lemma}
%\begin{proof}
%This should follow from \eqref{eq1} and that $\tphi_{i+1}$ is homotopic to $P_{i+1}.$
%\end{proof}

We extend the map $\td \phi$ constructed above to a $\Psi$-equivariant map $\Phi\colon G\times \td M \to G\times N_i$ which %will induce a map $\hat \Phi\colon M^\alpha \to (N_i/Lambda_i)_\rho$ which in turn
intertwines the $G$-actions up to a defect that we will later correct.  This will in turn induce a semiconjugacy between the $G$-actions on $M^\alpha$ and $(N_i/\Lambda_i)_\rho$.

Fix a right-invariant Riemannian metric $d_G$ on $G$.  This  induces a metric $d_{G/\Gamma}$ on $G/\Gamma$. We  fix for the remainder a {\bf Dirichlet fundamental domain} for  $\Gamma$: that is, let  $D\subset G$  be a fundamental domain  for $\Gamma$ such that:
\begin{enumerate}
 \item $D$ contains an open neighborhood of the identity $e$;
 \item $D$ contains an open dense subset of full Haar measure;
 \item If $g\in D$, then $d(g, e)=\min_{\gamma\in\Gamma}d(g,\Gamma)=d_{G/\Gamma}(g\Gamma,\Gamma)$.
\end{enumerate}
%We notice such a domain is not unique,
%We will use the following elementary fact about lattices in conjunction with the quasi-isometry between the word and Riemannian metrics discussed above.
We will frequently use the following standard fact.
\begin{lemma}\label{Integrable}If $\Gamma$ is a lattice in  a semisimple Lie group $G$  then
$$d_{G/\Gamma}(g\Gamma,\Gamma)\in L^1(G/\Gamma,\rmm_G).$$\end{lemma}
Note that this is equivalent to saying that $d(g,e)$ is in $L^1(D,\rmm_G)$.

Let $D$ be the Dirichlet fundamental domain fixed above.  % in Section \ref{sec:QI}.
Given $g\in G$ let $\gamma_g\in \Gamma$ be the unique element with $g\gamma_g \inv \in D$; that is, $g\in D \gamma_g$.
Define  $\tphi_g\colon \td M \to  N_i$ by
%$$\tphi_g(x) = \rho(g) \rho (\gamma_g ) \td \phi( \alpha( \gamma_g\inv)x).$$
\begin{equation}\label{eq1}\tphi_g(x) = \rho_i(g\gamma_g\inv) \td \phi(\td  \alpha( \gamma_g)(x)).\end{equation}
Define  $\Phi \colon G\times \td M \to G\times N_i$ by $\Phi(g, x) = (g, \tphi_g(x))$.
Note that the kernel of $\td p_{i,i+1}$ is the center of  $N_i$ which is necessarily preserved by $\rho_i(g)$ for every $g$.  It follows that $$\rho_{i+1}(g) \circ \td p_{i, i+1} = \td p_{i, i+1} \circ \rho_{i} (g).$$
In particular,  for $g\in G$ we have \begin{equation}\td p_{i,i+1} \circ \tphi_g = \rho_{i+1}(g) \circ \td h_{i+1}.\end{equation}

We note that for any $a\in G$ we have
\begin{equation}\label{eq6}
%\td p_{i, i+1}\left( \Phi(sg, x) \right)= s \cdot\td p_{i, i+1}(\Phi(g,x)), \text{ or} \quad
\td p_{i, i+1} \circ \tphi_{ag}( x)= \rho_{i+1}(a)( \td p_{i, i+1}\circ \tphi_g(x)).
\end{equation}
Indeed,
%Then \eqref{eq6} followa aa
%\begin{align}
%\td p_{i, i+1}\left( \Phi(ag, x) \right)&= \left(ag, \td p_{i, i+1} \left(\tphi_{ag}(x) \right)\right) \notag\\
%\notag	&= \left(ag, \td p_{i, i+1} \left( \rho_{i}(ag\gamma_{ag}\inv) \td \phi(\td  \alpha( \gamma_{ag})x)  \right)\right)\\
%\notag		&= \left(ag, \rho_{i+1}(ag\gamma_{ag}\inv)  \td p_{i, i+1} \left(\td \phi( \td \alpha( \gamma_{ag})x)  \right)\right)\\
%\label{eq2}&= \left(ag, \rho_{i+1}(ag\gamma_{ag}\inv)\rho_{i+1}( \gamma_{ag})  \td p_{i, i+1} \left(\td \phi( x)  \right)\right)\\
%\notag	&= \left(ag, \rho_{i+1}(a)\rho_{i+1}(g\gamma_g\inv)\rho_{i+1}(\gamma_g)\td p_{i, i+1} \left(\td \phi(x)  \right)\right)\\
%\label{eq3}		&= \left(ag,\rho_{i+1}(a)  \rho_{i+1}(g\gamma_{g}\inv)  \td p_{i, i+1} \left(\td \phi( \td \alpha( \gamma_{g})x)  \right)\right)\\
%\notag&= \left(ag,\rho_{i+1}(a) \td p_{i, i+1} \left( \rho_i(g\gamma_{g}\inv)  \td \phi( \td \alpha( \gamma_{g})x)  \right)\right)\\
%	&= a \cdot\td p_{i, i+1}(\Phi(g,x)).\notag
%\end{align}
\begin{align}
 \td p_{i, i+1} \left(\tphi_{ag}(x) \right) \notag
\notag	&= \td p_{i, i+1} \left( \rho_{i}(ag\gamma_{ag}\inv) \td \phi(\td  \alpha( \gamma_{ag})x)  \right)\\
\notag		&= \rho_{i+1}(ag\gamma_{ag}\inv)  \td p_{i, i+1} \left(\td \phi( \td \alpha( \gamma_{ag})x)   \right)\\
\label{eq2}&=   \rho_{i+1}(ag\gamma_{ag}\inv)\rho_{i+1}( \gamma_{ag})  \td p_{i, i+1} \left(\td \phi( x)   \right)\\
\notag	&=   \rho_{i+1}(a)\rho_{i+1}(g\gamma_g\inv)\rho_{i+1}(\gamma_g)\td p_{i, i+1} \left(\td \phi(x)  \right)\\
\label{eq3}		&=  \rho_{i+1}(a)  \rho_{i+1}(g\gamma_{g}\inv)  \td p_{i, i+1} \left(\td \phi( \td \alpha( \gamma_{g})x)  \right) \\
\notag&=  \rho_{i+1}(a) \td p_{i, i+1} \left( \rho_i(g\gamma_{g}\inv)  \td \phi( \td \alpha( \gamma_{g})x)  \right).
%	&= a \cdot\td p_{i, i+1}(\Phi(g,x)).\notag
\end{align}
Above, \eqref{eq2} and \eqref{eq3} follow from \eqref{eq4}.
Our goal below will be to modify the family $\tphi_g$ so  that \eqref{eq6} holds without the projection.  %$$ \tphi_{sg}(x) = \rho(s) \tphi_g(x). $$

%\begin{align}\label{eq:babyequivaction}
%
%\end{align}
We claim
\begin{lemma}\label{claim:eqiv}
$\Phi$ is $\Psi$-equivariant: $$\Phi((g,x)\cdot (\gamma, \lambda)) = \Phi(g,x) \cdot \Psi(\gamma, \lambda).$$
In particular,
\begin{enumerate}
	\item $\tphi_{g\gamma}(\td \alpha(\gamma\inv)(x))=  \tphi_g(x)$;
	\item $\tphi_g(x\lambda) = \tphi_g(x) \rho_i(g) (P_{*,i}(\lambda))$.
\end{enumerate}
\end{lemma}

\begin{proof}
Note  that $\gamma_{g\gamma} = \gamma_g \gamma$.   We then have
	\begin{align*}
	\Phi\left((g, x) \cdot (\gamma, \lambda)\right)& = \left(g\gamma, \tphi_{g\gamma} \big(\td \alpha(\gamma\inv)(x)\lambda\big)\right)\\
	& =  \left(g\gamma,   \rho_i((g\gamma) \gamma_{g\gamma} \inv) \td \phi\left( \td \alpha( \gamma_{g\gamma})\big(\td \alpha(\gamma\inv)(x)\lambda\big)\right)   \right)\\
	& =  \left(g\gamma,   \rho_i(g\gamma_{g} \inv) \td\phi\left( \td \alpha( \gamma_{g\gamma})(\td \alpha(\gamma\inv)(x)) \alpha_*( \gamma_{g\gamma}) (\lambda)\right)\right)   \\
		& =  \left(g\gamma,  \left[ \rho_i(g\gamma_{g} \inv) \td\phi\big( \td \alpha( \gamma_{g})(x)\big) \right] \rho_i(g\gamma_{g} \inv)  P_{*,i}\left(  \alpha_*( \gamma_{g\gamma}) (\lambda)\right)\right)\\
	& =  \left(g\gamma,  \left[ \rho_i(g\gamma_{g} \inv) \td\phi\big( \td \alpha( \gamma_{g})(x)\big) \right]    \rho_i(g\gamma_{g} \inv)  \rho_i( \gamma_{g\gamma}) \left(P_{*,i}(\lambda))\right)\right)   \\
	& =  \left(g\gamma,  \left[ \rho_i(g\gamma_{g} \inv) \td\phi\big( \td \alpha( \gamma_{g})(x)\big) \right]    \rho_i(g\gamma )\left(P_{*,i}(\lambda\right))\right)   \\
	& =  \left(g,   \tphi_g(x)   \right)  \cdot \Psi(\gamma, \lambda) . \qedhere\\
%		& =  \left(g\gamma,   \rho(g\gamma_{g} \inv) \td\phi\left( \alpha( \gamma_{g\gamma})\big(\alpha(\gamma\inv)(x))\right) +   \rho(g\gamma_{g} \inv)\left(\rho( \gamma_{g\gamma}) (\lambda)\right) \right)  \\
		\end{align*}
\end{proof}

% We extend the definition of $\td p_{i, i+1}\colon N_i \to N_{i+1}$  to  $ \td p_{i, i+1}\colon G\times N_i \to G\times N_{i+1}$ by acting trivially in the first coordinate.
%that it is $G$-equivariant as well.

Note that $\Phi$ is only a measurable function.  However $\tphi_g\colon \td M\to N_i$ is  defined and continuous for {\it every} $g\in G$.  Moreover, from Lemma \ref{claim:eqiv}, for each $g\in G$ the map $\tphi_g$ factors to a map   $\phi_g\colon M\to N_i/\left(\rho_i(g) \Lambda\right)$.  In particular,  $\tphi_g$ is uniformly continuous for each $g\in G$.

\subsubsection{Central defect} \label{sss:central}
\def\liez{\mathfrak z}
Recall that the center $Z_i$ of $N_i$ is the kernel of $\td p_{i, i+1}\colon N_i\to N_{i+1}$.  Let $\liez_i$ denote the Lie algebra of $Z_i$.
%Identifying $Z_i$ and $z_i$ This gives $N_i$ the structure of an affine bundle\footnote{Not a vector bundle as there is no origin} over $N_i$.
%Fix an identification of $Z_i$ with $\R^d$.

\new{Recall the Lie algebra  $\liea $ defined in section \ref{sec:NonRes} and let $A$ be the analytic subgroup associated with $\liea$.}   By condition (4) of Theorem \ref{main:factorsfull}, for some $s\in A$, $S= \restrict{D\rho(s)}{T_eN} \in \Aut(\lien_i)$ is a hyperbolic matrix. Therefore, the restriction of $S$ to $\liez_i$ is hyperbolic.

We fix such an element $s$ from now on. Let $E^s$ and $E^u$ be the {stable and unstable subspaces for the restriction} of $S$ to $\liez_i$. % action of  $\rho(s)$ on the kernel $\R^d$.

From \eqref{eq6} it follows that given $g\in G$ and $x\in \td M$ there are unique vectors $\psi^{s}(g,x) \in E^{s}$ and $\psi^{u}(g,x) \in E^{u}$ such that
\begin{equation}\rho_i(s) \tphi_g(x) = \tphi _{sg} (x) \exp( \psi^{s}(g,x) )\exp(\psi^{u}(g,x))\end{equation}
where $\exp\colon \lien_i\to N_i$ is the Lie-exponential map.

\begin{lemma}\label{claim:invar}
The maps $G\times \td M\to E^{s/u}$,  given by $(g,x) \to \psi^{s/u}(g,x) $ are $\Gamma\ltimes_{\alpha_*}\Lambda_M$-invariant.
\end{lemma}
\begin{proof}
We have
\begin{align*}
s \cdot \Phi(g,x) = (sg, \rho_i(s)\tphi_g(x)) =  (sg,  \tphi _{sg} (x) \exp(\psi^{s}(g,x))\exp( \psi^{u}(g,x))).
\end{align*}
Moreover, since the left and right actions commute, repeatedly using Lemma \ref{claim:eqiv} we have
\begin{align*}
 s\cdot  \Phi&\left(g\gamma,\td \alpha(\gamma\inv)(x)\lambda)\right)\\
%	=&\rho(s) \Phi\left((g,x)\cdot (\gamma, \lambda)\right)\\
	 =& s\cdot \left(\Phi(g,x)\cdot \Psi (\gamma, \lambda)\right)\\
	 =&  \left(sg,  \tphi _{sg} (x) \exp(\psi^{s}(g,x))\exp( \psi^{u}(g,x))\right)\cdot\Psi (\gamma, \lambda)\\
 	 =& \left (sg\gamma,  \tphi _{sg} (x) \exp(\psi^{s}(g,x))\exp( \psi^{u}(g,x)))\rho_i(sg\gamma)(P_{*,i}(\lambda))\right)\\
 	 =& \left (sg\gamma,  \tphi _{sg} (x) \rho_i(sg\gamma)(P_{*,i}(\lambda))    \exp(\psi^{s}(g,x))\exp( \psi^{u}(g,x)))\right)\\
 	% =& {\red \left (sg\gamma,  \tphi _{sg} (x)P_{*,i}( \alpha_*(sg\gamma)(\lambda))    \exp(\psi^{s}(g,x))\exp( \psi^{u}(g,x)))\right) } \\
 	 =& { \left (sg\gamma,  \tphi _{sg} (x)\rho_i(sg)(P_{*,i}( \alpha_*(\gamma)(\lambda)))    \exp(\psi^{s}(g,x))\exp( \psi^{u}(g,x)))\right) } \\
	 =& {\left (sg\gamma,  \tphi _{sg} (x\alpha_*(\gamma)(\lambda))     \exp(\psi^{s}(g,x))\exp( \psi^{u}(g,x)))\right)}\\
%
%	 =& {\green \left (sg\gamma,  \tphi _{sg} (x\rho(\gamma)(P_{*,i}(\lambda)))     \exp(\psi^{s}(g,x))\exp( \psi^{u}(g,x)))\right)}\\
% 	 =& {\green \left (sg\gamma,  \tphi _{sg\gamma} (\alpha(\gamma\inv)(x\rho(\gamma)(P_{*,i}(\lambda)))      \exp(\psi^{s}(g,x))\exp( \psi^{u}(g,x)))\right)}\\
 	 =& \left (sg\gamma,  \tphi _{sg\gamma} (\td \alpha(\gamma\inv)(x\alpha_*(\gamma)(\lambda)))      \exp(\psi^{s}(g,x))\exp( \psi^{u}(g,x)))\right)\\
 	 =& \left (sg\gamma,  \tphi _{sg\gamma} ([\td \alpha(\gamma\inv)(x)]\lambda)     \exp(\psi^{s}(g,x))\exp( \psi^{u}(g,x))\right).
\end{align*}
%By Lemma \ref{claim:eqiv} we have
%$$
% \tphi _{sg} (x) \rho(sg)(\lambda)    \exp(\psi^{s}(g,x))\exp( \psi^{u}(g,x))) =
%\tphi _{sg\gamma} (\alpha(\gamma\inv) x\lambda)     \exp(\psi^{s}(g,x))\exp( \psi^{u}(g,x)))$$
It follows that
$$\rho_i(s) \tphi_{g\gamma}([\td \alpha(\gamma\inv)(x)]\new{\lambda}) = \tphi _{sg\gamma} ([\td \alpha(\gamma\inv)(x)]\new{\lambda}) \exp( \psi^{s}(g,x) )\exp(\psi^{u}(g,x)).\qedhere $$\end{proof}

\subsubsection{Subexponential growth of central defects} Fix any norm on $\lien$.
By the  invariance in Lemma \ref{claim:invar}, the maps $(g,x) \to \psi^{s/u}(g,x) $ descend to  maps on the suspension space  $M^\alpha\to E^{s/u}$.
In particular, as $M = \td M/\Lambda_M$ is compact, for every $g\in G$, the functions $\|\psi^{s/u}(g,x)\|$ are bounded in $x$ .
The main technical obstruction to building the conjugacy is that (as $G/\Gamma$ is not assumed   compact) the  functions $\| \psi^{s/u}(g,x)\|$ need not be bounded over $(g,x).$

Let $$C^{s/u}(g):= \max_{x\in \td M} \|\psi^{s/u}(g,x)\|.$$
For $\gamma\in \Gamma$, the above invariance gives  $C^{s/u}(g) = C^{s/u}(g\gamma).$
In particular, the functions $C^{s/u}(g)$ descend to functions on $G/\Gamma$.

%\red{bound in $\lien$ or in $N$?}
%As $\td M/\Lambda$ is compact, this max exists and is finite.
\begin{lemma}\label{claim:login} for $\sigma \in \{s,u\}$ we have
$$\int_{G/\Gamma} \log^+( C^\sigma(g\Gamma) ) \ d (g\Gamma)<\infty.$$
\end{lemma}
\begin{proof}
Recall our fundamental domain $D$.
We show  $\int_D\log^+( C^\sigma( g)) \ d g<\infty. $

Let $\psi(g,x) = \psi^{s}(g,x) +  \psi^{u}(g,x).$
For $g\in D$ we have $\gamma_g= e$ and
\begin{align*}
\exp(-\psi(g,x))
%	:&= \rho(s) (\tphi_g(x) )( \tphi _{sg} (x) )\inv\\
%	&=\rho(s)(\rho(g\gamma_g\inv) \td \phi( \alpha( \gamma_g)(x))( \rho(sg\gamma_{sg}\inv) \td \phi( \alpha( \gamma_{sg})x))\inv\\
%	&=\rho(sg)\left( \rho(\gamma_g\inv) (\td \phi( \alpha( \gamma_g)(x)))( \rho(\gamma_{sg}\inv) \td \phi( \alpha( \gamma_{sg})x))\inv\right)\\
%	&=\rho(sg)\left(\td \phi( x)( \rho(\gamma_{sg}\inv) \td \phi( \alpha( \gamma_{sg})x))\inv\right)
	:&= (\rho_i(s) (\tphi_g(x) ))\inv( \tphi _{sg} (x)) \\
	&=\rho_i(s)(\rho_i(g\gamma_g\inv) \td \phi( \alpha( \gamma_g)(x))\inv( \rho_i(sg\gamma_{sg}\inv) \td \phi( \td \alpha( \gamma_{sg})(x)))\\
	&=\rho_i(sg)\left( \rho_i(\gamma_g\inv) (\td \phi( \alpha( \gamma_g)(x))\inv)( \rho_i(\gamma_{sg}\inv) \td \phi( \td \alpha( \gamma_{sg})(x)))\right)\\
	&=\rho_i(sg)\left(\td \phi( x)\inv( \rho_i(\gamma_{sg}\inv) \td \phi( \td \alpha( \gamma_{sg})(x)))\right)
\end{align*}
Let $F=\{\gamma_\ell\}$ be a the finite set of generators for $\Gamma$ fixed above and write
$$\gamma_{sg} = \gamma_{\ell(1)}\gamma_{\ell(2)}\cdots \gamma_{\ell (n(g))}$$
where $n(g)$ is the word-length of $\gamma_{sg}$ relative to the generators $\{\gamma_\ell\}.$
%word length of $\gamma_{sg}$.  %\prod _{i = 1} ^{n(g)} \gamma_{\ell(i)}.$$
%\red{\st{ Recall that   $\td\phi(\td \alpha( \gamma_{\ell})(x)))$ is $P_{*,i}$-equivariantly homotopic to $\rho_i(\gamma_{\ell})(\td \phi(x))$.}}    %In particular, f
 From \eqref{eq4}, for each $x\in \td M$, we have 
$$\rho_i (\gamma_\ell)\inv \td\phi(\td \alpha( \gamma_{\ell})(x))) =\td \phi(x) z$$
for some $z=z_{\gamma_\ell}(x)\in Z_i$.   Moreover, the function $z_{\gamma_\ell}$ is $\Lambda_M$-invariant, hence  there is a uniform constant $D_\ell>0$ such that $$\|\exp\inv(z_{\gamma_\ell})\|\le D_\ell.$$ % $$\sup_{x\in \td M}\left \{\|\exp\inv\left(\rho_i (\gamma_\ell)\inv \td\phi(\td \alpha( \gamma_{\ell})(x)))\right)\| \right\}\le D_\ell.$$

Then for each $x$ we have a sequence $z_j\in Z_i$ for $1\le j\le n(g)$ with  $\|\exp\inv(z_j)\| \le D_{\ell(j)}$ and

\begin{align*}
\rho_i(\gamma_{sg}\inv) &\td \phi( \td \alpha( \gamma_{sg})(x)))\\
 :&=  %\rho(\gamma_{\ell(n(g))})\inv\rho(\gamma_{\ell(n(g)-1)})\inv\cdots \rho(\gamma_{\ell(1)})\inv\left(
% \td \phi(\alpha(\gamma_{\ell(n(g))})\cdots \alpha(\gamma_{\ell(2)})\alpha(\gamma_{\ell(1)})(x)))\right)\\
% &= \rho(\gamma_{\ell(n(g))})\inv\rho(\gamma_{\ell(n(g)-1)})\inv\cdots \rho(\gamma_{\ell(1)})\inv
% \left(
% \td \phi(\alpha(\gamma_{\ell(1)}) \alpha(\gamma_{\ell(2)})\cdots \alpha(\gamma_{\ell(n(g))})(x)))\right)\\
% &=
 \rho_i(\gamma_{\ell(n(g))})\inv\rho_i(\gamma_{\ell(n(g)-1)})\inv\cdots \rho_i(\gamma_{\ell(1)})\inv\left(
 \td \phi( \td \alpha(\gamma_{\ell(1)})\cdots \td \alpha(\gamma_{\ell(n(g))})(x))\right)\\
  &= \rho_i(\gamma_{\ell(n(g))})\inv\rho_i(\gamma_{\ell(n(g)-1)})\inv\cdots \rho_i(\gamma_{\ell(2)})\inv\left(
 \td \phi( \td \alpha(\gamma_{\ell(2)})\cdots \td \alpha(\gamma_{\ell(n(g))})(x))z_1\right)\\
  &= \rho_i(\gamma_{\ell(n(g))})\inv\rho_i(\gamma_{\ell(n(g)-1)})\inv\cdots \rho_i(\gamma_{\ell(3)})\inv \\&\quad \quad \left(
 \td \phi( \td \alpha(\gamma_{\ell(3)})\cdots \td \alpha(\gamma_{\ell(n(g))})(x))\rho_i(\gamma_{\ell(2)})\inv(z_1) z_2\right)\\
&\quad    \vdots\\
 &  =\td \phi(x) \prod_{j = 1} ^{n(g) } \rho_i(\gamma_{\ell(n(g))})\inv\rho_i(\gamma_{\ell(n(g)-1)})\inv\cdots \rho_i(\gamma_{\ell(j+1)})\inv(z_j).
\end{align*}

Let
\begin{itemize}
%\item $X_i(x) = \exp\inv(z_i)$;
\item $S_\ell= \restrict{D \rho_i(\gamma_\ell)}{T_eN_i}$;
\item $S = \restrict{D \rho_i(s)}{T_eN_i}$;
\item  $S_g = \restrict{D  \rho_i(g)}{T_eN_i}$;
\item $C = \max \|S_\ell\|$;
\item $D= \max D_\ell$.
\end{itemize}

Then, as
\begin{align*}
\exp(-\psi(g,x)) &=
\rho_i(sg)\left(\td \phi( x)\inv( \rho_i(\gamma_{sg}\inv) \td \phi( \td \alpha( \gamma_{sg})(x)))\right)\\
	&= \rho_i(sg)\left( \prod_{j = 1} ^{n(g) } \rho_i(\gamma_{\ell(n(g))})\inv\rho_i(\gamma_{\ell(n(g)-1)})\inv\cdots \rho_i(\gamma_{\ell(j+1)})\inv(z_j)\right)
\end{align*}
we have 	$$\|\psi(g,x) \| \le \|S\| \|S_g \| n(g) C^{n(g)}D.$$

%\red{need citation: should be comparable to $d(g\Gamma, \Gamma)$}

Note that (as $D\rho\colon G\to \Aut(\lien)$ is a continuous representation) there is a constant $C_1$ with
$$\log\|D\rho (g)\|\le  C_1d(g,e),$$  %\marginnote{I think the $+C_1$ here was extraneous}
%hence by \eqref{eq:FDQI} we have
%$$\int _D \log \|S_g\| \le C_1 + \int _D ad_G(g\Gamma, \Gamma) \ d g + b.$$
%\begin{mdframed}
hence  we have
$$\int _D \log \|S_g\| \le   C_1\int _D d_G(g\Gamma, \Gamma) \ d g \Gamma.$$
By Lemma \ref{Integrable}, $\int _D \log \|S_g\|  <\infty$.

Moreover,  from \eqref{eq:QI}
\begin{align}
\int _D n(g) \ d g& \le \int_D A d_G(e, \gamma_{sg}) + B\ d g \notag\\
			& \le \int_D A\left[ d_G(e, g) + d_G(g,sg) + d_G( sg, \gamma_{sg})\right] + B\ d g \notag\\
			& \le \int_D A\left[ d_G(e, g) + d_G(e,s) + d_G( sg \gamma_{sg}\inv, e)\right] + B\ d g\label{eq:bound}
\end{align}

From the choice of fundamental domain, we have
$$d_G(e, g)= d_G(\Gamma, g\Gamma)$$ and
 $$d_G( sg \gamma_{sg}\inv, e) =
d_G( sg \gamma_{sg}\inv\Gamma, \Gamma) \le d_G(g\Gamma,\Gamma) + d_G(sg\Gamma,g\Gamma)$$
hence
$$\int _D n(g) \ d g\le   \int_D 2 A\left[ d_G(g\Gamma,\Gamma) + d_G(sg\Gamma,g\Gamma) \right] + B\ d g $$
and
 it follows that again from Lemma \ref{Integrable} that $\int _D n(g) \ d g$
  is finite.  The claim then follows.
\end{proof}

%From \eqref{eq:FDQI} we have $$\int_D d_G(e, g) \ d g \le  \int _{G/\Gamma}\left(a d_G(\Gamma, g\Gamma) + b\right) \ d (g\Gamma)$$  {and} $$
%\int_D d_G(sg\gamma_{sg}, e) \ d g \le  \int_{G/\Gamma}\left( a d_G( sg \Gamma, \Gamma) + b\right) \ d (g\Gamma).$$
%As $g\Gamma\to sg\Gamma$ is measure preserving on $G/\Gamma$  it follows that again from \eqref{eq99} that   \eqref{eq:bound} is finite.  The claim then follows.
%\end{proof}

From Lemma \ref{claim:login} and  standard  tempering kernel arguments  (\cite[Lemma 3.5.7]{MR2348606})  we immediately obtain the following.
\begin{proposition}\label{prop:subexp}
For any $\epsilon>0$, there is a measurable, $\Gamma$-invariant function $L= L_\epsilon \colon G\to [0, \infty)$ so that for almost every $g\in G$ and every $k\in \Z$
\begin{enumerate}
\item $C^{s/u}(g) \le L(g)$;
\item $ L(s^k g) \le e^{\epsilon |k|} L(g)$.
\end{enumerate}
\end{proposition}

\section{Construction of the semiconjugacy}\label{sec:5}
In this section, we build a continuous semiconjugacy $H$ between the left $G$-actions on $G\times \td M$ and $G\times N_i$.  Moreover, the conjugacy will be $\Psi$-equivariant and hence descend to a semiconjugacy between left $G$-actions on $M^\alpha$ and $(N_i/\Lambda_i)_\rho$.

We first construct a measurable function $H$ intertwining the action of $s$.  We then extend $H$ to intertwine the actions of the centralizer of $s$ and finally all of $G$.  That $H$ agrees almost everywhere  with a  continuous function will follow  from construction and the fact that  $H$ intertwines the left $G$-actions.

\subsection{Semiconjugating the action of $s$}

\def\sterm{ \exp\left(\sum_{k=1} ^\infty S^{k-1} \Big( \psi ^s(s ^{-k } g, x)  \Big) \right)
}
\def\uterm{ \exp\left(- \sum_{k=0} ^\infty S^{-k-1}\Big(\psi ^u(s ^{k } g, x)  \Big)\right)
}
		\def
		\hdef{
%		\begin{align*}
		h_g( x): =
		 \tphi_g(x)
		 	\sterm	 \uterm
%				\end{align*}
		}
		Recall we write $S\in \Aut(\lien)$ for $D\rho_i(s)$ where $s$ is a distinguished element of $A$ fixed in \ref{sss:central}.  %\in \Aut(N)$.
Given $g\in G$ define a \emph{formal conjugacy} by adding (central) correction terms to the family $\tphi_g$:
\begin{equation}\label{eq:hdef}\hdef .\end{equation}
We check formally that
	$$\rho_i(s)(h_g(x)) = h_{sg}(x).$$
Indeed,
\begin{align*}
%s\cdot H(g,x) &= (sg,
&\rho_i(s)(h_g(x))\\
=& \rho_i(s)\left[ \tphi_g(x) \sterm\uterm\right]\\
%	&= \rho(s)\left( \tphi_g(x) \exp\left(-\sum_{k=0} ^\infty S^k \Big( \psi ^s(s ^{-k } g, x)  \Big) \right)\exp\left( \sum_{k=1} ^\infty S^{-k}\Big(\psi ^u(s ^{k } g, x)  \Big)\right)\right)\\
	=&\tphi _{sg} (x) \exp( \psi^{s}(g,x) )\exp(\psi^{u}(g,x))\\
	& \quad \quad \quad \cdot	\exp\left(\sum_{k=1} ^\infty S^{k} \Big( \psi ^s(s ^{-k } g, x)  \Big) \right)\exp\left( - \sum_{k=0} ^\infty S^{-k}\Big(\psi ^u(s ^{k } g, x)  \Big)\right)\\
	=&\tphi _{sg} (x) 	\exp\left(\sum_{k=0} ^\infty S^{k} \Big( \psi ^s(s ^{-k } g, x)  \Big) \right)\exp\left( - \sum_{k=1} ^\infty S^{-k}\Big(\psi ^u(s ^{k } g, x)  \Big)\right)\\
	=&\tphi _{sg} (x) 	\exp\left(\sum_{\ell =1} ^\infty S^{\ell-1} \Big( \psi ^s(s ^{-\ell } (s g), x)  \Big) \right)\exp\left( - \sum_{\ell=0} ^\infty S^{-\ell-1}\Big(\psi ^u(s ^{\ell} (sg), x)  \Big)\right)\\
	=& h_{sg}(x).
\end{align*}

We say a family of maps $g\to h_g\colon \td M\to N_i$ parametrized by $g\in G$ is \emph{$\Psi$-equivariant} if the map $G\times \td M\to G\times N_i$ defined by $(g, x)\to (g, h_g(x))$ is $\Psi$-equivariant.

\begin{lemma}
There is a full measure  set of $g\in G$ such that
$h_g\colon \td M\to N_i$ is well-defined, continuous, and $\Psi$-equivariant.  Moreover, for such $g$ $$\td p_{i,i+1} \circ h_g = \rho_{i+1}(g) \td h_{i+1}.$$
\end{lemma}
\begin{proof}
Taking $\epsilon< \|S\|/100$, the claim holds for all $g$ such that Proposition \ref{prop:subexp} holds.
\end{proof}

Define $H\colon G\times \td M \to G\times N_i$ by $H(g,x) = (g, h_g(x))$.  Then $H$ defines a measurable conjugacy between the action of $s$:  for almost every $g\in G$  $$H(s\cdot (g, x)) = s\cdot H(g,x).$$

\begin{lemma}\label{clm:unique}
The family of maps $g\to h_g$ is unique among the measurable family of  $\Psi$-equivariant, continuous functions $\td M\to N_i$ with the property that $ h_{sg}(x) = \rho_i(s) h_g(x)$ and $\td p_{i, i+1}\circ h_g = \rho_{i+1}(g) \td h_{i+1}$.
\end{lemma}
\begin{proof} Suppose $g\to \bar h_g$ is another such family.  % of measurable, $\Psi$-equivariant, continuous functions with the property that $\bar h_{sg}(x) = \rho(s) h_g(x)$ and $\td p_{i, i+1}(\bar h_g) = \rho(g) \td h_{i+1}$.
 As $\td p_{i,i+1} \circ \bar h_g =  \td p_{i,i+1} \circ  h_g $ it follows that $\bar h_g(x) = h_g(x) \exp(\hat \psi(g,x))$ for some $\hat \psi\colon G\times \td M\to \liez_i$.  Write $\hat \psi(g,x)  = \hat \psi^s(g,x)  +\hat \psi^u(g,x)  $

By the $\Psi$-equivariance of $h_g$ and $\bar h_g$ and as $\hat \psi^s(g,x) \in \liez_i$ it follows that $$\hat \psi^s((g,x)\cdot (\gamma, \lambda))=  \hat \psi^s(g,x).$$
Similarly, $\hat \psi^u((g,x)\cdot (\gamma, \lambda))=  \hat \psi^u(g,x).$  In particular, % $\hat \psi^{s/u}$ descends to a map on $M^\alpha$ and
 for almost every $g\in G$, the function $\hat \psi^{s/u}$ descends to a   continuous function from  $M$ to $\liez_i$.  It follow that $\|\hat \psi^{s/u}(g,x)\|$ is bounded uniformly in $x$ for almost every $g$.

As the families $h_g$ and $\bar h_g$ intertwine the dynamics we moreover have that
\begin{itemize}
\item  $\hat \psi^s(g,x)= S^k \hat \psi^s\left(s^{-k}\cdot (g,x)\right)$
\item $\hat \psi^u(g,x)= S^{-k} \hat \psi^u\left(s^{k}\cdot (g,x)\right)$
\end{itemize}
By Poincar\`e recurrence to sets on which $g\mapsto \max_{x\in \td M} \|\hat \psi^{s/u}(g,x)\|$ is uniformly bounded, it follows that $\hat \psi^s(g,x)= 0= \hat \psi^u(g,x)$ for almost every $g$ and every $x\in M$.    Hence $h_g= \bar h_g$ almost everywhere.
\end{proof}
\subsection{Extending  the semiconjugacy  to the centralizer of $s$}  
{\new{Recall $s\in A$ where $A\subset G$ is a maximal split Cartan subgroup.}}
Let $C_G(s)\subset G$ denote the centralizer of $s$ in $G$.  Note in particular that $A\subset C_G(s)$.  Moreover, every compact almost simple factor of  $G$ is contained in $C_G(s)$.

\begin{proposition}\label{CentralizerSemiconj}
Let $\bar s\in C_G(s)$.  Then for almost every $g\in G$
$$  h_{\bar sg} = \rho_i(\bar s) h_g.  $$ %,\quad \text{ hence } \quad  H(\bar sg, x) =\bar s\cdot H(g,x).$$
\end{proposition}
\begin{proof}
Let $\bar s\in C_G(s)$.  For $g\in G$ define $\bar h_g:=\rho_i(\bar s\inv)h_{\bar s g}$.  We check  that $g\to \bar h_g$ defines a measurable  family of  $\Psi$-equivariant, continuous functions $\td M\to N_i$ with the property that $\td p_{i, i+1}(\bar h_g) = \rho_i(g) \td h_{i+1}$.  Moreover
\begin{align*}\bar h_{sg}(x) :&= \rho_i(\bar s\inv)h_{\bar s s g} =
\rho_i(\bar s\inv)h_{s \bar s  g} =
\rho_i(\bar s\inv)\rho_i(s)h_{ \bar s  g} \\&=
\rho_i(s) \rho_i(\bar s\inv)h_{ \bar s  g} =
 \rho_i(s) \bar h_g(x).\end{align*}
By Lemma \ref{clm:unique}, $\bar h_g = h_g$ for almost every $g$.
\end{proof}
It follows from standard ergodic theoretic constructions that there is a full measures subset $X_0$ subset $G$ such that---after  modifying the family $g\to h_g$ on a set of measure zero---we have $$\rho_i(s')h_g = h_{s'}(g) $$ for all $g\in X_0$ and all $s'\in C_G(s)$.
In particular, the map $H\colon G\times \td  M\to G\times N_i$ defines a measurable conjugacy between the action of $C_G(s)$.

\subsection{Extension of the semiconjugacy to $G$}
We show that $H$ intertwines the full $G$-actions on $G\times \td  M$ and $G\times N_i$ via the following proposition.
\begin{proposition}\label{UnipotentSemiconj}
Let $\zeta $ be a restricted root of $\lieg$ that is  non-resonant with the representation $D\rho_i$.  Let $X\in \lieg^\zeta$ and let $a= \exp(X).$   Then  for almost every $g\in G$
$$ h_{ag} = \rho_i(a) h_g.$$
%Hence $H$ defines a measurable conjugacy between the action of $a$: $$H(ag, x) = aH(g,x).$$
\end{proposition}
\begin{proof}
As $\zeta$ is not positively proportional to any weight of $\rho_i$, we may find $s_1, s_2\in A$ and a splitting $\liez_i = E\times F$ so that writing $S_j = D\rho(s_j),$
\begin{enumerate}
\item $\zeta(s_1) = \zeta(s_2) = 0$;
\item $\|\restrict{S_1}{E}\|<1$;
\item $\|\restrict{S_2}{F}\|<1$.
%\item $S_2$ contracts $F$.
\end{enumerate}
Note then that $s_i\inv a s_i = a$ for $i\in \{1,2\}.$

Recall that for almost every $g$ and any $x\in \td M$ we have $$\td p_{i,i+1} \circ \rho_i(a) h_g(x) = \td p_{i, i+1} h_{ag}(x).$$
Thus, given almost every $g\in G$ and any $x\in\td  M$ there is a unique $\eta(g,x)\in \liez_i$ with $$\rho_i(a) h_g(x) = h_{ag}(x) \exp (\eta(g,x)).$$
For almost every $g\in G$,  $h_g$ is continuous and descends to a  function  defined on $M$.  It follows that $\|\eta(g,x)\|$ is bounded uniformly in $x$ for almost every $g$.
As the family $h_g$ is $\Psi$-equivariant, we have  $\eta(g,x)= \eta(g\gamma, \td \alpha(\gamma\inv)(x))$ and hence the function $g\to \max_{x\in \td M} \|\eta(g,x)\|$ is $\Gamma$-invariant.  %descends to a function on $(G\times \td M)/\Gamma$
Fix $C>0$ and let $B\subset X_0\subset G/\Gamma$ be such that for $g\in B$, $\|\eta(g,x)\|<C$ for all $x$.  Taking $C$ sufficiently large we may ensure $B$ has measure arbitrarily close to $1$.

Now, consider $g\in B$ such that $s_j^{-k}g\Gamma\in B$ for $j=\{1,2\}$ and infinitely many $k\in \bN$.  % under the actions of $s_1$ and $s_2$.
Note that as
\begin{align*}
 \rho_i(a)  h_{s_jg}(x) &=
 \rho_i(a) \rho_i(s_j) h_g(x) \\
&=
\rho_i(s_j)\left( \rho_i(a) h_g(x) \right)\\
&= \rho_i(s_j) \left(h_{ag}(x) \exp (\eta(g,x))\right)\\
&=h_{s_jag}(x) \exp (S_j\eta(g,x))  \\
&=h_{as_jg}(x) \exp (S_j\eta(g,x))
\end{align*}
we have $\eta(s_jg, x) = S_j\eta(g,x)$.
Write $\eta(g,x) = \eta^E(g,x) + \eta^F(g,x)$.
We then have
$$\eta^E(g,x)= S_1^{k}(\eta^E(s_1^{-k}g,x)),\quad \quad \eta^F(g,x)= S_2^{k}(\eta^F(s_2^{-k}g,x)).$$
It then follows that $$\eta^E(g,x)= \eta^F(g,x)= 0$$
proving the proposition.
\end{proof}

Recall that we  assume that the representation $D\rho\colon G\to \Aut(\lien)$ is weakly non-resonant with the adjoint representation.
In particular, every $g\in G$ can be written as  $$g= \exp(X_1) \exp (X_2)\cdots \exp(X_\ell)$$ where each $X_j$ is either a vector in $\lieg^0$, or a vector in  $\lieg^\xi$ for some restricted root $\xi$ that is non-resonant with  the representation $D\rho_i$. Recall that for $X_j\in\lieg^0$, $\exp(X_j)$ lies in $C_G(s)$.

It follows from Proposition \ref{CentralizerSemiconj} and \ref{UnipotentSemiconj} that, after modifying $H$ on a set of measure zero, we have  $$a\cdot H(g,x) = H(a\cdot(g,x))$$ for almost every $g\in G$, every $x\in \td M$ and every $a\in G$.  As $G$ acts transitively on itself
$$a\cdot H(g,x) = H(a\cdot (g,x))$$ holds for  {\it every} $g\in G$, every $x\in \td M$ and every $a\in G$.
Now, fix $g$ so that $ h_g\colon \td M\to N_i$ is continuous.  As  $ h_{ag} = \rho(a)  h_g$, it follows that $ h_g$ is continuous for every $g\in G$ and, moreover, the family $g\to h_g$ various continuously in the parameter $g$ whence $H\colon G\times \td M\to G\times N_i$ is continuous.
%and, moreover, the function $\td H\colon G\times \td M\to G\times N_i$ is continuous.
%In particular,

Finally, recall that  the family $ h_g$ is $\Psi$-equivariant (whence the map $H\colon G\times \td M\to G\times N_i$ is $\Psi$-equivariant) and $\td p_{i, i+1} \circ  h_g = \rho_{i+1}(g) \td h_{i+1}$.  Indeed, these  properties hold for almost every $g$ and extend to every $g$ by  continuity.
In particular, this shows
\begin{corollary}\label{cor:conj}
There is a continuous, $\Psi$-equivariant function $H\colon G\times \td M\to G\times N_i$ of the form $H(g,x)= (g,h_g(x))$ with $a\cdot H(g,x)= H(a \cdot (g,x))$ for any $a\in G$.
\end{corollary}

%
%Write $$H(g,x) = (g, \td h_g(x)).$$

To complete the proof of Theorem \ref{thm:inductive} define $\td h_i\colon  \td M\to N_i$  by $$\td h_{i}:=  h_e.$$
Then, $\td h_i$ satisfies
\begin{enumerate}
	\item $\td h_i(x\lambda)= \td h_i(x) P_{*,i}(\lambda)$ for $\lambda\in \Lambda_M$;
	\item $\td p_{i, i+1}\circ  \td h_i = \td h_{i+1}$;
	\item $\td h_i(\alpha (\gamma) (x)) =    h_e(\alpha (\gamma) (x)) =   h_\gamma(x) = \rho(\gamma)  h_e(x) = \rho(\gamma)\td h_i(x)$ for all $\gamma\in \Gamma$.
\end{enumerate}
Moreover, writing $\td h_i = \rho(g\inv ) h_g$ for some $g\in G$ such that $h_g$ coincides with the family defined by \eqref{eq:hdef}, it follows that $\td h_i$ is $\Lambda_M$-equivariantly homotopic to $\td \phi$.
Consequently, $\td h_i\colon \td  M\to N_i$ descends to a function $h_i \colon M \to N_i/\Lambda_i$ satisfying the conclusions of the theorem.

\new{Moreover, if $h \colon M \to N_i/\Lambda_i$ is any continuous function having a lift $\td h\colon \td M\to N_i $ as in the conclusion of the theorem, it follows that $\td h(x\lambda) = \td h(x) P_{*,i}(\lambda)$ and $\td p_{i,i+1} \td h = \td h_{i+1}$.  That $h = h_{i}$ then follows from the  uniqueness given in Lemma  \ref{clm:unique}.  }
\def\Gl{\GL}

\section{Superrigidity, arithmeticity, and orbit closures}\label{sec:bigfacts}
In this section we collect a number of classical facts that will be used in the sequel to prove Theorems \ref{ConjThm} and \ref{SmoothThm}.

 Theorem \ref{ConjThm} follows from verification of  the hypotheses of Theorem \ref{thm:mainnilmanifoldfull}.
To show (1) of Theorem \ref{thm:mainnilmanifoldfull}, we use the  superrigidity theorem of Margulis.  % which we present in Section \ref{sec:arithSR}.
Note that for a lattice  $\Gamma$ as in Hypothesis \ref{HigherRank} and a  linear representation  $\psi$ of $\Gamma$, the standard superridigity theorem of  Margulis (\cite[Therorem IX.6.16]{MR1090825}, see also \cite[16.1.4]{MR3307755}) guarantees (if $G$ is a simply connected real semisimple algebraic group or if the Zariski closure of $\rho(\Gamma)$ is center-free) that the linear representation $\rho\colon \Gamma \to \GL(d, \R)$ extends on a finite-index subgroup to a continuous representation $\rho\colon G\to \Gl(d,\R)$  up to a bounded  correction.  We use the arithmeticity theorem of Margulis and the arithmetic lattice version of superrigidity below to  ignore the compact correction by replacing $G$ with a compact extension.

The  proof of  Theorem  \ref{SmoothThm} will require two additional facts that we also present here: the  superrigidity theorem for cocycles due to Zimmer and the classification of orbit closures for the action of the linear data $\rho$ associated to an action $\alpha$ which follows from the orbit classification theorem of Ratner.

\subsection{Arithmeticity and superrigidity} \label{sec:arithSR}
Let $G$ be a connected, semisimple Lie group with finite center and let $\Gamma\subset G$ be a lattice.   For this section, we make the following standing assumption:
\begin{equation} \label{hyp} \text{$\Gamma$ has  dense image in every $\R$-rank 1, almost-simple factor of $G$}.\end{equation}  In particular, \eqref{hyp} holds under Hypothesis \ref{HigherRank}.

Let $G'$ denote the quotient of $G$ by the maximal compact normal subgroup and let $\Gamma'$ be the image of $\Gamma$ in $G'$. Then the projection to $\Gamma'$ has finite kernel in $\Gamma $. Moreover, $(G',\Gamma')$ still satisfies \eqref{hyp}.

Let $\lieg'$ denote the Lie algebra of $G'$.  Note that $G'$ acts on $\lieg'$ via the adjoint representation.  Let $G^*=\Ad G'\subset \GL(\lieg')$ denote the image of $G'$.  %Note that $\lieg'$ has a $\Q$-structure relative to which $G^*$ is an algebraic group defined over $\Q$.  %Moreover $G^*$ is semisimple and  has no compact factors and trivial center.
Let $\Gamma^*$ denote the image of $\Gamma'$ in $G^*$.    $\Gamma^*$ is a lattice in $G^*$.

Let $\psi\colon \Gamma \to \GL(d,\Q)$ be a linear representation.  As $ \psi(\Gamma)$ is a finitely generated subgroup of $\GL(d,\R)$, it  contains   finite-index, torsion-free, normal subgroup %  $\Delta\subset \psi(\Gamma)$ with no torsion elements
(\cite[Theorem 6.11]{MR0507234}).
Restricting to a finite-index subgroup of  $\Gamma$, we may assume $\psi(\Gamma)$ has no torsion.   On the other hand, note that the kernel of $\Gamma \to \Gamma^*$ is finite.  It then follows that $\psi(\gamma)$ is torsion for every $\gamma$ in the kernel of $\Gamma \to \Gamma^*$.  Thus we may assume
 the representation $\psi\colon   \Gamma \to \GL(d,\Q)$ factors through a representation $ \psi^*\colon   \Gamma^* \to \GL(d,\Q)$.

Note that $G^*$ is a semisimple Lie group without compact factors and with trivial center. In this case $G^*=\bfG^*(\bR)^\circ$ for a semisimple algebraic group $\bfG^*$.  Moreover, $\Gamma^*\subset G^*$ is a lattice whose projection to every rank-1 factor is dense.

\begin{theorem}[Margulis Arithmeticity Theorem \cite{MR1090825}]\label{Arithmeticity} Let $\bfG^*$ be a semisimple algebraic Lie group defined over $\bR$, and $\Gamma^*$  a lattice in $G^*=\bfG^*(\bR)^\circ$ that satisfies hypothesis \eqref{hyp}. Then $\Gamma^*$ is arithmetic: there exist a connected semisimple algebraic group $\bfH$ defined over \new{$\bQ$} and a surjective algebraic morphism $\phi\colon \bfH\to\bfG^*$ defined over $\bR$, such that
\begin{enumerate}
 \item $\phi$ is a quotient morphism between algebraic groups, and the kernel of the surjective morphism $\phi:\bfH(\bR)^\circ\to G^*$ is compact with compact kernel;
 \item $\phi(\bfH(\bZ)\cap \bfH(\bR)^\circ)$ is commensurable with $\Gamma^*$.
\end{enumerate}
\end{theorem}

The case  where $\Gamma^*$ is irreducible follows from \cite[Introduction, Theorem 1']{MR1090825}, where irreducible lattices is defined as in  \cite[pg.\ 133]{MR1090825}. % \cite{MR1090825}*{III.5.9}. 
 In general, $\bfG$ decomposes as an almost direct product $\prod_i\bfG_i^*$ of normal subgroups defined over $\bR$, and there are irreducible lattices $\Gamma_i^*<G_i^*=\bfG_i^*(\bR)^\circ$  such that $\prod\Gamma_i$ has finite index in $\Gamma$.  Then each pair $(G_i,\Gamma_i)$ satisfies \eqref{hyp}. This reduces to the irreducible case.

 By passing to a finite cover we  may  assume the $\Q$-group $\bfH$ in Theorem \ref{Arithmeticity} is simply connected.  Let $ H= \bfH(\bR)^\circ $.  Let $\Delta= \phi\inv(\Gamma^*) \cap \bfH(\bZ)$.  Then $\Delta$ is commensurable with $\bfH(\bZ)\cap H$ and hence $\bfH(\bZ)$.  The map $ \phi$ induces a linear representation $\bar \psi\colon \Delta\to \Gl(d,\Q)$ given by $\bar \psi = \psi\circ \phi$.

As we assume  $\bfH$ is simply connected, $\bfH$ decomposes uniquely %\marginnote{\cite[Prop 1.4.10]{MR1090825}}
as the direct product of simply connected almost $\Q$-simple, $\Q$-groups $\bfH = \prod \bfH_i$. Moreover, since $\phi$ has compact kernel and $G^*$ has no compact simple factors, our assumption \eqref{hyp} implies every $\Q$-simple factor $\bfH_i$ has $\R$-rank at least-2.

We have the following version of Margulis Superrigidity.
%\marginnote{Took out the $\Q$-representation... I don't think it was needed}
\begin{theorem}[Margulis Superrigidity Theorem; arithmetic lattice case \cite{MR1090825}]\label{Superrigidity} Let $\bfH$ be a simply connected semisimple algebraic group defined over $\bQ$ all of whose almost $\bQ$-simple factors have $\R$-rank $2$ or higher,  and let $\Delta\subset\bfH(\bR)$ be a subgroup commensurable to $\bfH(\bZ)$. Suppose $k$ is a field of characteristic $0$, $\bfJ$ is an algebraic group defined over $k$, and $\psi\colon \Delta\to\bfJ(k)$ is a group morphism. Then there is a $k$-rational  morphism  $\psi'\colon \bfH\to\bfJ$ and a finite-index  normal subgroup $\Delta'\lhd\Delta$  such that $\psi(\Delta)=\psi'(\Delta)$ for all $\Delta\in\Delta'$.
\end{theorem}

When $\bfH$ is almost $\bQ$-simple, this is a direct consequence of \cite[VIII, Theorem B]{MR1090825}. See also \cite[Corollary 16.4.1]{MR3307755}.  In general,  $\bfH$ is the direct product of its almost $\bQ$-simple factors $\bfH = \prod \bfH_i$ and $\prod \Delta_i$ has finite index in $\Delta$ where each  $\Delta_i$ is defined by $\Delta_i=\Delta\cap\bfH_i(\bR)$ and is commensurable with $\bfH_i(\bZ)$. On each $\Delta_i$, $\psi$ coincides with a $\bQ$-representation $\psi'_i$ of $\bfH_i$ on a finite-index subgroup.  Via projection to $\Q$-simple factors, the  $\psi_i'$ % form commuting representations of $\bf H$ which
assemble into a coherent representation $\psi'$; moreover $\psi'$ coincides with $\psi$ on a finite-index subgroup  of $\Delta$.

Replace $\Delta$ with $\Delta'$ coming from  Theorem \ref{Superrigidity}.
Recall we have   projections $p_1\colon G\to G^*$ and $p_2\colon H\to G^*$.
Consider the Lie group $$\bar L:= \{ (g,h)\in G\times H: p_1(g) = p_2(h)\}.$$
Let $L$ be the identity component  of $\bar L$ and let $\hat \Gamma \subset L$ be
 $$\hat \Gamma := \{ (\gamma ,\delta)\in \Gamma \times \Delta : p_1(\gamma) = p_2(\delta)\}\cap L.$$

We have natural maps $L\to G$ and $L\to H$ given by coordinate projections.  Because $G\to G^*$ and $H\to G^*$ are surjective, so are $L\to G$ and $L\to H$.  Moreover, the kernel of $L\to G$ is the set $\{(e, h) : p_2 (h) = e\}$.  As the kernel of $p_2$ is compact,   $L$ is a compact extension of $G$.  Moreover the image of $\hat \Gamma$ in $G$ has finite index in $\Gamma$.  Thus $\hat \Gamma$ is a lattice in $L$.  Restricting to a finite-index subgroup, we may assume $\hat\Gamma$ maps into $\Gamma$ and $\Delta$.

Summarizing the above we have the following.
\begin{proposition}\label{prop:ourSR}
For $G$ a connected, semisimple Lie group with finite center, $\Gamma\subset G$ a lattice satisfying \eqref{hyp},  an algebraic group $\bfJ$ defined over a field $k$ with $\mathrm{char}(k)=0$, and any group morphism $\psi\colon \Gamma \to \bfJ(k)$,  there are
\begin{enumerate}
\item semisimple Lie groups $G^*$, $H$ and $L$, such that $G^*$ has trivial center, and $H=\bfH(\bR)^\circ$ for some simply connected semisimple algebraic group $\bfH$ defined over   \new{$\bQ$;}
\item a finite-index subgroup $\bar \Gamma\subset \Gamma$;
%\item a linear representation $\psi' \colon H\to \Gl(d,\R)$;
\item lattices $\hat \Gamma\subset L, \Delta\subset H,$ and $ \Gamma^*\subset G^*$;
\item  surjective homomorphisms $\pi_1\colon L\to G, \pi_2\colon L\to H,
p_1\colon G\to G^*,$ and $p_2\colon H\to G^*$, all of which have compact kernels;
\item a   representation $\psi^* \colon \Gamma^*\to \bfJ(k)$ such that $\restrict{\psi}{\bar\Gamma} = \psi^*\circ p_1$;
\item a $k$-rational morphism $\psi' \colon \bfH \to\bfJ$,
\end{enumerate}
such that the following diagrams commute:
%$$\begin{CD}
%L @>\pi_2>> H@>\psi'>>\Gl(d,\R)\\ @V\pi_1VV @VVp_2V @.\\ G @>>p_1>G^*@.  \end{CD}
%\quad\quad\quad\quad\quad\
%\begin{CD}
%\hat \Gamma @>\pi_2>> \Delta @ .  \\ @V\pi_1VV @VVp_2V @.\\ \bar \Gamma @>>p_1>\Gamma^*@>>\psi^*>\Gl(d,\R)  \end{CD}
%$$

%\marginpar{figure temporarily hidden to avoid latex package conflict on my computer -ZW}
%\hide
{
$$\xymatrix{
L \ar[d]_{\pi_1} \ar[r]^{\pi_2} &H\ar[d]^{p_2} \ar[r]^-{\psi'} &\bfJ\\
G \ar[r]_{p_1}          &G^* & }
\quad\quad\quad
\xymatrix{
\hat \Gamma \ar[d]_{\pi_1} \ar[r]^{\pi_2} &\Delta\ar[d]^{p_2} \ar[r]^-{\psi'} &\bfJ(k) \\
\bar\Gamma \ar[r]_{p_1}          &\Gamma^* \ar[ru]_{\psi^*}& }
$$
}

%%%$$\xymatrix{
%%%\hat \Gamma \ar[d]_{\pi_1} \ar[r]^{\pi_2} &\Delta\ar[d]^{p_2} \ar[rd]^-{\psi'} & \\
%%%\bar\Gamma \ar[r]^{p_1} \ar@/_1pc/[rr]_-{\psi}         &\Gamma^* \ar[r]^-{\psi^*}&   \Gl(d,\R)} $$
\end{proposition}

\subsection{Cocycle Superrigidity}
The superrigidity theorem (with bounded error) for representations admits a generalization due to Zimmer for linear cocycles over measure preserving actions of higher-rank groups.  We state here a version that will be sufficient for our later purposes.
%
%To prove (\ref{zariskidensesemigroup}) of Proposition \ref{semigroup}, we follow an approach developed in   \cite{KLZ}.  This relies on the following theorem.
\begin{theorem}[Zimmer Cocycle Superrigidity Theorem \cite{MR776417,MR2039990}]\label{CocycleSuperrigidity}
Let  $\bfG$ be a simply connected semisimple algebraic group defined over $\R$ all of whose almost-simple factors %{\sout{\red {are anisotropic and}}}
 have $\R$-rank $2$ or higher.  Let $G= \bfG(\R)^\circ$, and let  $\Gamma\subset G$ be  a lattice.
%Let $G$ and $\Gamma$ be as in Hypothesis \ref{hyp7}.
Let $(X,\mu)$ be a standard probability space and  let $\alpha\colon \Gamma\to \Aut( X,\mu)$ be an ergodic action by measure-preserving transformations.  Let  $\psi\colon \Gamma\times X\to\GL(d,\bR)$ be a measurable cocycle over $\alpha$ such that for any $\gamma$, $\log^+\|\psi(\gamma, x)\|$ is in $L^1(X,\mu)$. Then there exist a representation $\psi'\colon G\to\GL(d,\bR)$, a cocycle $\beta\colon \Gamma\times X\to\SO(d)$, and a measurable map $\theta\colon   X\to\GL(d,\bR)$  such that:\begin{enumerate}                                                                                                                                                                                                                                                         \item $\psi(\gamma,x)=\theta(\alpha(\gamma)(x))\psi'(\gamma)\beta(\gamma,x)(\theta(x))^{-1}$ for every $\gamma$ and $\mu$-almost every $x$;                                                                                                                                                                                                                                                          \item for all $g\in G$, $\gamma\in\Gamma$ and $\mu$-almost every $x$, $\psi'(g)$ and $\beta(\gamma,x)$ commute.                                                                                                                                                                                                                                             \end{enumerate}
\end{theorem}

The version of Zimmer Cocycle Superrigidity given in Theorem \ref{CocycleSuperrigidity}  was   proved by Fisher and Margulis in \cite[Theorem 1.4]{MR2039990}.  Note that in the proof of  Theorem  \ref{SmoothThm} below, we do not assume the action $\alpha$ preserves a measure.   However because in this setting   the semiconjugacy $h$ between $\alpha$ and the linear data $\rho$ is a conjugacy, we are able to induce $\alpha$-invariant measures from  $\rho$-invariant measures.

\subsection{Orbit closures for linear data}
We present here a fact that will be important in Section \ref{Sec:llppookjjjd}.  We assume that $G$ and $ \Gamma$ are as in Hypothesis \ref{HigherRank}, $M= N/\Lambda$ is a compact nilmanifold, $\alpha\colon \Gamma\to \Homeo(M)$ is an action that lifts to an action on $N$, and   $\rho\colon \Gamma\to \Aut(N/\Lambda)$ is the linear data.
We recall all notation from Proposition \ref{prop:ourSR}.  In particular, there is a Lie group $L$, a lattice $\hat \Gamma\subset L$,  a surjective homomorphism $\pi\colon L\to G$ with  $\pi(\hat \Gamma)\subset  \Gamma$  with finite index,
and a continuous representation $\psi\colon L\to \Aut(\lien)$ with  $D\rho(\pi(\hat \gamma) ) = \psi (\hat \gamma)$ for all $\hat \gamma \in \Gamma$.  Then $\psi$ extends to a representation $\hat \rho\colon L\to \Aut (N)$ such that $\hat \rho(\hat \gamma) = \rho(\pi(\hat \gamma))$ for $\hat \gamma \in \hat \Gamma$.

Let $M^{\hat \rho}$ denote the suspension space  $(L\ltimes_{\hat \rho} N)/(\hat \Gamma\ltimes_{\hat \rho}\Lambda)$ discussed in Remark \ref{rem:homospace} of  in Section \ref{sec:suspspace}.   As remarked there, $\hat\Gamma\ltimes_{\hat \rho}\Lambda$ is a lattice in $L\ltimes_{\hat \rho} N$.    Let $L'\subset L$ be the connected subgroup generated by all non-compact, almost-simple factors.
% As discussed below the $L'$-orbit of the identity coset in $L/\hat \Gamma$ is homogeneous whence, replacing $\hat \Gamma$ with a finite-index subgroup, we may assume the $L'$ orbit of any point is dense in $L/\hat \Gamma$.
Note that $L'$ is generated by unipotent elements of $L$.   Moreover, if $u\in L$ is a unipotent element of $L$ then $(u,e)$ is unipotent in $(L\ltimes_{\hat \rho} N)$.    It follows that the natural embedding $L'\subset (L\ltimes_{\hat \rho} N)$ is generated by unipotent elements. 

 From the orbit classification theorem of  Ratner \cite[Theorem 11]{MR1403920}, it follows that the $L'$-orbit closure of any point in $M^{\hat \rho}$ is a homogeneous submanifold.  As $L$ is a compact extension of $L'$, it similarly follows that the $L$-orbit closure of any point in $M^{\hat \rho}$ is a homogeneous submanifold.
Note that the closures of $L$-orbits on $M^{\hat \rho}$ are in one-to-one correspondence with closures of $\hat \rho(\hat \Gamma)$-orbits on $N/\Lambda$.   It follows that all $\hat \rho(\hat \Gamma)$-orbit closures on $N/\Lambda$ are $\hat \rho(\hat \Gamma)$-invariant, homogeneous submanifolds.  Note that the Haar measure on every  $\hat \rho(\hat \Gamma)$-invariant, homogeneous submanifold is $\hat \rho(\hat \Gamma)$-invariant.
Moreover, if $\hat\rho(\hat \gamma_0)$ is hyperbolic for some $\hat \gamma_0$ then each of these measures are ergodic.

Since $\pi(\hat \Gamma)$ is of finite index in $\Gamma$, we have the following.
 \begin{proposition}\label{prop:Ratner}
 For $\Gamma$ and $\rho$ as above and any $x\in N/\Lambda$ the orbit closure $\overline{\rho(\Gamma)(x)}$ is the finite union of homogeneous sub-nilmanifolds of the same dimension.  In particular, every orbit closure $\overline{\rho(\Gamma)(x)}$ coincides with the support of a $\rho$-invariant probability measure $\mu_x$ on $N/\Lambda$.
 Moreover, if $\rho(\gamma_0)$ is hyperbolic for some $\gamma_0$, then the measures $\mu_x$ are ergodic.
 \end{proposition}

\section{Topological rigidity for actions with hyperbolic linear data}\label{sec:6}
\def\Gl{\GL}
In this section we prove Theorem \ref{ConjThm} by verifying the hypotheses of Theorem \ref{thm:mainnilmanifoldfull}.  Note that the H\"older continuity of $h$ in Theorem \ref{ConjThm}  follows from standard arguments once the action $\alpha$ is by $C^1$ diffeomorphisms  and the linear data is hyperbolic (\cite[Section 19.1]{MR1326374}).

\subsection{Verification of (1) of Theorem \ref{thm:mainnilmanifoldfull}}\label{sec:verify1}
Let $G$ and $\Gamma$ be as in Hypotheses \ref{HigherRank}.  Given a nilmanifold $M= N/\Lambda$, let $\alpha\colon \Gamma \to  \Homeo(M)$ be an action.  We assume $\alpha$ lifts to an action $\td \alpha \colon \Gamma\to \Homeo(N)$ and let $\rho\colon \Gamma \to\Aut(M) \subset \Aut (N)$ be the induced linear data.
%Recall that  $\rho$ preserves the lattice $\Lambda$ and   $\exp\inv(\Lambda)$ generates a $\Q$-structure on $\lien$.
Identifying $\lien$ with $\bR^d$, the derivative of $\rho$ induces a linear representation $\psi= D\rho\colon \Gamma \to\Aut(\lien)\subset  \GL(d,\R)$.  We remark that $\Aut(\lien)$ is a real algebraic group.

%Via the exponential map $\exp\colon \lien\to N$, we view $\rho\colon \Gamma \to \GL(d,\bR)$, where $\lien$ is identified with $\bR^d$.  % \marginnote{check that his makes sense; note $\exp^{-1}(\Lambda)$ is {\bf not} a lattice} and the lattice in $\R^d$ generated by the set $\exp^{-1}(\Lambda)$ is identified with $\bZ^d$ (see Section \ref{sec:StructureNil}). Under

Let $\bar\Gamma, L, \hat \Gamma, \Delta, \pi_1, \pi_2, \psi', \psi^*$ be as in Proposition \ref{prop:ourSR}, with $k=\R$ and $\bfJ=\Aut(\lien)$.  As $\pi_1(\hat \Gamma)\subset \bar \Gamma$ we have an induced action $\hat \alpha \colon \hat \Gamma \to \Homeo(M)$, by $\hat \alpha(\hat \gamma) = \alpha(\pi_1(\hat \gamma))$.  The action $\hat \alpha$ lifts to an action on $N$ and the induced linear data is $\hat \psi(\hat \gamma) = \psi^*(p_1(\pi_1(\hat \gamma)))= \psi'(\pi_2(\hat \gamma))$.  It follows that the linear data $\hat \psi$ of $\hat \alpha$ extends to a continuous representation $\hat \psi'\colon L\to \Aut(\lien)$ given by $\hat \psi'(\ell) = \psi'(\pi_2(\ell))$.
 Via the exponential map, we have $ \hat \rho\colon \hat \Gamma \to \Aut (N) $ extends to
$ \hat \rho\colon L \to \Aut (N) $ by $\hat \rho(\ell) = \exp(\hat \psi'(\ell)).$
 Thus, after replacing $(G,\Gamma, \alpha, \rho)$ with $(L, \hat \Gamma, \hat \alpha, \hat \rho)$,  (1) of Theorem \ref{thm:mainnilmanifoldfull} holds.
%
%
%Recall that $M= N/\Lambda$ and the linear data $\rho$ defines a group morphism $\Gamma\to\Aut(M)\subset\Aut(N)$.
%
% Moreover, $\Aut(M)$ can be identified with a subgroup in $\GL(d,\bZ)$, where $\lien$ is identified with $\bR^d$ \marginnote{check that his makes sense; note $\exp^{-1}(\Lambda)$ is {\bf not} a lattice} and the lattice in $\R^d$ generated by the set $\exp^{-1}(\Lambda)$ is identified with $\bZ^d$ (see Section \ref{sec:StructureNil}). Under this identification, $D\rho$ takes values in $ \GL(d,\bZ)$. Therefore, Corollary \ref{ArithSuper} applies with $\psi=D\rho$. Let $\bfG$, $\Gamma'$, $\psi'$ and $\phi$ be as in Corollary \ref{ArithSuper}.
%
%$\Gamma'$ acts on $M$ through $\Gamma$ by $\rho'=\rho\circ\phi$. As $\phi(\Gamma')$ is of finite index in $\Gamma$, we see that in order to show Theorem \ref{ConjThm}, it suffices to prove:
%
%
%

\subsection{All weights are non-trivial} \label{sec:allweight}\label{sec:HypRep}

%\subsection{Hyperbolic representations}
%We verify, in the case $G$ is a real algebraic group and $\rho$ is a rational morphism with $\rho(g)$ hyperbolic for some $g$, that all weights of the representation are nontrivial.

Since, in the previous section, the linear data $\hat \rho\colon \hat \Gamma \to \Aut(N)$ extends to a continuous $\hat \rho'\colon L\to \Aut(N)$ and $\hat \rho'$ factors through a $\rho'\colon H\to \Aut(N)$, to establish (4) of Theorem \ref{thm:mainnilmanifoldfull} for $(L,\hat\rho)$, it is sufficient to show non-triviality of the weights of $D\rho'=\psi'$.
% show that the hypothesis of Lemma \ref{NonResonance} is satisfied in the case that

As $H= \bfH(\R)^\circ$ is real algebraic, the Lie subalgebra $\liea$ is the lie algebra of a maximal $\bR$-split torus, so the result follows from the following basic fact.

\begin{lemma}\label{CartanHyp} Suppose $G=\bfH(\bR)^\circ$ for a semisimple algebraic group $\bfH$ defined over $\bR$ and $\psi:\bfH\to\mathbf{GL}(d)$ is a $\R$-rational representation. Let $\bfA$ be a maximal $\bR$-split torus in $\bfH$. Suppose $\psi(g)$ is hyperbolic for some $g\in G$.  %Then there is an element $a\in\bfA$ such that $\chi(a)\neq 1$ for all restricted weights $\chi$ of $\psi$ with respect to $\bfA$.
Then all restricted weights of $\chi$ of $\psi$ with respect to $\bfA$ are non-trivial.
\end{lemma}
%\marginnote{there is some inconsistency here.  Here we are looking at the exponential of the weights rather than the weights themselves}
\begin{proof}
$g$ has a unique Jordan decomposition $g=g_sg_u=g_ug_s$ where $g_s$ is semisimple and $g_u$ is unipotent, and $g_s, g_u\in\bfH(\bR)$. Then $\psi(g)=\psi(g_s)\psi(g_u)$ is the Jordan decomposition of $\psi(g)$.  Since $\psi(g)$ and $\psi(g_s)$ have the same eigenvalues, $\psi(g_s)$ is a hyperbolic matrix.

The semisimple element $g_s\in\bfH(\bR)$ belongs to a $\bR$-torus $\bfT$.  It follows that for all weights $\lambda\in X^*(\bfT)$ of $\psi$, $|\lambda(g_s)|\neq 1$ where $X^*(\bfT)$ denotes the group of characters of $\bfT$.

There is a unique decomposition $\bfT=\bfT_s\bfT_a$ into an $\bR$-split torus $\bfT_s$ and an $\bR$-anisotropic torus $\bfT_a$. We further decompose $g_s=g_{s,s}g_{s,a}$ with $g_{s,s}\in\bfT_s(\bR)$ and $g_{s,a}\in\bfT_a(\bR)$.

Write $\tlambda(t)=\lambda(t)\overline{\lambda(\overline t)}$, then $\tlambda$ is a character of $\bfT$ defined over $\bR$ and is hence trivial on $\bfT_a$. Then $\tlambda(g_{s,s})=\tlambda(g_s)=|\lambda(g_s)|^2\neq 1$. Moreover, $\lambda|_{\bfT_s}$ is defined over $\bR$ for $\bfT_s$ is split, and thus $\tlambda(g_{s,s})=\lambda^2(g_{s,s})$. So $\lambda(g_{s,s})$, which is real, is not equal to $\pm1$. This shows $\psi(g_{s,s})$ is a hyperbolic matrix.

On the other hand, $\bfT_s$ is contained in some maximal $\bR$-split torus $\bfA'$. It follows that all restricted weights $\chi$ of $\bfA'$ are non-trivial. As all maximal $\bR$-split tori are $\bfH(\bR)$-conjugate, the lemma follows.
\end{proof}

\subsection {Proof of  Theorem \ref{ConjThm}}
As discussed above (1) and (4) of Theorem \ref{thm:mainnilmanifoldfull} hold for $(L, \hat \Gamma, \hat \alpha, \hat \rho)$.   (2)  of Theorem \ref{thm:mainnilmanifoldfull}   follows imediately from Theorem \ref{QI}.  Once we know that all weights of the representation given by the linear data are non-trivial,
 (3) of Theorem \ref{thm:mainnilmanifoldfull}   follows immediately   from Lemma \ref{NonResonance}.
  Theorem \ref{thm:mainnilmanifoldfull} then gives a map $h$ intertwining the actions $\hat \alpha$ and $\hat \rho$.  Since $\hat \alpha$ and $\hat \rho$ factor through the restriction of the actions $\alpha$ and $\rho$ to a finite index subgroup $\bar\Gamma\subset \Gamma\subset G$, the same $h$ intertwines $\restrict{\alpha}{\bar\Gamma}$ and $\restrict{\rho}{\bar\Gamma}$ and Theorem \ref{ConjThm} follows.

\section{Smooth rigidity for Anosov actions}\label{sec:Smooth}

In this section we prove Theorem \ref{SmoothThm}.  Our approach uses many of the same ideas  as \cite{MR1367273}.   %heavily uses Zimmer's Cocycle Superridigity Theorem which we recall in Section \ref{sec:ZCSR}

\subsection{Reductions and proof of Theorem \ref{SmoothThm}}\label{sec:SmoothReduction}
Let $G$ and $ \Gamma$ satisfy Hypothesis \ref{HigherRank} and let $\alpha$ be as in Theorem \ref{SmoothThm}.  Let $\rho$ be the linear data of $\alpha$.  Note that if $\alpha(\gamma)$ is Anosov then $D\rho(\gamma)$ is hyperbolic.   Replacing $\Gamma$ with a finite-index subgroup $\Gamma_1\subset \Gamma$, we may assume from Theorem \ref{ConjThm} that there is a (H\"older) continuous $h\colon M\to M$ intertwining  the actions  $\alpha$ and $\rho$.  Recall we assume $\alpha(\gamma_0)$ is Anosov for some $\gamma_0\in \Gamma$. Taking a power, it follows that $\alpha(\gamma_1)$ is Anosov for some $\gamma_1\in \Gamma_1$.  By Manning's Theorem \cite{MR0358865}, $h$ is a bi-H\"older  homeomorphism (see also \cite[Section 18.6]{MR1326374}); in particular, the linear data $\rho$ uniquely determines the nonlinear action $\alpha$ of $\Gamma_1.$

We recall the notation  the constructions from Section  \ref{sec:verify1}.  In particular, there is a center-free, semisimple Lie group $G^*$ without compact factors, a continuous surjective morphism $p_1\colon G\to G^*$, finite-index subgroup $\bar \Gamma\subset \Gamma_1$,  and representation $\rho^*\colon \Gamma^* \to \Aut(N)$ where $\Gamma^* =p_1(\bar \Gamma)$ such that  $\restrict \rho {\bar \Gamma} = \rho^* \circ p_1$.  Since the linear data $\restrict \rho {\bar \Gamma}$ uniquely determines the nonlinear action $\restrict \alpha {\bar \Gamma}$, it follows that $\alpha$ factors through an action $\td \alpha^*$ of $\Gamma^*$: $\td \alpha^*( \gamma^*) = \alpha (p_1^{-1}(\gamma^*))$.

Recall that the Lie group $G^*$ is a real algebraic group, that is, $G^*=\bfG^*(\bR)^\circ$ for a semisimple algebraic group $\bfG^*$ defined over $\bR$.
Let $\td\bfG$ be the algebraically simply connected cover of $\bfG^*$. Then $\td G=\td \bfG(\bR)^\circ$ is a finite cover of $G^*$, whence  $\td G$ has finite center and no compact factors. Let $\td \Gamma$ be the lift of $\Gamma^* $ to $\td G$. The  projection $\td \Gamma\to \Gamma^*$ induces an action $\td \alpha$ of $\td \Gamma$ by $C^\infty$ diffeomorphisms of $M$; moreover, the action $\td \alpha$ lifts to an action by diffeomorphism of  $N$ and induces linear data $\td \rho$ which factors through $\rho^*$.    %Moreover,

Note that map $h$ guaranteed by Theorem \ref{ConjThm} intertwining the actions of $\alpha$ and $\rho$ also intertwines the  actions $\td \rho$ and $\td \alpha$.
It is therefore sufficient to prove Theorem \ref{SmoothThm} under the following stronger hypotheses.

\begin{hypothesis}\label{hyp7}
 $\bfG$ is a simply connected semisimple algebraic group defined over \new{$\bR$} all of whose $\bR$-simple factors %{\sout{\red{are anisotropic and}}} 
 have $\R$-rank $2$ or higher, $G= \bfG(\R)^\circ$, and  $\Gamma\subset G$ is a lattice.
 $\alpha$ is an action of $\Gamma$ by $C^\infty$ diffeomorphisms of a nilmaniofld $M= N/\Lambda$  which lifts to an action by diffeomorphisms of $N$ and $\rho\colon \Gamma\to \Aut(M)$ is the associated linear data.
 $h\colon M\to M$ is a homeomorphism such that $h\circ \alpha(\gamma)= \rho(\gamma)\circ h$ for all $\gamma\in \Gamma$.
% be a finite-index subgroup of $\bfG(\bZ)\cap\bfG(\bR)^\circ$. Let $\alpha$ be a smooth $\Gamma$-action on a nilmanifold $M=N/\Lambda$.  Assume the action $\alpha$ lifts to an action on $N$ \marginnote{added lifting here so that linear data is defined} and let $\rho\colon \Gamma\to\Aut(M)$ denote the associated linear data.  Suppose:
\end{hypothesis}

Assuming Hypothesis \ref{hyp7}, the proof of  Theorem \ref{SmoothThm} proceeds by studying the restriction of $\alpha$ and $\rho$ to an appropriately chosen finitely generated discrete abelian subgroup $\Sigma\subset \Gamma$.

We recall the following definition.
\begin{definition}
For an abelian group $\Sigma$ and two actions $\rho\colon \Sigma\to\Aut(M)$, $\rho'\colon \Sigma\to\Aut(M')$ by nilmanifold automorphisms, we say $\rho'$ is an {\bf algebraic factor action} of $\rho$, if there is an algebraic factor map $\pi\colon M\to M'$ such that $\pi\circ\rho(g)=\rho'(g)\circ\pi$ for all $g\in\Sigma$. We further say $\rho'$ is a {\bf rank-one algebraic factor action} of $\rho$  if in addition there is a finite-index subgroup $\Sigma'<\Sigma$ such that the image group $\rho'(\Sigma')<\Aut(M')$ is cyclic.
\end{definition}

%Theorem \ref{SmoothThm} will follow immediately  from the following two facts:

The following proposition is the main result of this section.
\begin{proposition}\label{AbelianSubaction}
%Let $G$ and $\Gamma$ be as in Hypothesis \ref{HigherRank} and \eqref{hyp2}, $\alpha$ be a $C^\infty$ action  of $\Gamma$ on a nilmanifold $M$.  Suppose $\alpha$ can be lifted to an action on $\Gamma$ and let $\rho$ be the associated linear data.
Let $G$, $\Gamma$, $\alpha$, and $\rho$ be as in Hypothesis \ref{hyp7}.
Suppose $\alpha(\gamma_0)$ is an Anosov diffeomorphism for some $\gamma_0\in\Gamma$. Then  %after restricting to a finite-index subgroup of $\Gamma$,
 there exists a free abelian subgroup $\Sigma\subset\Gamma$  such that $\rho|_\Sigma$ has no rank-one algebraic factor actions and $\alpha(\gamma_1)$ is Anosov for some $\gamma_1\in\Sigma$.
\end{proposition}

Having found an appropriate $\Sigma\subset \Gamma$, Theorem \ref{SmoothThm}
follows from the following proposition.
We recall that for an action $\alpha$ of a discrete abelian group $\Sigma$ by diffeomorphisms of a nilmanifold  with an Anosov element, there is always an abelian action, called the \emph{linearization}  of $\alpha$, by affine nilmanifold transformations.  Moreover these two actions are   conjugate. The  main result of \cite{MR3260859} shows this conjugacy is smooth.

\begin{proposition}[\cite{MR3260859}]\label{SmoothAbelian}
Let $\alpha$ be a $C^\infty$ action by a discrete abelian group $\Sigma$ on a nilmanifold $M$, and $\rho$ be its linearization. Suppose that $\rho$ has no rank-one algebraic factor action and $\alpha(\gamma_1)$ is Anosov for some $\gamma_1\in\Sigma$.  Then $\alpha$ is conjugate to $\rho$ by a $C^\infty$ diffeomorphism that is homotopic to the identity.
\end{proposition}

Let $G$, $\Gamma$, $\alpha$, and $\rho$ be as in Hypothesis \ref{hyp7}.
Suppose $\alpha(\gamma_0)$ is an Anosov diffeomorphism for some $\gamma_0\in\Gamma$.   Let $ \Sigma\subset \Gamma$ and $\gamma_1\in \Sigma $ be as in Proposition \ref{AbelianSubaction}
Note that,  the conjugacy $h$ in Hypothesis \ref{hyp7}  guarantees that the linearization of $\restrict \alpha \Sigma$ coincides with the restriction of the linear data $\rho$ of $\alpha$ to $\Sigma$.
By   Proposition \ref{SmoothAbelian}, there is
a $C^\infty$ diffeomorphism $h'\colon M\to M$ homotopic to the identity which intertwines $\restrict \alpha \Sigma$ and the linearization of $\restrict \alpha \Sigma$.  Furthermore, %for any     $\gamma_1\in \Sigma$ such that  $\alpha(\gamma_1) $ is Anosov,
 there is a lift of $h'$ intertwining the lifts of $\alpha(\gamma_1)$ and $\rho(\gamma_1)\in \Aut(N)$.
From the uniqueness criterion of  semiconjugacies, it follows that  $h$ coincides with $h'$ and hence is $C^\infty$.   Theorem \ref{SmoothThm} follows immediately

In the remainder of this section, we prove Proposition \ref{AbelianSubaction}.  % in the rest of this section.

\subsection{The semigroup of Anosov elements} We recall the setting of  Proposition \ref{AbelianSubaction}.
Let $\gamma_0\in \Gamma$ be such that $\alpha(\gamma_0)$ is an Anosov diffeomorphism. Let $E^\sigma_{\gamma_0}(x)$, $x\in M$, $\sigma=s,u$, be the stable and unstable bundles for $\alpha(\gamma_0)$. Given $\epsilon>0$ and $\sigma=s,u$, let $C^{\sigma}_{\epsilon}(x)$ be the $\epsilon$-cone around $E^\sigma_{\gamma_0}(x)$; that is, for $\sigma=u$, decomposing $v=v^s+v^u$ with respect to the splitting $E^s_{\gamma_0}(x)\oplus E^u_{\gamma_0}(x)$ we have  $v\in C^u_{\epsilon}(x)$ if and only if $|v^s|\leq \epsilon |v^u|$. %(Here $|\cdot|$ is assumed to be an adapted norm for convenience.)

Let $\epsilon>0$ be small (to be determined later). Let $S$ be the set of all $\gamma\in \Gamma$ such that for every $x\in M$,  $$D_x\alpha(\gamma) C^u_{\epsilon}(x)\subset  C^u_{\frac{1}{2}\epsilon}(\alpha(\gamma) (x))$$ and $$D_x(\alpha(\gamma))^{-1} C^s_{\epsilon}(x)\subset  C^s_{\frac{1}{2}\epsilon}((\alpha(\gamma))^{-1}(x))$$
  (that is, $\alpha(\gamma)$ preserves the $\epsilon$-stable and unstable cones)
and, moreover, for every vector $v\in C^u_{\epsilon}(x), |D_x\alpha(\gamma)v|\geq 2|v|$ and for $v\in C^s_{\epsilon}(x), |D_x(\alpha(\gamma))^{-1}v|\geq 2|v|$.

\begin{proposition}\label{semigroup} Suppose $\alpha$,  $\Gamma$, and $\gamma_0$ are as in the assumptions of Proposition \ref{AbelianSubaction}.
For every $0<\epsilon\leq 1$, the set $S$ defined above satisfies the following conditions:
\begin{enumerate}
\item $S$ is a semigroup;

\item\label{anosovsemigroup} for every $\gamma\in S$, $\alpha(\gamma)$ is an Anosov diffeomorphism;

\item for some $N_0>0$, $\gamma_0^{N_0}\in S$;

\item\label{zariskidensesemigroup} there is a Zariski open set $W\subset G$ such that for every $\eta\in \Gamma\cap W$, there is $N>0$ such that $\gamma_0^N\eta\gamma_0^N\in S$;

\item $S$ is Zariski dense in $G$.
\end{enumerate}
\end{proposition}
 In what follows, we will always assume that $N_0=1$ by choosing our Anosov element as $\alpha(\gamma_0^{N_0})$ if necessary.

\begin{proof}[Proof of Proposition \ref{semigroup}]
We postpone the proof of item (\ref{zariskidensesemigroup}).  % until the next section.  % and proof the rest assuming this is true.

That $S$ is a semigroup is clear from definition.  That $\gamma_0^{N_0}\in S$ is straightforward by choosing $N_0>0$ large enough. (\ref{anosovsemigroup}) follows from standard cone estimates (see \cite[Section 6.4]{MR1326374}). We assume $N_0=1$ for the remainder.

Let $\bar S$ be the Zariski closure of $S$.  Then $\bar S$ is a group. By (\ref{zariskidensesemigroup}),  for every $\eta\in \Gamma\cap W$ there is an $N$ such that $\gamma_0^N\eta\gamma_0^N\in S\subset \bar S$.  As $\gamma_0\in S\subset \bar S$ and as $\bar S$ is a group, it follows that  $\eta\in \bar S$.  Hence $\Gamma\cap W\subset \bar S$.
Since $W$ is Zariski open in $G$ and $\Gamma$ is Zariski dense  in $G$ by Borel density theorem,
$\Gamma\cap W\subset \bar S$ is Zariski dense in $G$. Since  $\bar S$ is Zariski closed,  $\bar S=G$ and thus $S$ is Zariski dense.
\end{proof}

\def\Gl{\GL}

\subsection{Proof of item (\ref{zariskidensesemigroup}) of Proposition \ref{semigroup}}\label{Sec:llppookjjjd}
In this section, we follow and substantially extend the main argument in \cite{MR1367273}.
Recall that the derivative of the action $\alpha \colon \Gamma \to \diff(M)$ induces a linear cocycle $D_x\alpha(\gamma)$ over the action $\alpha$.    Recall also that, as $\alpha(\gamma_0)$ is Anosov for some $\gamma_0$, the semiconjugacy $h\colon M\to N/\Lambda$ is a homeomorphism.     The push-forward of the Haar measure on $N/\Lambda$ under $h^{-1}$  is then $\alpha$-invariant and, moreover, coincides with the measure of maximal entropy for $\alpha(\gamma_0)$.  By Theorem \ref{CocycleSuperrigidity}  (applied to the Jacobian-determinant cocycle) and Livsic's Theorem it follows that this measure is smooth.
Denote this smooth  $\alpha$-invariant  measure by $m$.
Note that, as the linear data associated to $\alpha(\gamma_0)$ is hyperbolic, the Haar measure on $N/\Lambda$ is ergodic whence $m$ is ergodic for the action $\alpha$.
Fix a trivialization of $TM=M\times \lien$ and an identification $\lien=\bR^d$.  %from the linear point of view.
Identify $\lien=\bR^d=V$  and $\GL(d,\bR)=\GL(V)=\GL(\lien)$ unless some confusion arises.
%Note that, as $\alpha$ preserves a smooth measure $m$, we view  $D_x\alpha(\gamma)$ as a $\SL(d, \R)$-valued cocycle.

By  Theorem \ref{CocycleSuperrigidity}, for each ergodic $\alpha$-invariant measure $\mu$, there are a measurable map $C^\mu\colon (M,\mu)\to \GL(d,\bR)$, a linear representation $D^\mu\colon G\to \GL(d,\bR)$, and a compact-group valued, measurable cocycle $K^\mu\colon \Gamma\times M \to SO(d)$ such that
\begin{equation}\label{triv}
D_x\alpha(\gamma)=C^\mu(\alpha(\gamma)(x))D^\mu(\gamma)K^\mu(\gamma,x)(C^\mu(x))^{-1}
\end{equation} %\marginnote{I see why $D$ is $\SL$ valued, why is $C$ SL valued?}
and
$K^\mu$ commutes with $D^\mu$:  for every $g\in G$, $\eta\in\Gamma$, and $x\in M$, $D^\mu(g)K^\mu(\eta,x)=K^\mu(\eta,x)D^\mu(g)$.
%\emph{When no confusion arises we shall drop the superscript $\mu$ on $D^\mu$, etc.}

%Indeed, there is a decomposition $V=\bigoplus_i V^\mu_i$ in such a way that each $V^\mu_i$ is a $D^\mu$-invariant subspace, with $p^\mu_i:V\to V^\mu_i$ the associated projections, and for every $i$, $V^\mu_i=U^\mu_i\otimes W^\mu_i$ and $D^\mu_i=D^\mu|V^\mu_i=\alpha^\mu_i\otimes id$ where $\alpha^\mu_i:G\to \SL(U^\mu_i)$ is an irreducible representation on $U_i$ and $id$ is the trivial representation on $W^\mu_i$ (i.e. $\dim(W^\mu_i)$ is the multiplicity with which the irreducible representation $\alpha^\mu_i$ appears in $D^\mu$), we shall denote $D^\mu=\bigoplus_i\alpha^\mu_i\otimes id$. We assume also that $\alpha^\mu_i$ is not conjugated to $\alpha^\mu_j$ if $i\neq j$.
%
%Finally, the cocycle $K^\mu$ is of the form $$K^\mu(\gamma,x)=\bigoplus_i id\otimes\kappa^\mu_i(\gamma,x)$$ where $\kappa^\mu_i:\Gamma\times M\to SO(W^\mu_i).$ is a compact valued cocycle.

Recall $\gamma_0$ is the distinguished element with $\alpha(\gamma_0)$ Anosov.  Fix an enumeration $\Gamma\sm\{\gamma_0\} = \{ \gamma_1, \gamma_2, \dots\}$.  For $j = 0, 1, 2, \dots$ let
$$\eta_j := \gamma_j\gamma_0\gamma_j^{-1}.$$
%Let $\langle\eta_1,\dots, \eta_r\rangle$ be a set of generators for $\Gamma$ and let $\gamma_i=\eta_i\gamma_0\eta_i^{-1}$.
Observe that $\alpha(\eta_j)$ is Anosov for every $j$.  For $x\in M$,
let $E^s_{\eta_j}(x)$ and $E^u_{\eta_j}(x)$ denote, respectively,  the stable or unstable bundle for $\alpha(\eta_j)$ at the point $x$.
For $\sigma=s,u$, let $\sigma=\dim E^\sigma_{\gamma_0}(x)$.
Note that $E^\sigma_{\eta_j}(x) = \left(D\alpha(\gamma_j)E^\sigma_{\gamma_0}\right)(x)$

For $g\in G$   let $E^{s,\mu}_g$ and    $E^{u,\mu}_g$ denote, respectively, the stable and unstable spaces of the linear map $D^\mu(g)$.  Note that $D^\mu(g)$ need not be hyperbolic whence the subspaces $E^{\sigma,\mu}_g$ may not be transverse.  However,  (as $K^\mu(\gamma,x)$ is compact-valued and commutes with $D^\mu$) for $\mu$-a.e. $x$, we have $$E^\sigma_{\eta_i}(x)=C^\mu(x)E^{\sigma,\mu}_{\eta_i}.$$

\newcommand{\bfS}{\mathbf{S}}

For $r= 0, 1, 2, \dots$ let $S^{r,\mu}\subset \Gl(V)$ be given by
	$$S^{r,\mu}= \bigcap_{\substack{i=1,\dots, r \\ \sigma=s,u}} \Stab(E^{\sigma,\mu}_{\eta_i})=\bigcap_{\substack{i=1,\dots, r \\ \sigma=s,u}}  \Stab(D^\mu(\gamma_i)E^{\sigma,\mu}_{\gamma_0}).$$
Note that each $S^{r,\mu} = \bfS^{r,\mu}(\R)$ where $\bfS^{r,\mu}$ is an algebraic group defined over $\R$.  Moreover, $\bfS^{r+1,\mu}\subset \bfS^{r,\mu}$ and   each $\bfS^{r,\mu}$ has finitely many components.
Counting dimension and connected components it follows there is a $r(\mu)$ so that
$ \bfS^{r,\mu} =  \bfS^{r(\mu),\mu} $ for all $r\ge r(\mu)$.
Let $\bfS^{\mu} = \bfS^{r(\mu),\mu} $ and $S^\mu= \bfS^{\mu} (\R).$
We then have
$$S^\mu = \bigcap_{\substack{\gamma\in \Gamma \\\sigma=s,u}} D^\mu (\gamma)\Stab(E^{\sigma,\mu}_{\gamma_0})D^\mu(\gamma)^{-1}.$$
It follows that $D^\mu(\Gamma)$ normalizes $S^\mu$ whence by Zariski density of $\Gamma$ in $G$,
$$S^\mu= \bigcap_{\substack{g\in G \\\sigma=s,u}} D^\mu(g)\Stab(E^{\sigma,\mu}_{\gamma_0})D^\mu(g)^{-1}.$$
%Given $\gamma\in \Gamma$ we have

For $\sigma=s,u$, denote by $Gr(V,\sigma)$ the Grassmannian of subspaces in $V$ of dimension $\dim E_{\gamma_0}^{\sigma,\mu}$. Let $$\Phi\colon \GL(V)\times\left( (Gr(V,s))^{r(\mu)}\times (Gr(V,u))^{r(\mu)}\right)\to (Gr(V,s))^{r(\mu)}\times (Gr(V,u))^{r(\mu)}$$ be the natural action.
As $\Phi$ is an algebraic action,
\begin{lemma}\label{orbitopen}
Let $E\in (Gr(V,s))^{r(\mu)}\times (Gr(V,u))^{r(\mu)}$. Then $\Orb_\Phi(E)$, the orbit of $E$ under $\Phi$, is open in its Zariski closure $\overline{\Orb_\Phi(E)}$.
\end{lemma}

%where
%\begin{eqnarray*}
%S^\mu&=&\bigcap_{i=1,\dots, r,\sigma=s,u} \Stab(E^{\sigma,\mu}_{\gamma_i})=\bigcap_{i=1,\dots, r,\sigma=s,u}  \Stab(D(\eta_i)E^\sigma_{\gamma_0})\\&=&\bigcap_{i=1,\dots, r,\sigma=s,u}  D(\eta_i)\Stab(E^\sigma_{\gamma_0})D(\eta_i)^{-1}\\
%&=&\bigcap_{\gamma\in \Gamma,\sigma=s,u} D(\gamma)\Stab(E^\sigma_{\gamma_0})D(\gamma)^{-1}\\
%&=&\bigcap_{g\in G,\sigma=s,u} D(g)\Stab(E^\sigma_{\gamma_0})D(g)^{-1}.
%\end{eqnarray*}

%Observation: This is what Zimmer calls a smooth action in his book. Since this $\Phi$ action is a boundary action, it is smooth.

Let  $\tau\colon M\to (Gr(V,s))^{r(\mu)}\times (Gr(V,u))^{r(\mu)}$ be defined by
		$$\tau(x)=\big((E^{s}_{\eta_i}(x))_i,(E^{u}_{\eta_i}(x) )_i\big)_{i = 0, \dots, r(\mu)}.$$
	%$$\tau(x)=\big((E^{s,\mu}_{\eta_i})_i,(E^{u,\mu}_{\eta_i})_i\big)_{i = 0, \dots, r(\mu)}.$$
As observed above, for $\mu$-a.e.\ $x$  $$\tau(x)=\Phi\left(C^\mu(x),\left((E^{s,\mu}_{\eta_i})_i,(E^{u,\mu}_{\eta_i})_i\right)\right).$$
Let $\Orb^\mu$ be the orbit under $\Phi$ of $\big((E^{s,\mu}_{\eta_i})_i,(E^{u,\mu}_{\eta_i})_i\big)_{i = 0, \dots, r(\mu)}$.  $\Phi$ induces a smooth parameterization $$\phi^\mu\colon \GL(V)/S^\mu\to \Orb^\mu.$$

Since $\tau$ is continuous, we have that $\tau(\supp(\mu))$ is compact. By the   Lemma \ref{orbitopen}, for a $\mu$-.a.e.\ $x$, the orbit  $\Orb_\Phi(\tau(x))=\Orb^\mu$ is open  in $\overline{\Orb^\mu}\supset \tau(\supp(\mu))$. Hence   $$U^\mu:=\tau^{-1}(\Orb^\mu),$$ is (relatively) open and dense in $\supp(\mu)$. Moreover $\mu(U^\mu)=1$.
Via the parametrization $\phi^\mu$, we have  a continuous map $\hat C^\mu\colon U^\mu\to \GL(V)/S^\mu$ given by  $$\hat C^\mu(x):=(\phi^\mu)^{-1}\circ\tau(x).$$ Moreover,  $C^\mu(x)S^\mu=\hat C^\mu(x)$ for $\mu$-a.e.\ $x$.

\begin{lemma}\label{imageofbundle}
Let $\mu$ be an  ergodic $\alpha$-invariant measure.  The set $U^\mu$ is $\alpha$-invariant and
for every $\gamma\in\Gamma$, % and  %every ergodic $\alpha$-invariant measure $\mu$,
  $x\in U^{\mu}$, and  $\sigma=s,u$ we have %that $E^\sigma_{\gamma_0}(x) = \hat C^\mu(x) E^{\sigma,\mu}_{\gamma_0}$ and
$$D_x\alpha(\gamma)E^\sigma_{\gamma_0}(x)=D_x\alpha(\gamma)\hat C^\mu(x)E^{\sigma,\mu}_{\gamma_0}=\hat C^{\mu}(\alpha(\gamma)(x))D^{\mu}(\gamma)E^{\sigma,\mu}_{\gamma_0}.$$
\end{lemma}

%Note that as  $U^\mu$ is $\Gamma$-invariant, for $x\in U^{\mu}$ the expression  $\hat C^{\mu}(\alpha(\gamma)(x))D^{\mu}(\gamma)E^{\sigma,\mu}_{\gamma_0}$ makes sense.
\begin{proof}[Proof of Lemma \ref{imageofbundle}]
For $\mu$-a.e.\ $x\in U^\mu$ we have that $\tau(\alpha(\gamma)(x)) = \Phi\left(D^\mu(\gamma), \tau(x)\right)$ whence, for such $x$,
	\begin{equation}\label{KKKLLLLDDDD}\alpha(\gamma)(x) \in \tau\inv\left(\Phi\left(D^\mu(\gamma), \tau(x)\right)\right).\end{equation}
Since $\tau\circ \alpha(\gamma)$ and $x\mapsto  \Phi\left(D^\mu(\gamma), \tau(x)\right)$ are continuous and since $U^\mu$ is open and dense in $\supp(\mu)$,  \eqref{KKKLLLLDDDD} holds for all $x\in U^\mu$.  It follows that $\alpha(\gamma)(x) \in U^\mu$ for all $x\in U^\mu$.

For any $x\in M$ and $\gamma\in \Gamma$ we have  that $$D_x\alpha(\gamma)E^\sigma_{\gamma_0}(x)=E^{\sigma}_{\gamma\gamma_0\gamma^{-1}}(\alpha(\gamma)(x)).$$  Also $$D^\mu(\gamma)E^{\sigma,\mu}_{\gamma_0}=E^{\sigma,\mu}_{\gamma\gamma_0\gamma^{-1}}.$$  Recall that  $S^\mu$   stabilizes each of the spaces $E^{\sigma,\mu}_{\gamma\gamma_0\gamma^{-1}}$ and that there is a measurable $C^\mu(x)$ with
	$$C^\mu(x) E^{\sigma,\mu}_{\gamma\gamma_0\gamma^{-1}} = E^{\sigma}_{\gamma\gamma_0\gamma^{-1}}(x)$$
for $\mu$-a.e.\ $x$.  As the  function $\hat C^\mu(x)$ and the bundles $D_x\alpha(\gamma)E^\sigma_{\gamma_0}(x)$ are continuous on $U^{\mu}$, the result follows.
%
%and use the measurable $C^\mu(x)$for the matching.
\end{proof}

%Moreover $U^\mu$ is $\Gamma$-invariant.

 %Moreover,
% For $\mu$-a.e. $x$, we have $$\hat C(x):=(\phi^\mu)^{-1}\circ\tau(x) =C(x)S^\mu.$$

We summarize the above with the following lemma.

\begin{lemma}\label{lemma:measureshomo}
There are countably many ergodic, $\alpha$-invariant probability measures $\mu_i$, and relatively-open, relatively-dense, $\alpha$-invariant sets $U_i\subset \supp(\mu_i)$,  and   continuous maps $\hat C^{\mu_i}\colon U_i\to \GL(V)/S^{\mu_i}$
such that
\begin{enumerate}
	\item $M$ is the    union $M = \bigcup_{i= 0}^\infty U_i$;
	\item $C^{\mu_i}(x)S^{\mu_i}=\hat C^{\mu_i}(x)$ for $\mu_i$-a.e.\ $x\in U_i$.
\end{enumerate}
%a relatively-open,
%For every $x\in M$ there is an ergodic, $\alpha$-invariant probability measure $\mu_x$, a relatively-open, relatively-dense, $\alpha$-invariant set $U_x\subset \supp(\mu_x)$
%and a continuous maps $\hat C^{\mu_x}\colon U_x\to \SL(V)/S^{\mu_x}$ such that $C^{\mu_x}(x)S^{\mu_x}=\hat C^{\mu_x}(x)$ for $\mu_i$-a.e.\ $x$.
%
%
\end{lemma}

%Summarizing the above,
% \begin{lemma}
%There are finitely many ergodic, $\alpha$-invariant measures $\mu_i$, $i=0,\dots, l$ with $\mu_0=m$ and corresponding relatively open, relatively dense, $\Gamma$-invariant sets $U^{\mu_i}\subset \supp(\mu_i)$
% such that $M$ is the disjoint union $$M=\bigcup_{i=0}^l U^{\mu_i}.$$  \marginnote{why disjoint}
%	%Moreover, $U^{\mu_i}\cap U^{\mu_j}=\emptyset$ if $i\neq j$.
%For each $i$, we have continuous maps $\hat C^{\mu_i}\colon U^{\mu_i}\to \SL(V)/S^{\mu_i}$ such that $C^{\mu_i}(x)S^{\mu_i}=\hat C^{\mu_i}(x)$ for $\mu_i$-a.e.\ $x$.
%\end{lemma}

\begin{proof}
We start with the smooth measure $\mu_0=m$ and the corresponding open set $U_0 = U^m$ as constructed above.  The image of $U_0$ under the conjugacy $h$, is a $\rho$-invariant, open dense subset of the nilmanifold $M$.  It follows that the complement of $h(U_0)$ coincides with the boundary of $h(U_0)$ and is a closed, $\rho$-invariant set.  In particular, the complement of $h(U_0)$ is saturated by orbit closures.  From Proposition \ref{prop:Ratner} and using that $U_0$ is dense in $M$,
 the complement of $h(U_0)$ is a   countable union
	$$N/\Lambda \sm h(U_0)= \bigcup   V_i$$ where each $V_i$ is a finite union of $\rho$-invariant sub-nilmanifolds of $N/\Lambda$ of  dimension at most $d-1$  where $d = \dim N$.  Moreover, each $V_i$ coincides with the support of an ergodic,  $\rho$-invariant   $\nu^{V_i}$.  We note (as $\rho(\gamma)$ is hyperbolic for some $\gamma$), that there are a countable number of such $V_i$.

Let $\mu^{V_i}:= (h^{-1})_* \nu^{V_i}$.  For each $\mu^{V_i}$ we may repeat the above procedure  and obtain sets $U_{V_i}$ such that $h(U_{V_i})$ is open and dense in $V_i$.  As $U_{V_i}$ is $\alpha$-invariant and $V_i$ is the finite union of submanifolds of dimension at most $d-1$, it follows that $V_i \sm h(U_{V_i}) $ is a countable union  $$V_i \sm h(U_{V_i}) = \bigcup W_j$$ where each $W_j$ is a finite union of $\rho$-invariant sub-nilmanifolds of $N/\Lambda$ of  dimension at most $d-2$.

Proceeding recursively, we define a countable collection of ergodic $\alpha$-invariant measures $\mu_i$ with corresponding sets $U_i$.  That every $x \in M$ is contained in a $U_i$ follows as the dimension of the complement decreases at each step of recursion.   %That the sets $U_i$ are disjoint follows are they are relatively open and relatively dense subsets of $\supp(\mu_i)$ with  full measure, and (as each $\mu_i$ is ergodic) the measures $\mu_i$ are pairwise disjoint.
\end{proof}

The set $W$ appearing in   (\ref{zariskidensesemigroup}) of Proposition \ref{semigroup} is defined as the set $W$ in the following lemma.

\begin{lemma}\label{transversality}
Let $\mu_i$ be the measures in
Lemma \ref{lemma:measureshomo}.
Let $W$ be the set of all $g\in G$ such that $D^{\mu_i}(g)E^{s,\mu_i}_{\gamma_0}$ is transverse to  $E^{u,\mu_i}_{\gamma_0}$  and $D^{\mu_i}(g)E^{u,\mu_i}_{\gamma_0}$ is transverse to   $E^{s,\mu_i}_{\gamma_0}$ for every $i$.
Then $W$ is a non-empty Zariski open set in $G$.
\end{lemma}
\begin{proof}
Up to conjugation, there at only finitely many representations of $G$ into $\GL(d, \R)$.  In particular, up to conjugation there are only finitely many values of  $D^{\mu_i}(g)$ and $E^{s,\mu_i}_{\gamma_0}$.  Since $\alpha(\gamma_0)$ is Anosov, for every $\mu_i$,
$E^{s,\mu_i}_{\gamma_0}$ is transverse to  $E^{u,\mu_i}_{\gamma_0}$.  Then $W$ is the finite intersection of Zariski open sets indexed by conjugacy classes of representations $D^{\mu_i}(g)$.
\end{proof}

With the above lemmas we show that  the set $W$  satisfies   condition of  (\ref{zariskidensesemigroup}) of Proposition \ref{semigroup} via the following proposition.

%
%
%The following proposition solves our problems.

\begin{proposition}\label{abstractanosov}
Let $f$ be an Anosov diffeomorphism with splitting $T_xM=E^s(x)\oplus E^u(x)$, and let $0<\epsilon\leq1$. Let $g$ be a diffeomorphism and assume   for every $x\in M$ that  $D_xg E^s(x)$ intersects transversally $E^u(g(x))$ and  $D_xg E^u(x)$ intersects transversally $E^s(g(x))$.
%and also interchanging $s$ and $u$.
Then there is $N>0$ such that for every $n\geq N$, writting $F=f^n\circ g\circ f^n$,  for every $x\in M$,  $$D_xF C^u_{\epsilon}(x)\subset  C^u_{\frac{1}{2}\epsilon}(F(x))$$ and $$D_xF^{-1} C^s_{\epsilon}(x)\subset  C^s_{\frac{1}{2}\epsilon}F^{-1}(x))$$ and for every vector $v\in C^u_{\epsilon}(x), |D_xFv|\geq 2|v|$ and for $v\in C^s_{\epsilon}(x), |D_xF^{-1}v|\geq 2|v|$.
\end{proposition}
Recall that here $C^\sigma_\epsilon(x)$ denotes the $\epsilon$-cone  around $E^\sigma(x)$.
The proof is a  standard argument. We include it here    for completeness.

\begin{proof}[Proof of Proposition \ref{abstractanosov}]
By symmetry  it is sufficient to prove the result for the   unstable cone.  % the stable works by taking $h^{-1}$ and a possibly larger $N$.
We write all linear transformations with respect to the the continuous splitting $T_xM=E^s(x)\oplus E^u(x)$. Then,
$$D_xf=\left( \begin{array}{cc}
A(x) & 0  \\
0 & B(x)
\end{array} \right)$$
and
$$D_xg=\left( \begin{array}{cc}
a(x) & b(x)  \\
c(x) & d(x)
\end{array} \right)$$
where $a(x)\colon E^s(x)\to E^s(g(x))$, $b(x)\colon E^u(x)\to E^s(g(x))$,  $c(x)\colon E^s(x)\to E^u(g(x))$, and  $d(x)\colon E^u(x)\to E^u(g(x))$.
As  $D_xg E^u(x)$ and $E^s(g(x))$ are transverse, it follows that  $d(x)$ is invertible for every $x$. By  continuity of the bundle  $ E^u(x)$, continuity of the derivative $D_xg$, and compactness of $M$ there is $r>0$ such that $m(d(x))\geq r$ for every $x\in M$.  Here $m(L)$ denotes the %=\min_{|v|=1}|Lv|$ for a linear map $L$ is the
conorm of a linear map $m(L)=\|L^{-1}\|^{-1}$. Let $C=\max_{x\in m}\|D_xg\|$. Observe that
$$D_xf^n=\left( \begin{array}{cc}
A^{(n)}(x) & 0  \\
0 & B^{(n)}(x)
\end{array} \right)$$
where $A^{(n)}(x)=A(f^{n-1}(x))\dots A(x)$.  %, similarly with $B$.
Choose a norm on $TM$ adapted to $f$; that is, decomposing $v= v^s+ v^u$ acording to the splitting $T_xM = E^s(x) \oplus E^u(x)$ we have $|v| = \max \{|v^u| , |v^s|\}$ and there is a constant $\lambda<1$ such that for every $x\in M$ and $n\geq 0$, $\|A^{(n)}(x)\|\leq \lambda^{n}$ and $m(B^{(n)}(x))\geq \lambda^{-n}$.%\marginnote{define adapted}

Let $\epsilon>0$.   The $N$ in the proposition will depend on $r$, $C$, $\lambda$ and $\epsilon$. We first show for $N$ sufficiently large  that $DF$ preserves the $\epsilon$-unstable cone. First observe that for every positive real number $t>0$, natural number $n\geq 0$ and $x\in M$ we have that $$Df^nC^u_t(x)\subset C^u_{\lambda^{2n}t}(f^n(x)).$$

Now the next Lemma uses the  transversality between $DgE^u$ and $E^s$ is used; in particular, the proof used that for $r$ as defined above, $r>0$.
\begin{lemma}\label{coneforg}
There are $\delta_0>0$ and $T>0$ such that for every $0<\delta\leq\delta_0$ we have that $$D_xgC^u_{\delta}(x)\subset C^u_{T}(g(x)).$$
\end{lemma}
Given Lemma \ref{coneforg}, for $N$ sufficiently large and   $n\geq N$,
\begin{eqnarray*}
D_xFC^u_\epsilon(x)&=&D_x(f^n\circ g\circ f^n)C^u_\epsilon(x)=D_{f^n(x)}(f^n\circ g)D_x f^nC^u_\epsilon(x)\\
&\subset& D_{f^n(x)}(f^n\circ g)C^u_{\lambda^{2n}\epsilon}(f^n(x))\\
&=&D_{g(f^n(x))}f^nD_{f^n(x)}gC^u_{\lambda^{2n}\epsilon}(f^n(x))\\
&\subset& D_{g(f^n(x))}f^nC^u_{T}(g(f^n(x)))\subset C^u_{\lambda^{2n}T}(f^n(g(f^n(x))))\\
&=&C^u_{\lambda^{2n}T}(F(x))
\end{eqnarray*}
Choosing $N$ large enough so that $\lambda^{2N}\epsilon\leq\delta_0$ and $\lambda^{2N}T\leq \frac{1}{2}\epsilon$ it follows that $F$ preserves the $\epsilon$-unstable cones.

We now consider the the growth of the vectors. Let $v\in C^u_\epsilon(x)$. Recall that $|v|=\max\{ |v^s|,|v^u|\}$, hence $|v|=|v^u|$. Since $DF$ preserves the $\epsilon$-unstable cone we have that $|(D_Fv)^s|\leq\frac{\epsilon}{2}|(D_Fv)^u|$ and, since $\epsilon\leq 1$, we have that $|D_Fv|=|(D_Fv)^u|$. Now,
\begin{eqnarray*}
(D_Fv)^u=B^{(n)}(g(f^n(x)))[c(f^n(x))A^{(n)}(x)v^s+d(f^n(x))B^{(n)}(x)v^u]
\end{eqnarray*}
so
\begin{eqnarray*}
|D_Fv|&=&|(D_Fv)^u|\geq\lambda^{-n}|c(f^n(x))A^{(n)}(x)v^s+d(f^n(x))B^{(n)}(x)v^u|\\
&\geq&\lambda^{-n}[r\lambda^{-n}|v^u|-C\lambda^n\epsilon|v^u|]\\
&=&\lambda^{-n}(r\lambda^{-n}-C\lambda^n\epsilon)|v|.
\end{eqnarray*}
Take $N$ large enough such that $\lambda^{-n}(r\lambda^{-n}-C\lambda^n\epsilon)\geq 2$ for all $n\geq N$.
%Take $v=v^s+v^u$ with $|v^s|\leq\epsilon|v^u|$ then a direct computation gives that $D_xFV^s$ equals
%\begin{eqnarray*}
%A^{(n)}(g(f^n(x)))a(f^n(x))A^{(n)}(x)v^s+B^{(n)}(g(f^n(x)))c(f^n(x))A^{(n)}(x)v^s
%\end{eqnarray*}
%and $D_xFv^u$ is
%\begin{eqnarray*}
%A^{(n)}(g(f^n(x)))b(f^n(x))B^{(n)}(x)v^u+B^{(n)}(g(f^n(x)))d(f^n(x))B^{(n)}(x)v^u
%\end{eqnarray*}
\end{proof}

\begin{proof}[Proof of Lemma \ref{coneforg}]
Take a vector $v=v^s+v^u$ in $C^u_{\delta_0}(x)$ ($\delta_0$ to be determined) then $$D_xgv=(a(x)v^s+b(x)v^u)+(c(x)v^s+d(x)v^u).$$ We prove that $$|a(x)v^s+b(x)v^u|\leq T|c(x)v^s+d(x)v^u|$$ for some $T>0$. Note
$$|a(x)v^s+b(x)v^u|\leq (C\delta_0+C)|v^u|=C(\delta_0+1)|v^u|$$ and $$|c(x)v^s+d(x)v^u|\geq r|v^u|-C\delta_0|v^u|=(r-C\delta_0)|v^u|.$$ Take $\delta_0$ such that $r-C\delta_0>0$ and let $T=\frac{C(\delta_0+1)}{(r-C\delta_0)}$.
\end{proof}

For $W$   as in Lemma \ref{transversality}, from Lemma \ref{imageofbundle} we verify for $\eta \in W\cap \Gamma$ that, with $f=\alpha(\gamma_0)$ and $g=\alpha(\eta)$, the transversality in Proposition \ref{abstractanosov} holds.  Item (\ref{zariskidensesemigroup}) of Proposition \ref{semigroup} the follows immediately from  Proposition \ref{abstractanosov}.
% with $f=\alpha(\gamma_0)$ and $g=\alpha(\eta)$ for  $\eta\in W\cap\Gamma$ where $W$ is as in Lemma \ref{transversality}.

\subsection{Abelian subactions without rank-one factors}
In this part we assume

\begin{hypothesis}\label{hypH}
$\bfH$ is a simply connected semisimple algebraic group defined over $\bQ$ whose all $\bR$-simple factors are either anisotropic or of $\bR$-rank 2 or higher. $H=\bfH(\R)^\circ$, and  $\Gamma\subset\bfH(\Z)\cap H$ is an arithmetic lattice in $H$. $\rho\colon\Gamma\to \Aut(M)$ is a $\Gamma$-action by linear automorphisms on a compact nilmanifold $M=N/\Lambda$. Let us denote also $\rho:\Gamma\to\Aut(N)$ the lift and assume that $D\rho:\Gamma\to\Aut(\lien)$ extends to a $\bQ$-rational representation $D\rho:\bfH\to\mathbf{GL}(d)$.
\end{hypothesis}

\begin{proposition}\label{RegularCentralizer} Let $\bfH$, $\Gamma$, $M$, $\rho$ be as in Hypothesis \ref{hypH}  and let $ S<\Gamma$ be a Zariski dense semigroup such that $D\rho(\eta)$ is hyperbolic for all $\eta\in S$. %Suppose every isotropic almost simple factor of $G$ has  rank $2$ or higher.
 Then there is an element $\gamma\in S$ such that:\begin{enumerate}
\item the identity component $Z_\bfH(\gamma)^\circ$ of the centralizer of $\gamma$ in $\bfH(\bR)$ is a maximal $\bQ$-torus contains a maximal $\bR$-split torus defined over $\bQ$ of $\bfH$ and   $\gamma\in Z_\bfH(\gamma)^\circ$; %Property (1) above is true for $\bfH$;
% \item the identity component $Z_\bfH(\gamma)^\circ$ of the centralizer of $\gamma$ in $\bfG(\bR)$ contains a maximal $\bR$-split torus and $\gamma\in Z_\bfG(\gamma)^\circ$;  % defined over $\bQ$ of $\bfG$;
\item $Z_\bfH(\gamma)^\circ\cap\Gamma$ contains a free abelian group $\Sigma\cong\bZ^{\rank_\bR\bfH}$ of finite index;
\item the restricted action $\rho|_\Sigma$ has no algebraic factor action of rank $1$.
\end{enumerate}
\end{proposition}

\begin{proof} The proof is based on the works of Prasad and Rapinchuk \cite{MR1960120, MR2150880}.

We recall that, if $\bfT\subset \bfH$ is an algebraic torus defined over $\bQ$, then there is a canonical $\Gal(\bar\bQ/\bQ)$-action on the character group $X^*(\bfT)$ defined over $\bar\bQ$, given by $(\sigma.\chi)(t)=\sigma^{-1}(\chi(\sigma(t)))$. Moreover, if $\bfT$ is a maximal $\bQ$-torus in a semisimple algebraic group $\bfH$ defined over $\bQ$, then the Galois action permute the roots.

An element $\gamma\in\bfH(\bR)$ is called {\bf regular} (resp. {\bf $\bR$-regular}), if the number of eigenvalues of $\Ad_\gamma$ that are equal to $1$ (resp. on the unit circle), counted with multiplicity, is minimum possible. It is called {\bf hyperregular}, if the number of eigenvalues of $\bigwedge \Ad_\gamma$ %\marginnote{  $\bigcap\Ad_\gamma$:correct notation.}
that are equal to $1$, again counted with multiplicity, is minimum possible.  Here $\bigwedge \Ad_\gamma$ denotes  the action on the exterior powers of $\lieh$, $\bigwedge \lieh$. It is known that hyperregular elements are regular, and that both hyperregular and regular are Zariski open conditions \cite[Remark 1.2]{MR0302822}. Furthermore, for a regular and $\bR$-regular element $\gamma$, $Z_\bfH(\gamma)^\circ$ is a maximal torus that contains $\gamma$ (see \cite[Introduction]{MR1960120}).

By the discussion in \cite[pg.\ 240--241]{MR2150880}, there is a Zariski-dense subset $S'\subset   S$ such that for every $\gamma\in S'$:
\begin{enumerate}
% \item[(1)]   the identity component $Z_\bfH(\gamma)^\circ$ of the centralizer of $\gamma$ in $\bfH(\bR)$ is a maximal $\bQ$-torus contains a maximal $\bR$-split torus defined over $\bQ$ of $\bfG$ and   $\gamma\in Z_\bfH(\gamma)^\circ$; %Property (1) above is true for $\bfH$;

 \item[(1)] Property (1) holds;
  \item[(2')] $\gamma$ is regular and $\bR$-regular;
 \item[(3')] For $\bfT=Z_\bfH(\gamma)^\circ$, the Galois action contains all elements from the Weyl group $W(\bfH,\bfT)$, and the cyclic group $\langle\gamma\rangle$ is Zariski dense in $\bfT$.
\end{enumerate}

By Zariski openess of hyperregular elements, we can find a $\gamma\in S'$ that is hyper-regular.
Prasad-Raghunathan proved in \cite[Lemma 1.15]{MR0302822} that hyperregularity and $\bR$-regularity together implies (2) for $\gamma$. It remains to prove (3).

  We assume $\rho_1\colon \Sigma\to\Aut(M_1)$ is a  rank-one algebraic factor of $\rho|_\Sigma$ and obtain a contradiction. Write $M=N/\Lambda$ and $M_1=N_1/\Lambda_1$. By passing to a subgroup if necessary, we main assume $\rho_1(\Sigma)$ is a cyclic group.

Since $M_1$ is a compact quotient nilmanifold, $N_1=N/N_0$ where $N_0$ is a $\rho|_\Sigma$-invariant subgroup of $N$ defined over $\bQ$. The Lie algebra $\lien_0$ is hence $D\rho|_\Sigma$-invariant rational subspace of $\lien$.

Notice that: the cyclic group $\langle\gamma\rangle$ is Zariski dense in $\bfT$; $\langle\gamma\rangle\cap\Sigma$ has finite index in $\langle\gamma\rangle$; and $\bfT$ is Zariski connected. It follows that $\langle\gamma\rangle\cap\Sigma$ is also Zariski dense in $\bfT$, therefore $\lien_0\otimes_\bR\bC$ is $D\rho|_\bfT$ invariant. So $\rho|_\bfT$ projects to a $\bQ$-representation $D\rho_1$ of the $\bQ$-torus $\bfT$ on $\lien_1\otimes_\bR\bC$, where $\lien_1$ is the Lie algebra of $N_1$.

By property (1), $\bfT$ contains a maximal $\bR$-split $\bQ$-torus $\bfT_s$ of $\bfH$. Denote $r=\dim\bfT_s=\rank_\bR\bfH$. The restriction to $\bfT_s$ is a morphism $X^*(\bfT)\to X^*(\bfT_s)$ defined over $\bQ$, and hence interwines the canonical actions by $\Gal(\bar\bQ/\bQ)$ on both character groups. By \cite[Corollary 21.4]{MR1102012}, every element $\omega\in W(\bfH,\bfT_s)$ is represented by the restriction of some element $\tilde\omega\in W(\bfH, \bfT)$ to $\bfT_s$. Hence the Galois action on $X^*(\bfT_s)$ contains the action by the Weyl group $W(\bfH,\bfT_s)$, thanks to property (3').

The $\bR$-torus $\bfT$ decomposes as an almost direct product $\bfT_a\cdot\bfT_s$ where $\bfT_a$ is a maximal $\bR$-anisotropic torus in $\bfT$. $\bfT_a\cap\bfT_s$ is a finite subgroup of torsion elements in $\bfT_s\cong\bfH_m^r$. In particular, $\bfT_a(\bR)\cap\bfT_s(\bR)$ is a finite subgroup all of whose elements have order $2$. Therefore for $t\in\bfT(\bR)$ and $g=t^2$ there is a unique decomposition $g=g_ag_s$ with $g_a\in\bfT_a(\bR)$ and $g_s\in\bfT_s(\bR)$. Moreover, $g_s$ is in $\bfT_s(\bR)^\circ\cong(\bR_{>0})^r$.

Without loss of generality, one may replace $\Sigma$ with $\{\sigma^2:\sigma\in\Sigma\}$. Then all elements $\sigma\in\Sigma$ can be decomposed as above. $\sigma\to \sigma_a$ and $\sigma\to \sigma_s$ are group morphisms on $\Sigma$.

For any non-trivial $\sigma\in\Sigma$, $\sigma_s$ is not trivial. Otherwise $\sigma=\sigma_a$ lie in the compact $\bfT_a(\bR)$. Since $\Sigma\subset\Lambda$ is discrete, $\sigma$ must be an torsion element, contradicting our assumption that $\Sigma$ is free abelian. It follows that $\sigma\to \sigma_s$ is an isomorphism from $\Sigma$ to $\Sigma_s=\{\sigma_s:\sigma\in\Sigma\}$. Furthermore, again because $\bfT_a(\bR)$ is compact and $\Sigma$ is discrete, $\Sigma_s$ is discrete and hence is a lattice in $\bfT_s(\bR)^\circ$.

Fix a basis $\sigma_1,\cdots,\sigma_r$ of $\Sigma$ and write $(\sigma_i)_s$ as $(e^{\theta_{i1}},\cdots, e^{\theta_{ir}})$ in $\bfT_s(\bR)^\circ\cong(\bR_{>0})^r$. Then $(\theta_{ij})$ is a non-degenerate matrix. Define a group morphism $\cL\colon X^*(\bfT_s)\to\bR^r$ by
$$\cL(\chi)=\big(\log|\chi((\sigma_1)_s)|,\cdots,\log|\chi((\sigma_r)_s)|\big).$$
Recall that with $\bfT_s$ identified with $\bfH_m^r$, the coordinate maps $\pi_j\colon (t_1,\cdots, t_r)\to t_j$ form a basis of $X^*(\bfT_s)\cong\bZ^r$. Note $\cL(\pi_j)=(\theta_{1j},\cdots,\theta_{rj})$. Thus by non-degeneracy of $(\theta_{ij})$, $\cL$ embeds $X^*(\bfT_s)$ as a lattice into $\bR^r$.

For any character $\chi\in X^*(\bfT)$, $\tchi\colon t\to \chi(t)\overline{\chi(\overline t)}$ is defined over $\bR$ and its restriction to $\bfT_a$ is trivial. In particular, for $\sigma\in\Sigma$, $\tchi(\sigma)=\tchi(\sigma_s)$.  Notice that $\chi|_{\bfT_s}$ is defined over $\bR$, and hence $\tchi|_{\bfT_s}=(\chi|_{\bfT_s})^2$.

Denote by $\Lambda\subset X^*(\bfT)$ the sets of weights of the $\bQ$-representation $D\rho_1$ of $\bfT$, which is invariant under the canonical Galois action. For $\gamma\in S$, by \cite{MR0358865}, $D\rho(\gamma)$ is a hyperbolic matrix, and hence so is $D\rho_1(\gamma)$. It follows, because $\gamma\in\bfT(\bR)$, that $\tlambda(\gamma)=|\lambda(\gamma)|^2\neq 1$ for $\lambda\in\Lambda$. Thus $\tlambda((\gamma^2)_s)=\tlambda(\gamma^2)=\tlambda(\gamma)^2\neq 1$.

In particular, $\tlambda|_{\bfT_s}$ is non-trivial. The set $$\tLambda_s:=\{\tlambda|_{\bfT_s}:\lambda\in \Lambda\}=\{(\lambda|_{\bfT_s})^2:\lambda\in \Lambda\}$$ is invariant under the Galois action, and therefore $W(\bfH,\bfT_s)$-invariant. One can alway decompose $\bfH$ into a direct product $\prod\bfH_i$ of almost $\bR$-simple factors in such a way that $\bfT_s=\prod{\bfT_{s,i}}$ where ${\bfT_{s,i}}$ is a maximal $\bR$-split torus in $\bfH_i$. Fix a $\lambda\in\Lambda$, its restriction $\tlambda|_{{\bfT_{s,i}}}$ to some ${\bfT_{s,i}}$ is non-trivial. The action by $W(\bfH_i,{\bfT_{s,i}})\subset W(\bfH,\bfT_s)$ on $X^*(\bfT_{s,i})$ preserves no proper rank-one subgroup. Hence as $\dim{\bfT_{s,i}}\geq 2$ by assumption, the $W(\bfH,{\bfT_{s,i}})$-orbit of $\tlambda|_{{\bfT_{s,i}}}$ is not contained in any cyclic subgroup. Thus, the same is true for the $W(\bfH,\bfT_s)$-orbit of $\tlambda|_{\bfT_s}$, and for the $W(\bfH,\bfT)$-orbit of $\tlambda$. In other words, there is no cyclic subgroup of $X^*(\bfT_s)$ that contains $\tLambda_s$, which is the same as that there is no one-dimensional subgroup of $\bR^r$ that contains $\cL(\tLambda_s)$.

However, on the other hand, as $\rho_1(\Sigma)$ is assumed to be a cyclic group $\{A^n\}$, there are integers $n_1,\cdots, n_r$ such that $\rho_1(\sigma_i)=A^{n_i}$. Then for each $\lambda\in\Lambda$, there is $a_\lambda\in\bC$ such that $\lambda(\sigma_i)=a_\lambda^{n_i}$. Therefore,
$$\begin{aligned}
\cL(\tlambda|_{\bfT_s})=&\big(\log|\tlambda((\sigma_1)_s)|,\cdots,\log|\tlambda((\sigma_r)_s|\big)
=\big(\log\tlambda(\sigma_1),\cdots,\log\tlambda(\sigma_r)\big)\\
=&\big(2\log|\lambda(\sigma_1)|,\cdots,2\log|\lambda(\sigma_r)|\big)\\
=&2\log|a_\lambda|(n_1,\cdots,n_r)
\end{aligned}$$ belongs to a given one-dimensional subspace for every $\lambda\in\Lambda$.
This produces the desired contradiction and completes the proof.
\end{proof}

\subsection{Proof of Propostion \ref{AbelianSubaction}}

We deduce Proposition \ref{AbelianSubaction} from Propositions \ref{semigroup} and \ref{RegularCentralizer}.

\begin{proof}[Proof of Proposition \ref{AbelianSubaction}]
Under Hypothesis \ref{hyp7}, Theorem \ref{Arithmeticity} applies to $\bfG$ and $\Gamma$. Let $\bfH$ and $\phi:\bfH\mapsto\bfG$ be as in Theorem \ref{Arithmeticity}. Denote $H=\bfH(\bR)^\circ$.

Let $\hat\Gamma=\phi^{-1}(\Gamma)\cap\bfH(\bZ)$. Then $\hat\Gamma$ is a lattice in $H$ and is of finite index in $\bfH(\bZ)$. Define $\hat\Gamma$-actions $\hat\alpha=\alpha\circ\phi$ and $\hat\rho=\rho\circ\phi$ which act through $\Gamma$. Note $\hat\rho$ is the linear data of $\hat\alpha$.

By the discussion in Section \ref{sec:StructureNil}, there is a $\bQ$-structure of $\lien$ such that $D\rho$ sends $\Gamma$ into $\GL(d,\bQ)$. Hence the image of $\hat\Gamma$ under $D\hat\rho$ is in $\GL(d,\bQ)$ as well. Applying Theorem \ref{Superrigidity} with $k=\bQ$ and $\bfJ=\GL(d)$, we know that, after restricting to a finite index subgroup $\hat\Gamma'\subset\hat\Gamma$, $D\hat\rho$ extends to a representation $\bfH\to\mathbf{GL}(d)$ defined over $\bQ$. We replace $\hat\Gamma$ with $\hat\Gamma'$ in the sequel.

We apply Proposition \ref{semigroup} to obtain a Zariski open set $W\subset G$ and a semigroup $S\subset \Gamma$. Let $\hat W=\phi^{-1}(W)$, which is a Zariski open set in $H$. Note that because no $\bQ$-simple factor of $\bfH$ is $\bR$-anisotropic, $\bfH(\bZ)$, and hence $\hat\Gamma$ as well, are Zariski dense in $\bfH$ by Borel density theorem (see \cite[Corollary 4.5.6]{MR3307755}).

In addition, define $\hat S=\phi^{-1}(S)\cap\hat\Gamma$, and fix a preimage $\hat\gamma_0\in\phi^{-1}(\gamma_0)$.

Note $\hat\alpha=\alpha\circ\phi$ defines a $C^\infty$ action by $\hat\Gamma$. We claim that the $(H,\hat\Gamma,\hat S,\hat W, \hat\alpha,\hat\gamma_0)$, instead of $(G, \Gamma, S, W, \alpha, \gamma_0)$, also satisfies the conclusion of Proposition \ref{semigroup}. In fact, properties (1)-(4) follows directly from the corresponding properties in Proposition \ref{semigroup}. Property (5) follows for $\hat S$ exactly as in the proof of \ref{semigroup} using  now that $\hat\Gamma$ is Zariski dense in $H$ and $\hat W$ is Zariski open.

Recall that if $\hat\alpha(\gamma)$ is Anosov then $D\hat\rho(\gamma)$ is a hyperbolic matrix. Hence $(H,\hat\Gamma,\hat S,\hat\rho)$ satisfies Hypothesis \ref{hypH} as well as the conditions in Proposition \ref{RegularCentralizer}. Let $\hat\gamma\in\hat S$ and $\hat\Sigma\subset Z_\bfH(\hat\gamma)^\circ\cap\Lambda$ be given by  Proposition \ref{RegularCentralizer}. Then $\hat\rho|_{\hat\Sigma}$ has no rank-one algebraic factor.
Since
%As remarked in the proof of Proposition \ref{RegularCentralizer},
 $\hat\gamma\in Z_\bfH(\hat\gamma)^\circ$ and  as $\hat\Sigma$ is of finite index in $Z_\bfH(\hat\gamma)^\circ\cap\Lambda$, a non-trivial power $\hat\gamma^k$ is in $\hat\Sigma$.  Then $\hat\alpha(\hat\gamma^k)$ is an Anosov diffeomorphism as $\hat\alpha(\hat\gamma)$ is.

Consider $\Sigma=\phi(\hat\Sigma)\subset\Gamma$, then as $\hat\rho|_{\hat\Sigma}$ acts through $\rho|_\Sigma$, it has no rank-one algebraic factor actions. Moreover, $\alpha(\gamma)$ is Anosov for $\gamma=\phi(\hat\gamma)\in\Sigma$.\end{proof}

%By the discussion in Section \ref{sec:SmoothReduction}, Theorem \ref{SmoothThm} follows.

\section{Cohomological obstruction to lifting actions on nilmanifolds}\label{sec:Lifting}

\def\inn{\mathrm{Inn}}
\def\Out{\mathrm{Out}}
\def\out{\Out }
\def\Aut{\mathrm{Aut}}
\def\aut{\Aut }

In this section, we justify Remark \ref{LiftingRmk}, which in turn gives Corollary \ref{AnosovLifting}.

Let $M$ be a finite CW complex.  Let $\Lambda_M= \pi_1(M)/K$ be a quotient  of the fundamental group of $M$.  Let $\td M$ be the normal cover of $M$ with deck transformation group $\Lambda_M$.   %We write the action of the deck group on the right.
Let $\Gamma$ be a finitely generated discrete group and let $\alpha \colon \Gamma \to \Homeo(M)$ be an action.
Given $\gamma\in \Gamma$  select an arbitrary lift $\td \alpha(\gamma)\colon \td M\to \td M$ of $\alpha(\gamma)\colon M\to M$.  Given $\lambda\in \Lambda_M$   select any  $x\in \td M$ and  let $\td \alpha(\gamma)_*\colon \Lambda_M \to \Lambda_M$ be such that
\begin{equation}\label{eq:kkkkkkkllllllloooooooo} \td \alpha(x\lambda) = \td \alpha(x) \td \alpha(\gamma) _*(\lambda) .\end{equation}
Note $\alpha_*(\lambda) $ is independent of the choice of $x$ by continuity.

Given a second lift $\td\alpha'(\gamma) $ there is some $\lambda'\in \Lambda_M$ so that  for all $x$,
	$$\td\alpha'(\gamma)(x) = \td \alpha(\gamma)(x)\lambda'.$$
Then \begin{equation}\label{eq:eeeerrrrrssssttttta}
\td \alpha'(x\lambda) = \td \alpha(\gamma) (x\lambda)\lambda'=
\td \alpha(\gamma) (x) \td \alpha(\gamma) _*(\lambda) \lambda'= \td \alpha'(\gamma) (x) (\lambda')\inv\td \alpha(\gamma) _*(\lambda) \lambda'.\end{equation}
It follows that $\td \alpha(\gamma)_*$ is defined up to $\inn(\Lambda_M)$.  We thus obtain a well-defined homomorphism $\alpha_\#\colon \Gamma \to \Out(\Lambda_M).$

Let $N$ be a simply connected, $m$-dimensional, nilpotent Lie group and let $\Lambda\subset N$ be a lattice.  Let $P_*\colon \Lambda_M\to \Lambda$ be a surjective homomorphism.  As the kernel of $P_*$ is normal in $\Lambda_M$, whether or not $\alpha_\#$ preserves the kernel of $P_*$ is well defined.  We assume for the remainder that $ \alpha_\#$ preserves this kernel.  Replacing $\Lambda_M$ with $\Lambda_M/(\mathrm{Ker}(P_*))$ if necessary, we may further assume    that $P_*$ is an isomorphism.  We then identify $\Lambda_M$ and $\Lambda$ for the remainder
%$P_*$ then induces and action $\rho_\#\colon \Gamma \to \Out (\Lambda)$.
and continue to write  $\td M$ for the normal cover of $M$ with deck group $\Lambda$.  Recall $P\colon M\to N/\Lambda$ is the continuous map induced by $P_*$.

 Recall we have a series of central extensions in \eqref{eq:centralextensionN} and \eqref{eq:centralextensionL}.
%\begin{equation}\label{eq:centralextensionN} N= N_1 \to N_2 \to \dots \to N_{r-1} \to N_r = \{e\}\end{equation}
%and
%\begin{equation}\label{eq:centralextensionL} \Lambda= \Lambda_1 \to \Lambda_2\to \dots\to \Lambda_{r-1} \to \Lambda_r = \{e\}.\end{equation}
Let $Z_i $ denote the kernel of $N_i\to N_{i+1}$.  Then $Z_i \cap \Lambda_i\simeq \Z^{d_i}$ is the center of $\Lambda_i$ and is the kernel of each map $\Lambda_i \to \Lambda_{i+1}  $.
As each $Z_i\cap \Lambda_i$ is the center of $\Lambda_i$, an element $\psi\in \out (\Lambda_i)$  restricts to an automorphism $\restrict{\psi}{Z_i\cap \Lambda_i} \in \aut(Z_i\cap \Lambda_i)$.
Moreover, as automorphisms fix centers, % and as each $Z_i\cap \Lambda_i$ is normal in $\Lambda_i$,
we have natural maps
$$\out(\Lambda_0) \to \out(\Lambda_1) \to \dots \to\Out(\Lambda_{r-1}) .$$
In particular, $\alpha_\#$ induces representations $\alpha_{\#, i} \colon \Gamma \to \Aut(Z_i\cap \Lambda_i)=\GL(d_i,\bZ)$.

We have the following proposition  which guarantees the action   $\alpha$ lifts to an action of $\homeo(\td M) $   given the vanishing of certain cohomological obstructions.

\begin{proposition}\label{prop:lifting2}
%Let $M= N/\Lambda$ be a compact nilmanifold of dimension $m$.  Let $\Gamma$ be a finitely generaed discrete group and let $\alpha\colon \Gamma \to  \Homeo(M)$ be an action.
 Suppose for every
representation $\alpha_{\#, i} \colon \Gamma \to \Aut(Z_i\cap \Lambda_i)$,  the cohomology group $H^2_{\alpha_{\#, i}}(\Gamma,\bR^{d_i})$ is trivial.
Then, there is a finite-index subgroup $\Gamma'\subset \Gamma$ such the restricted action $\alpha \colon \Gamma'\to \Homeo(M)$ lifts to an action $\td \alpha \colon \Gamma'\to \Homeo (\td M) $.
\end{proposition}

In particular,  whether or not the action $\alpha\colon \Gamma \to \homeo(M)$ lifts (when restricted to a finite-index subgroup) is determined only by the data of the linear representations $\alpha_{\#, i}$ associated to $\alpha$.

%
%\begin{proposition}\label{prop:lifting}
%Suppose for all $d\le m$, and every representation $\rho\colon  \Gamma \to \GL(d, \Z)$, the cohomology group $H^2_\rho(\Gamma; \R^d)$ is trivial.
%
%Then there is a finite-index subgroup $\Gamma'\subset \Gamma$ such the restricted action $\alpha \colon \Gamma'\to \Homeo(M)$ lifts to an action $\td \alpha \colon \Gamma'\to \Homeo (\td M)$.
%\end{proposition}
%
%Note that we don't actually need the finiteness of cohomology for all representations; rather we need the vanishing of certain cohomology classes for specific representations constructed from the action $\rho_\#$.  See Proposition \ref{prop:lifting2} for a stronger refined statement.

The vanishing of $H^2_{\rho}(\Gamma; \bR^d)$ has been studied in \cite{MR642850} and \cite{MR0228504}. In particular, it is known to vanish in   cases %\ref{liftrem:1} and
  \ref{liftrem:3}  of Remark \ref{LiftingRmk};   case  \ref{liftrem:2}  follows from computations in \cite{MR1866848}. Cases  \ref{liftrem:4} and  \ref{liftrem:5}  will be discussed in Section \ref{sec:LiftingMeasure}.

\subsection{Candidate liftings and  defect functional}\label{sec:candidatelifting}Recall we identify $\Lambda\subset N$ with the deck group of the cover $\td M\to M$.
For $\gamma \in \Gamma$ consider an arbitrary lift $\td \alpha(\gamma) \colon \Gamma \to \Homeo(\td M)$.  The collection of lifts $\{\td \alpha(\gamma) :\gamma\in \Gamma \}$ need not assemble into an action.
The defect of the  lifts  $\{\td \alpha(\gamma)\}$ forming a coherent action is measured by  the associated defect functional $\beta \colon \Gamma \times \Gamma \to \Lambda$ defined      by
\begin{equation}\label{eq:defect} \td \alpha(\gamma_1)(\td \alpha(\gamma_2)(x)) \beta(\gamma_1, \gamma_2)=  \td \alpha(\gamma_1\gamma_2)(x) .\end{equation}
Note that the value $\beta(\gamma_1, \gamma_2)$ is independent of the choice of $x$ by continuity.  % and discreteness of $\Lambda$.
Clearly,
\begin{claim}
$\alpha \colon \Gamma \to \homeo(M)$ lifts to an action $\td \alpha \colon \Gamma \to \homeo(N)$ if and only if the lifts $\td \alpha(\gamma)$ above can be chosen so that $\beta \equiv e$.
\end{claim}

 We consider the range of $\beta$ modulo the kernel of the map $\Lambda\to \Lambda_i$.  Write $\Delta_i$ for the kernel of this  projection.
Let $\beta_i \colon \Gamma \times \Gamma \to \Lambda_i$ denote the induced map $\beta_i(\gamma_1, \gamma_2) = \beta(\gamma_1, \gamma_2)  \bmod \Delta_i.$
Below, we will have the inductive hypothesis that $\beta_{i+1}\equiv e$.  Note that this ensures that $\beta_{i}$ takes values in  $(Z\cap \Lambda_{i})\simeq \Z^{d_{i}}$.  In particular,  if $\beta_{i+1}\equiv e$ then the representation $\alpha_{\#, i}\colon \Gamma \to \Out(\Lambda_i)$ induces a linear representation $\alpha_{\#, i}\colon \Gamma \to \GL(d_i, \Z)$ on the range of $\beta_i$.

We recall the definition of group cohomology for non-trivial representations. Let $V$ be an abelian group and let $\psi\colon \Gamma\to\Aut(V)$ be a $\Gamma$-action on $V$.   % by group automorphisms.
 Denote by $C^k(\Gamma, V) $ the space of functions from $\Gamma^{k }$ to $V$. Define a map $d_{\psi,k}\colon C^k(\Gamma,V)\to C^{k+1}(\Gamma,V)$ by
\begin{align*}
&d_{\psi,k}f(\gamma_1,\cdots,\gamma_{k+1})\\
=&\psi(\gamma_1).f(\gamma_2,\cdots,\gamma_{k+1})+\sum_{j=1}^k(-1)^{j } f(\gamma_1,\dots,\gamma_{j-1},\gamma_j\gamma_{j+1},\gamma_{j+2},\dots,\gamma_{k+1})\\
&+(-1)^{k+1}f(\gamma_1,\dots,\gamma_k).
\end{align*}
It is a standard fact that $d_{\psi,k+1}\circ d_{\psi,k}=0$ and thus
$$0\to C(\Gamma,V)\xrightarrow{d_{\psi,1}}C(\Gamma^2,V)\xrightarrow{d_{\psi,2}}C(\Gamma^3,V)
\xrightarrow{d_{\psi,3}}\cdots$$
forms a cochain complex, denoted by $X_\psi(\Gamma,V)$. The group cohomology $H_\psi^\bullet(\Gamma,V)$ consists of the homology groups of $X_\psi(\Gamma,V)$. $f\in C^k(\Gamma, V)$ is called a $k$-cocyle if $f$ is in the kernel of $d_{\psi,k}$ and a $k$-coboundary if $f$ is in the image of $d_{\psi,k-1}$.

\begin{claim}\label{claim:2cocycle} Assume $\beta_{i+1}\equiv e$.  Then
$\beta_{i}$ is 2-cocycle over the representation $ \alpha_{\#, i}$.
\end{claim}
\begin{proof}

For $\gamma_1, \gamma_2, \gamma_3\in \Gamma$ we have that
\begin{align*}
\td \alpha(\gamma_1)\circ \td \alpha(\gamma_2)\circ &\td \alpha(\gamma_3)(x)
\td \alpha(\gamma_1)_*\beta(\gamma_2, \gamma_3)
\beta(\gamma_1, \gamma_2\gamma_3)
\beta(\gamma_1 \gamma_2,\gamma_3)\inv
\beta(\gamma_1, \gamma_2)\inv\\
&= \td \alpha(\gamma_1)\left( \td \alpha(\gamma_2)\circ \td \alpha(\gamma_3)(x)
\beta(\gamma_2, \gamma_3)\right)
\beta(\gamma_1, \gamma_2\gamma_3)
\beta(\gamma_1 \gamma_2,\gamma_3)\inv
\beta(\gamma_1, \gamma_2)\inv\\
&= \td \alpha(\gamma_1)\left( \td \alpha(\gamma_2 \gamma_3)(x)
 \right)
\beta(\gamma_1, \gamma_2\gamma_3)
\beta(\gamma_1 \gamma_2,\gamma_3)\inv
\beta(\gamma_1, \gamma_2)\inv\\
&= \td \alpha(\gamma_1 \gamma_2 \gamma_3)(x)
\beta(\gamma_1 \gamma_2,\gamma_3)\inv
\beta(\gamma_1, \gamma_2)\inv\\
&= \td \alpha(\gamma_1 \gamma_2 )\circ \td \alpha(\gamma_3)(x)
\beta(\gamma_1, \gamma_2)\inv\\
&= \td \alpha(\gamma_1) \circ \td \alpha( \gamma_2 )\circ \td \alpha(\gamma_3)(x)
\end{align*}
whence
$$\td \alpha(\gamma_1)_*\beta(\gamma_2, \gamma_3)
\beta(\gamma_1, \gamma_2\gamma_3)
\beta(\gamma_1 \gamma_2,\gamma_3)\inv
\beta(\gamma_1, \gamma_2)\inv \equiv e.$$
Taken modulo $\Delta_i$ the terms commute and
\begin{align*}
d_{(\alpha_{\#,i}),2} \beta_{i}&(\gamma_1, \gamma_2,\gamma_3) \\
&=
{\alpha_{\#,i}}(\gamma_1)\beta_{i}(\gamma_2,\gamma_3)\left[\beta_i(\gamma_1\gamma_2,\gamma_3)\right]\inv\beta_{i}(\gamma_1,\gamma_2\gamma_3) \left[ \beta_{i}(\gamma_1,\gamma_2)\right]\inv\\
&=\left(\td \alpha(\gamma_1)_*\beta(\gamma_2, \gamma_3)
\beta(\gamma_1, \gamma_2\gamma_3)
\beta(\gamma_1 \gamma_2,\gamma_3)\inv
\beta(\gamma_1, \gamma_2)\inv   \right) \mod \Delta_i\\
&= e.\qedhere\end{align*}
\end{proof}

\subsection {Vanishing of defect  given vanishing of cohomology}

 Proposition \ref{prop:lifting2}   follows from    the following lemma.
\begin{lemma}\label{lem:inductivelifting}
Suppose for some $1\le i\le r$  the lifts $\{\td \alpha(\gamma):\gamma\in \Gamma\}$ are chosen so that $\beta_{i+1}\equiv e$. % for all $i+1 \le k\le r$.
  Then, under the hypotheses of Proposition \ref{prop:lifting2}, there is a finite-index subgroup $\hat \Gamma\subset \Gamma$ and  choice of lifts
$\{\hat  \alpha(\gamma):\gamma\in \hat \Gamma\}$ so that, for the new defect functional defined by $$\hat  \alpha(\gamma_1)(\hat \alpha(  \gamma_2)(x)) \hat \beta(\gamma_1, \gamma_2) = \left[ \hat \alpha(\gamma_1 \gamma_2)(x)\right ],$$
$\hat \beta_i \equiv e$.  %vanishes. %  for all
%$i \le k\le r$.
\end{lemma}
%(Note that if $\hat \beta_i$ vanishes then $\hat \beta_{k}$ vanishes for $k\ge i$.)

Clearly $\beta_r\equiv e$   for any choice of lifts $\{\td \alpha (\gamma):\gamma \in \Gamma\}$.   Proposition \ref{prop:lifting2} then follows from finite induction using Lemma \ref{lem:inductivelifting}.

In order to prove Lemma \ref{lem:inductivelifting}, we recall some elementary properties of group cohomology. First, note that if $\psi$ is a morphism into $\Aut(\bZ^d)=\GL(d,\bZ)$ then, for any abelian group $A$,  $\psi$ induces a morphism $\Gamma\to\Aut(A^d)$ which we still denote  by $\psi$. Moreover, we have the relation $$X_\psi(\Gamma,A^d)=X_\psi(\Gamma,\bZ^d)\otimes_\bZ A$$ between cochain complexes.
Any short exact sequence $0\to A\to B\to C\to 0$ of abelian groups induces a short exact sequence $$0\to X_\psi(\Gamma,A^d)\to X_\psi(\Gamma,B^d)\to X_\psi(\Gamma,C^d)\to 0$$ of cochain complexes. It follows that there is a long exact sequence
$$\cdots\to H_\psi^k(\Gamma,C^d)\to H_\psi^{k+1}(\Gamma,A^d)\to H_\psi^{k+1}(\Gamma,B^d)\to H_\psi^{k+1}(\Gamma,C^d)\to\cdots$$ between group cohomologies.
Finally, we recall that by the universal coefficients theorem, there is a short exact sequence
$$0\to H_\psi^k(\Gamma,\Z^d)\otimes_\bZ A\to H_\psi^k(\Gamma,A^d)\to\mathrm{Tor}(H_\psi^{k+1}(\Gamma,\bZ^d),A)\to 0.$$ In particular, when $A$ is a flat $\bZ$-module, or equivalently when $A$ is torsion-free, $\mathrm{Tor}(H_\psi^{k+1}(\Gamma,\bZ^d),A)$ vanishes and $H_\psi^k(\Gamma,\Z^d)\otimes_\bZ A\cong H_\psi^k(\Gamma,A^d)$.

\begin{claim}\label{claim:vanish}Under the assumption that $H^2_{\alpha_{\#, i}}(\Gamma;\bR^{d_i})=0$, after restricting to  a finite-index subgroup  $\hat \Gamma \subset \Gamma$, we have that $\beta_i$ vanishes in $ H^2_{\alpha_{\#, i}}(\hat\Gamma;\bZ^{d_i})$.\end{claim}

\begin{proof}
By universal coefficients theorem $H^2_{\alpha_{\#, i}}(\Gamma;\bZ^{d_i})\otimes_\bZ\bR=H^2_{\alpha_{\#, i}}(\Gamma;\bR^{d_i})=0$. Hence all elements in $H^2_{\alpha_{\#, i}}(\Gamma;\bZ^{d_i})$ are of torsion and again by the universal coefficients theorem $H^2_{\alpha_{\#, i}}(\Gamma;\bQ^{d_i})=H^2_{\alpha_{\#, i}}(\Gamma;\bZ^{d_i})\otimes_\bZ\bQ=0$. Moreover, as $\Gamma$ is  finitely generated, $H^2_{\alpha_{\#, i}}(\Gamma;\bZ^{d_i})$ is a finitely generated abelian group. Thus $H^2_{\alpha_{\#, i}}(\Gamma;\bZ^{d_i})$ is finite.

On the other hand, take the long exact sequence
$$\cdots\to H^1_{\alpha_{\#, i}}(\Gamma;(\bQ/\bZ)^{d_i})\to H^2_{\alpha_{\#, i}}(\Gamma;\bZ^{d_i})\to H^2_{\alpha_{\#, i}}(\Gamma;\bQ^{d_i})\to\cdots$$ We see that $\beta_i$ is the image of some $\eta\in H^1_{\alpha_{\#, i}}(\Gamma;(\bQ/\bZ)^{d_i})$. Since $\Gamma$ is finitely generated, there is a denominator $q$ such that $\eta$ can be chosen to take values in the finite abelian group $(\frac1q\bZ/\bZ)^{d_i}$. The zero set $\eta\inv(0)$ is a finite-index subgroup $\hat\Gamma$ of $\Gamma$, which establishes the claim.\end{proof}

Identifying $Z_i\cap\Lambda_i$ with $\Z^{d_i}$, from Claim \ref{claim:vanish} it follows under the hypotheses of Proposition \ref{prop:lifting2} that by restricting to a finite-index subgroup $\hat \Gamma \subset \Gamma$, we have that  $\beta_{i}$ is 2-coboundary over $\alpha_{\#, i}$.  That is, there is a function $\eta\colon \hat \Gamma \to Z_i\cap \Lambda_i$ with
$$d_{(\alpha_{\#,i}),1}\eta(\gamma_1,   \gamma_2) :=
\alpha_{\#, i} (\gamma_1) \eta(  \gamma_2) \cdot [ \eta(\gamma _1  \gamma_2)]\inv \cdot \eta(\gamma_1) = \beta_{ i} (\gamma_1,   \gamma_2)$$
for all $\gamma_1, \gamma_2\in \hat \Gamma$.

Note that $\eta$ takes vales in $\Lambda/\Delta_{i}$.  Given $\gamma\in \hat  \Gamma$,  let  $\td \eta(\gamma)\in \Lambda$ be any choice of representative.
We  use $\td \eta$ to correct the original choice of lifts $\td \alpha(\gamma)$:   given $\gamma\in \hat \Gamma$, let $\hat \alpha(\gamma) = \td \alpha (\gamma)   \td \eta(\gamma) $.
With the new family of lifts $\{\hat \alpha(\gamma):   \gamma \in \hat \Gamma\}$ define a new defect functional $\hat \beta\colon \hat  \Gamma\times \hat  \Gamma \to \Lambda$ as in the lemma and similarly define induced functionals $\hat \beta_k$.

We have
\begin{claim}
The defect $\hat \beta_i$ vanishes.
\end{claim}
\begin{proof}
By definition, we have
\begin{align*}
\td \alpha(\gamma_1)\left(\td \alpha(\gamma_2)(x) \td \eta(\gamma_2) \right)\td\eta(\gamma_1)
\hat \beta(\gamma_1,  \gamma_2)&=
	  \td \alpha(\gamma_1\gamma_2)(x)\td \eta(\gamma_1\gamma_2) .
\end{align*}
With $\beta$ the defect   of $\td \alpha$ we have
\begin{align*}
\td \alpha(\gamma_1 \gamma_2)(x)     \beta(\gamma_1,\gamma_2)\inv
\td \alpha(\gamma_1)_* (\td \eta(\gamma_2) )
\td\eta(\gamma_1)
\hat \beta(\gamma_1,  \gamma_2)&=
	  \td \alpha(\gamma_1\gamma_2)(x)\td \eta(\gamma_1\gamma_2)
\end{align*}
and
\begin{align*}
   \beta(\gamma_1,\gamma_2)\inv
\td \alpha(\gamma_1)_* (\td \eta(\gamma_2) )
\td\eta(\gamma_1)
\hat \beta(\gamma_1,  \gamma_2)&=
	 \td \eta(\gamma_1\gamma_2)
\end{align*}
Modulo $\Delta_{i}$ we have
\begin{align*}
 \hat \beta(\gamma_1,  \gamma_2)  \bmod \Delta_{i}
=&
\beta_i (\gamma_1, \gamma_2) \cdot   \bigg( \alpha_{\#, i}  (\gamma_1)(  \eta(\gamma_2)) \inv\cdot    \eta(\gamma_1\gamma_2)\cdot  \eta(\gamma_1)\inv \bigg)\\
 =& \beta_{i} (\gamma_1, \gamma_2)  \cdot  [d\eta(\gamma_1, \gamma_2)]\inv\\
=& e.\qedhere
\end{align*}\end{proof}
Lemma \ref{lem:inductivelifting} follows from the  above claims.  %, and then implies Proposition \ref{prop:lifting2} and Proposition \ref{prop:lifting} as remarked earlier.

\subsection{Vanishing of defect  in the case of an invariant measure}
\label{sec:LiftingMeasure}
Consider first an action $\alpha \colon \Gamma \to \Homeo(\T^d)$ on a torus preserving a probability measure $\mu$.
It will follow from the proof of the more general Proposition \ref{prop:generalmeasruelift} below that the action $\alpha$ lifts establishing \ref{liftrem:4} of Remark \ref{LiftingRmk}.
\begin{proposition}\label{prop:measure torus}
Suppose the action $\alpha \colon \Gamma \to \homeo(\T^d)$ preserves a Borel probability measure $\mu$.
 Then   $  \alpha $ lifts to an  action $\td \alpha \colon \hat \Gamma \to \Homeo(\R^d)$ when restricted to a finite-index subgroup $\hat \Gamma \subset \Gamma$.
\end{proposition}
In the case of actions on nilmanifolds or, more generally, actions on CW complexs admitting $\pi_1$-factors, the corresponding result is more complicated.  Recall we fix a connected finite CW-complex $M$ and an action $\alpha \colon \Gamma \to \Homeo(M)$.  We also fix a simply connected nilpotent Lie group $N$, a lattice $\Lambda\subset N$, and a normal cover $\td M$ of $M$ whose deck group is identified with $\Lambda$.

Recall the sequences of $N_i$ and $\Lambda_i$ in \eqref{eq:centralextensionN} and \eqref{eq:centralextensionL}.
For each $i$ let $\td M_i$ denote the intermediate normal cover of $M$ with deck group  $\Lambda_i$.  
The natural identification of $\Lambda_i$ with the deck groups of $\td M_i\to M$ and $N_i\to N_i/\Lambda_i$ induces a map $P_i\colon  M \to N_i/\Lambda_i$.  \new{Recall (as we identify $\Lambda_M$ and  $\Lambda$) we have distinguished lifts $\td P_i\colon \td M_i\to N_i$ with $\td P_i(x \lambda) = \td P_i(x)  \lambda $.}
Suppose for some $1\le i \le r-1$ that the action $\alpha \colon \Gamma \to \Homeo(M)$ lifts to an action $\td \alpha \colon \Gamma \to \homeo(\td M_{i+1})$.
We then obtain an action $\rho_{i+1}\colon \Gamma \to \aut(\Lambda_{i+1})$ which uniquely extends to an action $\rho_{i+1}\colon \Gamma \to \aut(N_{i+1})$; in particular,  $\rho_{i+1}\colon \Gamma \to \aut(N_{i+1}/\Lambda_{i+1})$ defines a $\pi_1$ factor of $\alpha$.

Below is the general  proposition guaranteeing the lifting of an action given an invariant measure.  Note that we use the existence of a semiconjugacy to guarantee the lifting.

\begin{proposition}\label{prop:generalmeasruelift} With the above setup,
suppose there is a continuous $h\colon M\to N_{i+1}/\Lambda_{i+1}$ homotopic to $P_{i+1}$ which lifts to %a $\Lambda_{i+1}$-equivariant 
a map $\td h \colon \td M_{i+1}\to N_{i+1}$ and which intertwines the actions $\td \alpha_{i+1}\colon \Gamma \to \homeo(\td M_{i+1})$ and $\rho_{i+1}\colon \Gamma\to \aut (N_{i+1})$ \new{and is $\Lambda_{i+1}$-equivariantly homotopic to $P_{i+1}$.  }

Then, if the action $\alpha \colon \Gamma \to \homeo(M)$ preserves a Borel probability measure $\mu$, the action $\alpha$ lifts to an action $\td \alpha_i\colon \hat  \Gamma\to  \homeo(\td M_i)$ when restricted to a  finite-index subgroup $\hat \Gamma \subset \Gamma$.
\end{proposition}
Proposition \ref{prop:measure torus} follows from Proposition \ref{prop:generalmeasruelift} with $h\colon \T^d\to \{e\}$. For measure preserving actions $  \alpha \colon \Gamma\to \homeo(N/\Lambda)$ on nilmanifolds, we automatically obtain the lifting of $ \alpha $ to  $ \td  \alpha_{r-1} \colon \hat  \Gamma\to \homeo(\td M_{r-1})$.  We can then define the $\pi_1$-factor $\rho_{r-1}\colon \hat \Gamma \to \aut(N_{r-1}/\Lambda_{r-1})$.  If $\rho_{r-1}$ satisfies the hypotheses of Theorem \ref{main:factorsfull}, the semiconjugacy satisfying the hypotheses of Proposition \ref{prop:generalmeasruelift} exists and we may  lift $ \alpha $ to  $\td  \alpha_{r-2} \colon \bar   \Gamma\to \homeo(\td M_{r-2})$.  We    then recursively verify whether or not the induced  $\pi_1$-factors $\rho_i\colon \hat \Gamma \to \aut(N_{i}/\Lambda_{i})$ satisfy the hypotheses of Theorem \ref{main:factorsfull} in order to continue to lift the action.  Under   Hypotheses \ref{HigherRank}, if $\rho(\gamma_0)$ is hyperbolic for some $\gamma_0\in \Gamma$ then the same  arguments as in Section \ref{sec:6} show that the representations $\rho_i$ satisfy the hypotheses of Theorem \ref{main:factorsfull} (after restricting to finite index subgroups and extending from $G$ to $L$ as in Section \ref{sec:verify1}.)  It follows that at each step an $h$ satisfying the hypotheses of Proposition \ref{prop:generalmeasruelift} can be found.
This establishes \ref{liftrem:5} of Remark \ref{LiftingRmk}.

\begin{proof}[Proof of Proposition \ref{prop:generalmeasruelift}]
Recall that we assume $\alpha \colon \Gamma \to \homeo(M)$ lifts to $\td \alpha_{i+1}\colon \Gamma \to \homeo (\td M_{i+1})$.  For every $\gamma\in \Gamma$ choose an arbitrary lift $\td \alpha_i(\gamma)\in \homeo (\td M_i)$ of $\td \alpha_{i+1}(\gamma)$.
For $\gamma\in \Gamma$ let  $\td \alpha_{i}(\gamma)_{*}\colon \Lambda_i\to \Lambda_i$ be defined as in \eqref{eq:kkkkkkkllllllloooooooo}.
Note that given a second lift $\td\alpha'_i(\gamma) $, we have  $\td \alpha_{i}'(\gamma) = \td \alpha_{i}(\gamma) \lambda'$ for a central $\lambda'\in Z_i\cap \Lambda_i$.
In particular, from \eqref{eq:eeeerrrrrssssttttta} we have for $\lambda\in \Lambda_i$ that
	$$\td \alpha'_{i}(\gamma)_*(\lambda)=( \lambda')\inv  \td \alpha_{i}(\gamma)_*(\lambda) \lambda'=   \td \alpha_{i}(\gamma)_*(\lambda).$$
Thus any choice of lifts $\{\td \alpha_i(\gamma):\gamma\in \Gamma\}$ of the action $\td \alpha_{i+1}$ induces a representation $\alpha_{i,*}\colon \Gamma \to \Aut (\Lambda_i)$.  This in turn induces a representation $\rho_i\colon \Gamma \to \Aut (N_i)$ which in turn induces a $\pi_1$-factor of $\td \alpha$ on $N_i/\Lambda_i$.

Fix   an arbitrary family of lifts $\{\td \alpha_{i} (\gamma):\gamma\in \Gamma\}$ of the action $\td \alpha_{i+1}$.
We define the defect functional $\beta_i(\gamma_1, \gamma_2)$ as in \eqref{eq:defect}.   As we assume
$ \alpha_{i+1} $ lifts, we have that $\beta_i$ has range $Z_i\cap \Lambda_i$.  % where
% $Z_i$ is  the kernel of $N_i\to N_{i+1}$.

Recall we have $\td h\colon  M  \to N_{i+1}/\Lambda_{i+1}$ homotopic to $P_{i+1}$ and lifting to a map $\td h\colon \td M_{i+1} \to N_{i+1}$ which  intertwines the actions of $\td \alpha_{i+1}$ and $\rho_{i+1}$.
Let $H\colon M \to N_{i}/\Lambda_i$ be any continuous map, homotopic to $P_i$ and lifting $h$.  As discussed in Section \ref{sec:lifting} we may find a lift  \new{$\td H\colon  \td  M_i\to N_i$ of $H$ which is $\Lambda_i$-equivariantly homotopic to  $\td P_i$ and also lifts $\td h\colon \td M_{i+1} \to N_{i+1}$.}

Given $\gamma\in \Gamma $ and $x\in \td M_{i+1}$ let $$\td \omega_\gamma(x) = H(\td \alpha _i (\gamma)(x)) \inv \rho_i(\gamma) (H(x)).$$
%Recall we write $Z_i$ for the kernel of $N_i\to N_{i+1}$.
Using that $\td h$ interwines the actions of $\td \alpha_{i+1}$ and $\rho_{i+1}$ and that \new{$\td H(x\lambda) = \td H(x)  \lambda $ for $\lambda\in \Lambda_i$,} we verify for every $\gamma$ that
\begin{enumerate}
	\item $\td\omega_\gamma$ is $\Lambda_i$-invariant, and
	\item $\td \omega_\gamma(x) \in Z_i$ for every $x$.
\end{enumerate}
It follows that $\td \omega_\gamma$ induces a function $  \omega_\gamma\colon M\to Z_i.$ %\simeq \R^{d_i}$.
Recall $\mu$ is the invariant measure for the action $\alpha$ on $M$. Identifying $Z_i\simeq \R^{d_i}$, define $\eta\colon \Gamma \to Z_i$ by $$\eta(\gamma) = \int _M \omega _\gamma \ d \mu.$$
%Identifying $Z_i\simeq \R^{d_i}$ and
Viewing $\restrict{\rho_i}{Z^i}\in \Aut(Z_i)\simeq \GL(\R^{d_i})$ we claim
\begin{claim}
$d_{\rho_i,1}\eta = \beta_i$.
\end{claim}
\begin{proof}
We have for any $x\in \td M_i$ that
$$\td \alpha_i(\gamma_1 \gamma_2)(x) = \td \alpha_i (\gamma_1)(\td \alpha_i(\gamma_2)(x))\beta_i(\gamma_1,\gamma_2).$$
Applying the map $H$ to both sides we have
\begin{align*}
\rho_i(\gamma_1\gamma_2)(H(x))&\cdot \td \omega_{\gamma_1 \gamma_2}(x)\inv
=H(\td \alpha_i(\gamma_1 \gamma_2)(x)) \\
& = H\left(\td \alpha_i (\gamma_1)(\td \alpha_i(\gamma_2)(x))\right) \cdot \beta_i(\gamma_1,\gamma_2)\\
& = \rho_i(\gamma_1)(H(\td \alpha_i (\gamma_2)(x)))\cdot \td \omega_{\gamma_1}(\td \alpha_i( \gamma_2)(x)) \inv
\cdot  \beta_i(\gamma_1,\gamma_2)\\
& = \rho_i(\gamma_1)(\rho_i (\gamma_2)(H(x)))\cdot
\rho_i(\gamma_1)(\td \omega_{\gamma_2}(x))\inv \cdot
\td \omega_{\gamma_1}(\td \alpha_i( \gamma_2)(x)) \inv \cdot
 \beta_i(\gamma_1,\gamma_2)\\
& = \rho_i(\gamma_1\gamma_2) (H(x)))\cdot
\rho_i(\gamma_1)(\td \omega_{\gamma_2}(x))\inv \cdot
\td \omega_{\gamma_1}(\td \alpha_i( \gamma_2)(x)) \inv \cdot
 \beta_i(\gamma_1,\gamma_2).
\end{align*}
It follows that for any $x\in  \td M_i$
\begin{align*}
 \beta_i(\gamma_1,\gamma_2)
 & = \td \omega_{\gamma_1 \gamma_2}(x)\inv \cdot \rho_i(\gamma_1)(\td \omega_{\gamma_2}(x))  \cdot
\td \omega_{\gamma_1}(\td \alpha_i( \gamma_2)(x)).
\end{align*}
Using that the measure $\mu$ is $ \alpha ( \gamma_2)$-invariant,
it follows for any $x\in  \td M$   that
\begin{align*}
 \beta_i(\gamma_1,\gamma_2)
 & =   \omega_{\gamma_1 \gamma_2}(x)\inv \cdot \rho_i(\gamma_1)(  \omega_{\gamma_2}(x))
 \cdot \omega_{\gamma_1}(  \alpha ( \gamma_2)(x))\\
  &= \rho_i(\gamma_1)(\eta(\gamma_2) )\cdot  \eta(\gamma_1 \gamma_2)\inv \cdot\eta(\gamma_1)
 \\
 &= d_{\rho_i,1}\eta(\gamma_1, \gamma_2). \qedhere
\end{align*}\end{proof}
It follows that $\beta_i$ vanishes as seen as an element of $H^2_{\rho_i}(\Gamma; \R^{d_i})$.  As discussed above, by passing to a finite-index subgroup $\hat \Gamma \subset \Gamma$, it follows that $\beta_i$ vanishes as an element of $H^2_{\rho_i}(\Gamma; \Z^{d_i})$ where the lattice $\Z^{d_i}$ is identified with $Z_i\cap \Lambda_i$.
Then, as in the proof of Lemma \ref{lem:inductivelifting}, we may correct the chosen of lifts $\td \alpha_i(\gamma)$ for $\gamma\in \hat \Gamma$ into a coherent action that lifts  the action $  \alpha $.  The proposition follows.
\end{proof}

%\bibliographystyle{AWBmath}
%
%\bibliography{bibliography}
%

\end{document}